%% file: minimax-density.tex
\documentclass[12pt]{article} %
\pdfoutput=1
\usepackage[T1]{fontenc}

\usepackage{lmodern} \normalfont %
\usepackage{anyfontsize}
\DeclareFontShape{T1}{lmr}{bx}{sc} { <-> ssub * cmr/bx/sc }{}
\DeclareFontShape{T1}{lmr}{m}{scit}{ <-> ssub * cmr/m/sc }{}
\DeclareFontShape{T1}{lmr}{bx}{scit}{ <-> ssub * cmr/bx/sc }{}

\usepackage{comment}
\usepackage{etoolbox}
\newtoggle{preprint}
\newcommand{\preprint}[1]{\iftoggle{preprint}{#1}{}}
\newcommand{\aos}[1]{\iftoggle{preprint}{}{#1}}

\toggletrue{preprint} %

\newtoggle{editing}
\toggletrue{editing}

\newcommand{\pref}[1]{\cref{#1}}

\newcommand{\edited}[1]{{#1}}

\preprint{
\include{header-files/global-macros}

\include{header-files/local-macros}

\newcommand\blfootnote[1]{%
  \begingroup
  \renewcommand\thefootnote{}\footnote{#1}%
  \addtocounter{footnote}{-1}%
  \endgroup
}
\title{Minimax Rates for Conditional Density Estimation\\ via Empirical Entropy}
\author{Blair Bilodeau$^{1}$, Dylan J. Foster$^{2}$, and Daniel M.~Roy$^1$}
\date{}

\setlength{\parindent}{0pt}
\setlength{\parskip}{6pt}
\usepackage{titlesec}
\usepackage{tocloft}

\newtheorem{examplex}{Example}
\renewenvironment{example}
  {\pushQED{\qed}\examplex}
  {\popQED\endexamplex}

\newcommand{\addperiod}[1]{#1.}
\titleformat{\paragraph}[runin]
  {\normalfont\normalsize\bfseries}
  {\theparagraph}
  {1em}
  {\addperiod}
\titlespacing{\section}{0pt}{\parskip}{0pt}
\titlespacing{\subsection}{0pt}{\parskip}{0pt}
\titlespacing{\subsubsection}{0pt}{\parskip}{0pt}
\titlespacing{\paragraph}{0pt}{0pt}{0.5em}
\usepackage[letterpaper, margin=1.1in]{geometry}
}
\usepackage{xspace}

\aos{
\RequirePackage{amsthm,amsmath,amsfonts,amssymb}
\RequirePackage{graphicx}

\include{header-files/global-macros}
\include{header-files/local-macros}
}

\input{dylan}
\begin{document}

\preprint{
\maketitle
\blfootnote{${}^1$University of Toronto and Vector Institute}
\blfootnote{${}^2$Microsoft Research, New England}
\blfootnote{\, Correspondence: blair.bilodeau[at]mail.utoronto.ca}
}
\aos{\input{ims-frontmatter}}

\preprint{
\vspace{-30pt} %
\begin{abstract}

\input{abstract}

\end{abstract}
}

\preprint{\newpage}
\preprint{\tableofcontents}
\input{section-files/introduction}
\input{section-files/notation}

\input{section-files/results-summary}
\section{Proofs for main theorems}
\label{sec:main}

\input{section-files/local-complexity}

\input{section-files/risk-upper}
\input{section-files/risk-lower}
\input{section-files/mle}
\input{section-files/literature}

\input{section-files/conclusion}

\preprint{

\input{acknowledgements}
}
\aos{
\input{acknowledgements}
}

\aos{
\bibliographystyle{imsart-nameyear}
\bibliography{bib-files/minimax-density}
}

\aos{
\newpage
\begin{supplement}
\stitle{Supplement to ``Minimax Rates for Conditional Density Estimation via Empirical Entropy''}
\slink[doi]{COMPLETED BY THE TYPESETTER}
\sdatatype{.pdf}
\sdescription{Additional proofs and derivations.}
\end{supplement}
}

\preprint{
\bibliographystyle{abbrvnat}
\bibliography{bib-files/minimax-density}
}

\appendix

\input{section-files/rates-proofs}
\input{section-files/examples-proofs}

\input{section-files/mle-proofs}
\input{section-files/regret}
\input{section-files/adaptive}

\end{document}

%% file: header-files/global-macros.tex
\usepackage{
	amsmath,
	amsthm,
	amssymb,
	wrapfig,
	cases,
	mathtools,
	thmtools,
	array,
	bbm,
	mathrsfs,
	breqn,
	booktabs,
	upgreek,
	changepage,
	multirow,
	xargs,
	xstring,
	multicol,
	graphicx,
	float,
	color,
	algpseudocode,
	enumerate,
	indentfirst,
	accents,
	ifthen,
	wasysym,
	tikz,
	pgf,
	lmodern,
	}

\usepackage[utf8]{inputenc}
\usepackage[ruled,vlined]{algorithm2e}
\usepackage[hyperfootnotes=false, hidelinks]{hyperref}

\usepackage{natbib}
\setcitestyle{numbers,square,sort,comma}

\usetikzlibrary{shapes,arrows,trees,automata,positioning}

\usepackage[capitalize,nameinlink]{cleveref}
\usepackage{crossreftools}
\pdfstringdefDisableCommands{%
    \let\Cref\crtCref
    \let\cref\crtcref
}

\hypersetup{%
    bookmarksnumbered, bookmarksopen=true, bookmarksopenlevel=1,%
}

\crefname{subsection}{Subsection}{Subsections}
\crefname{lemma}{Lemma}{Lemmas}
\crefname{corollary}{Corollary}{Corollaries}
\crefname{theorem}{Theorem}{Theorems}
\crefname{claim}{Claim}{Claims}
\crefname{examplex}{Example}{Examples}

\declaretheorem[name=Theorem]{theorem}

\declaretheorem[name=Lemma]{lemma}

\declaretheorem[name=Definition]{definition}
\declaretheorem[name=Assumption]{assumption}
\declaretheorem[name=Assumption, numbered=no]{assumption*}
\declaretheorem[qed=$\triangleleft$,name=Example]{example}
\declaretheorem[qed=$\triangleleft$,name=Remark]{remark}
\declaretheorem[qed=$\triangleleft$,name=Fact]{fact}

\makeatletter
\renewcommand{\maketag@@@}[1]{\hbox{\m@th\normalsize\normalfont#1}}%
\makeatother

\makeatletter
\let\reftagform@=\tagform@
\def\tagform@#1{\maketag@@@{\ignorespaces\textcolor{gray}{(#1)}\unskip\@@italiccorr}}
\renewcommand{\eqref}[1]{\textup{\reftagform@{\ref{#1}}}}
\makeatother

\newcommand{\EE}{\mathbb{E}}

\newcommand{\II}{\mathbb{I}}

\newcommand{\NN}{\mathbb{N}}

\newcommand{\PP}{\mathbb{P}}

\newcommand{\RR}{\mathbb{R}}

\newcommand{\ZZ}{\mathbb{Z}}

\newcommand{\Aa}{\mathcal{A}}
\newcommand{\Bb}{\mathcal{B}}
\newcommand{\Cc}{\mathcal{C}}

\newcommand{\Ff}{\mathcal{F}}
\newcommand{\Gg}{\mathcal{G}}
\newcommand{\Hh}{\mathcal{H}}

\newcommand{\Jj}{\mathcal{J}}
\newcommand{\Kk}{\mathcal{K}}
\newcommand{\Ll}{\mathcal{L}}
\newcommand{\Mm}{\mathcal{M}}
\newcommand{\Nn}{\mathcal{N}}

\newcommand{\Pp}{\mathcal{P}}

\newcommand{\Rr}{\mathcal{R}}

\newcommand{\Xx}{\mathcal{X}}
\newcommand{\Yy}{\mathcal{Y}}

\newcommand{\eps}{\varepsilon}

\def\[#1\]{\begin{equation}\begin{aligned}#1\end{aligned}\end{equation}}
\def\*[#1\]{\begin{equation*}\begin{aligned}#1\end{aligned}\end{equation*}}

\def\s*[#1\s]{\small\begin{align*}#1\end{align*}\normalsize}

\newcommand{\lcrx}[4][{-1}]{ 
	\IfEq{#1}{-1}{\left #2 {{{{#3}}}} \right #4}{
   	\IfEq{#1}{0}{#2 {{{{#3}}}} #4}{
	\IfEq{#1}{1}{\bigl #2 {{{{#3}}}} \bigr #4}{
	\IfEq{#1}{2}{\Bigl #2 {{{{#3}}}} \Bigr #4}{
	\IfEq{#1}{3}{\biggl #2 {{{{#3}}}} \biggr #4}{
	\IfEq{#1}{4}{\Biggl #2 {{{{#3}}}} \Biggr #4}{
    \GenericWarning{"4th argument to lcrx must be -1, 0, 1, 2, 3, or 4"}
    }}}}}}} %

\newcommand{\KL}[2]{\mathrm{KL}\rbra[0]{#1 \ \Vert \ #2}}

\newcommand{\negsetdelim}{\vert}
\newcommand{\setdelim}{\ \vert \ }
\newcommand{\Bigsetdelim}{\ \Big\vert \ }

\newcommand{\iid}{\text{i.i.d.}\xspace}

\newcommand{\as}{\text{a.s.}\xspace}

\newcommand{\stT}{\ \text{s.t.}\ }

\newcommand{\ind}[1]{\II{\left\{#1\right\}}} %

\def\multiset#1#2{\ensuremath{\left(\kern-.3em\left(\genfrac{}{}{0pt}{}{#1}{#2}\right)\kern-.3em\right)}}

\DeclareMathOperator*{\argmin}{\arg\min} %
\DeclareMathOperator*{\argmax}{\arg\max} %
\DeclareMathOperator*{\newlim}{\mathrm{lim}\vphantom{\mathrm{infsup}}}
\DeclareMathOperator*{\newmin}{\mathrm{min}\vphantom{\mathrm{infsup}}}
\DeclareMathOperator*{\newmax}{\mathrm{max}\vphantom{\mathrm{infsup}}}
\DeclareMathOperator*{\newinf}{\mathrm{inf}\vphantom{\mathrm{infsup}}}
\DeclareMathOperator*{\newsup}{\mathrm{sup}\vphantom{\mathrm{infsup}}}
\renewcommand{\lim}{\newlim}
\renewcommand{\min}{\newmin}
\renewcommand{\max}{\newmax}
\renewcommand{\inf}{\newinf}
\renewcommand{\sup}{\newsup}

\newcommand{\dee}{\mathrm{d}} %
\newcommand{\poissondist}{\mathrm{Pois}}

\newcommand{\bernoullidist}{\mathrm{Ber}}

\newcommand{\gammadist}{\mathrm{Gamma}}
\newcommand{\normaldist}{\mathrm{Gaussian}}

\newcommand{\uniformdist}{\mathrm{Unif}}

\newcommand{\rbra}[2][{-1}]{\lcrx[#1] ( {#2} ) }

\newcommand{\sbra}[2][{-1}]{\lcrx[#1] [ {#2} ] }

\newcommand{\abs}[2][{-1}]{\lcrx[#1] \vert {#2} \vert }

\newcommand{\norm}[2][{-1}]{\lcrx[#1] \Vert {#2} \Vert}
\newcommand{\inner}[3][{-1}]{\lcrx[#1] \langle {{#2},\ {#3}} \rangle}

\newcommand{\vcdim}{\mathrm{VCdim}}
\newcommand{\vcind}{\mathrm{Pdim}}

\newcommand{\compsym}{\text{\upshape c}}

\newcommand{\Nats}{\NN}

\newcommand{\Ints}{\ZZ}

\newcommand{\Reals}{\RR}

\newcommand{\PosInts}{\Ints_+}

\newcommand{\range}[2][{1}]{
	\IfEq{#1}{1}{\sbra{#2}}{\sbra{#2}_{#1}}}
\newcommand{\rangeO}[2][{0}]{
	\IfEq{#1}{0}{\sbra{#2}_0}{\sbra{#2}_{#1}}}

%% file: header-files/local-macros.tex
\newcommand{\rhs}{right-hand side\xspace}

\newcommand{\uppfuncname}{localization function}
\newcommand{\uppradname}{localization radius}

\newcommand{\erf}{\mathrm{erf}}

\newcommand{\hilbertspace}{\mathbb{H}}

\newcommand{\arbspace}{\Cc}
\newcommand{\arbvec}{c}

\newcommand{\powerset}{\Pp}
\newcommand{\contextspace}{\Xx}
\newcommand{\contextsigma}{\Aa}
\newcommand{\contextMspace}{(\contextspace, \contextsigma)}
\newcommand{\dataspace}{\Yy}
\newcommand{\datasigma}{\Bb}
\newcommand{\dataMspace}{(\dataspace, \datasigma)}
\newcommand{\historyspace}[1]{(\contextspace \times \dataspace)^{#1}}

\newcommand{\multidim}{k}
\newcommand{\holderfunc}{\psi}
\newcommand{\numderivs}{m}
\newcommand{\derivvec}{\alpha}
\newcommand{\derivvecc}{\alpha'}
\newcommand{\nummderivs}{m'}

\newcommand{\context}{x}
\newcommand{\packcontext}{\overline{x}}
\newcommand{\dumcontext}{x'}
\newcommand{\data}{y}
\newcommand{\dumdata}{y'}

\newcommand{\sample}[1]{S_{#1}}

\newcommand{\densityset}{\Mm}
\newcommand{\kernelset}{\Kk}

\newcommand{\contextdistn}{\mu_{\Xx}}
\newcommand{\contextdensity}{p_{\contextdistn}}

\newcommand{\datareldistn}{\nu}
\newcommand{\jointreldistn}{\lambda}

\newcommand{\datasize}{K}
\newcommand{\densbound}{B}

\newcommand{\functionset}{\Ff}
\newcommand{\dumfunctionset}{\Jj}
\newcommand{\vcclass}{\Cc}

\newcommand{\vcval}{v}

\newcommand{\jointfunctionset}{\Pi}

\newcommand{\regfunc}{\theta}
\newcommand{\func}{f}
\newcommand{\coverfunc}{g}

\newcommand{\coverfuncdum}{\coverfunc'}
\newcommand{\coverfuncdumm}{h}
\newcommand{\hellfunc}{h}
\newcommand{\truefunc}{\func^*}

\newcommand{\predfunc}[1]{{\hat f}_{\! #1}}
\newcommand{\avgpredfunc}[1]{{\overline f}_{\! #1}}
\newcommand{\riskpredfunc}[1]{{\hat f}^{\text{\tiny \upshape agg}}_{\! #1}}
\newcommand{\riskpredfuncmod}[2]{{\hat f}^{\text{\tiny \upshape agg}, (#2)}_{\! #1}}
\newcommand{\riskpredfuncadapt}[1]{{\hat f}^{\text{\tiny \upshape ada}}_{\! #1}}
\newcommand{\dumfunc}{h}
\newcommand{\vcfunc}{c}

\newcommand{\vcsmooth}{\theta}

\newcommand{\jointfunc}{\pi}
\newcommand{\jointcoverfunc}{\pi'}
\newcommand{\truejointfunc}{\pi^*}

\newcommand{\covfunctionset}{\mathscr{D}}

\newcommand{\gendens}{\pi}

\newcommand{\locsym}{\mathrm{loc}}

\newcommand{\entropyconst}{C_{\functionset,n}}

\newcommand{\covsym}{\Nn}
\newcommand{\packsym}{\overline{\Nn}}
\newcommand{\onecov}[1]{\covsym_{#1}}
\newcommand{\gencov}[2]{\covsym_{#1, #2}}
\newcommand{\brackcov}[2]{\covsym^{[\, ]}_{#1, #2}}
\newcommand{\genpack}[2]{\packsym_{\! #1, #2}}
\newcommand{\locgenpack}[3]{\packsym^{#3}_{\! #1, #2}}
\newcommand{\genjointcov}[1]{\covsym_{#1}}

\newcommand{\entsym}{\Hh}
\newcommand{\packentsym}{\overline{\Hh}}

\newcommand{\genent}[2]{\entsym_{#1, #2}}
\newcommand{\brackent}[2]{\entsym^{[\, ]}_{#1, #2}}
\newcommand{\genpackent}[2]{\packentsym_{#1, #2}}
\newcommand{\locgenpackent}[2]{\packentsym^{\locsym}_{#1, #2}}
\newcommand{\genjointent}[1]{\entsym_{#1}}

\newcommand{\reffunc}{f}

\newcommand{\distsym}{d}
\newcommand{\klsym}{\mathrm{KL}}
\newcommand{\tvsym}{\mathrm{TV}}
\newcommand{\hellsym}{\mathrm{H}}
\newcommand{\mutualinfo}{I}

\newcommand{\entropy}{H}

\newcommand{\funcnorm}[1]{L^{#1}}
\newcommand{\vecnorm}[1]{\ell_{#1}}

\newcommand{\coverset}{\Gg}
\newcommand{\packingset}{\overline{\Gg}}
\newcommand{\locpackingset}{\overline{\Gg}^{\locsym}}
\newcommand{\hellingerset}{\functionset^{\hellsym}}

\newcommand{\jointcoverset}{\Pi'}

\newcommand{\datadim}{d}
\newcommand{\funcdim}{p}
\newcommand{\holderparam}{\gamma}

\newcommand{\ratepow}{\beta}

\newcommand{\rad}{\mathfrak{R}}

\newcommand{\locradius}[1]{r_{\! #1}^*}
\newcommand{\uppfunc}{\phi}

\newcommand{\paramconstA}{A}

\newcommand{\smoothloss}{\phi}

\newcommand{\numepochs}{M}
\newcommand{\epochidx}{m}
\newcommand{\epochrd}[1]{n_{#1}}
\newcommand{\epochrds}[1]{\tau_{#1}}

\newcommand{\smoothparam}{\alpha}
\newcommand{\smooth}[2]{\widetilde{#2}^{(#1)}}
\newcommand{\smoothspace}{\Bb}

\newcommand{\prior}{\pi_{\text{\upshape \tiny mod}}}

\newcommand{\risksym}{\Rr}
\newcommand{\losssym}{\Ll}

\newcommand{\loss}[2]{\losssym\left(#1; #2 \right)}
\newcommand{\predrisk}[3]{\risksym_{#3\!}(#1; #2)}

\newcommand{\minimaxrisk}[2]{\risksym_{#2\!}(#1)}

\newcommand{\regretsym}{R}
\newcommand{\predregret}[3]{\regretsym_{#3\!}(#1; #2)}
\newcommand{\minimaxregret}[2]{\regretsym_{#2\!}\left(#1\right)}

\newcommand{\rate}{r}

\newcommand{\mlec}{c}

\newcommand{\mlefunc}[1]{{\hat f}_{\! #1}^{\scriptscriptstyle \mathrm{MLE}}}
\newcommand{\smoothmle}[1]{{\widetilde f}^{\scriptscriptstyle \mathrm{MLE}(#1)}}

\newcommand{\loglik}{L}

\newcommand{\truefunctionset}{\functionset^*}

\newcommand{\basefunc}{f}
\newcommand{\bumpfunc}{\omega}
\newcommand{\adjfunc}{g}
\newcommand{\fullfunc}{G}
\newcommand{\convexfullfunc}{\overline{G}}
\newcommand{\sepset}{A}
\newcommand{\sepsize}{N}
\newcommand{\sepcontext}{z}
\newcommand{\sepind}{i}

\newcommand{\mleclass}{\Gamma_{\!\!\holderparam}}
\newcommand{\convexmleclass}{\overline{\Gamma}_{\!\!\holderparam}}

\newcommand{\holdermle}[1]{\overline{\fullfunc}_{\holderparam, #1}^{\mathrm{\scriptscriptstyle MLE}}}
\newcommand{\sepsign}{\eps}

\newcommand{\specfunc}{\hat G_n}
\newcommand{\mlesubset}{S_n}
\newcommand{\convexind}{\ell}
\newcommand{\convexinddum}{\ell'}
\newcommand{\convexsize}{L}
\newcommand{\convexup}{^{(\convexind)}}
\newcommand{\convexupdum}{^{(\convexind')}}
\newcommand{\convexmleheight}{\overline{\mleheight}}
\newcommand{\convexprob}{\mu}

\newcommand{\bumpexp}{b}
\newcommand{\mlewidth}{\eta}
\newcommand{\mleheight}{\lambda}

\newcommand{\param}{\theta}
\newcommand{\paramdum}{\theta'}
\newcommand{\paramdensityset}{\Pp}
\newcommand{\paramdensity}[1]{\Pp_{#1}}
\newcommand{\parambound}{\Theta}

\newcommand{\paramdim}{p}
\newcommand{\sparsedim}{k}

\newcommand{\xbound}{X}
\newcommand{\wbound}{W}

\newcommand{\linearrad}{R}

\newcommand{\weightspace}{\mathcal{W}}
\newcommand{\weight}{w}

\newcommand{\databandwidth}[1]{{\hat h}_{#1}}

\newcommand{\dataknotset}[1]{{\hat S}_{\databandwidth{n}}}

\newcommand{\shortnormaldist}{\mathrm{Nor}}
\newcommand{\normalmean}{\theta}
\newcommand{\normalmeandum}{\theta'}

%% file: dylan.tex
\DeclarePairedDelimiter{\brk}{[}{]}
\DeclarePairedDelimiter{\crl}{\{}{\}}

\def\ddefloop#1{\ifx\ddefloop#1\else\ddef{#1}\expandafter\ddefloop\fi}
\def\ddef#1{\expandafter\def\csname bb#1\endcsname{\ensuremath{\mathbb{#1}}}}
\ddefloop ABCDEFGHIJKLMNOPQRSTUVWXYZ\ddefloop
\def\ddefloop#1{\ifx\ddefloop#1\else\ddef{#1}\expandafter\ddefloop\fi}
\def\ddef#1{\expandafter\def\csname b#1\endcsname{\ensuremath{\mathbf{#1}}}}
\ddefloop ABCDEFGHIJKLMNOPQRSTUVWXYZ\ddefloop
\def\ddef#1{\expandafter\def\csname sf#1\endcsname{\ensuremath{\mathsf{#1}}}}
\ddefloop ABCDEFGHIJKLMNOPQRSTUVWXYZ\ddefloop
\def\ddef#1{\expandafter\def\csname c#1\endcsname{\ensuremath{\mathcal{#1}}}}
\ddefloop ABCDEFGHIJKLMNOPQRSTUVWXYZ\ddefloop
\def\ddef#1{\expandafter\def\csname h#1\endcsname{\ensuremath{\widehat{#1}}}}
\ddefloop ABCDEFGHIJKLMNOPQRSTUVWXYZ\ddefloop
\def\ddef#1{\expandafter\def\csname hc#1\endcsname{\ensuremath{\widehat{\mathcal{#1}}}}}
\ddefloop ABCDEFGHIJKLMNOPQRSTUVWXYZ\ddefloop
\def\ddef#1{\expandafter\def\csname t#1\endcsname{\ensuremath{\widetilde{#1}}}}
\ddefloop ABCDEFGHIJKLMNOPQRSTUVWXYZ\ddefloop
\def\ddef#1{\expandafter\def\csname tc#1\endcsname{\ensuremath{\widetilde{\mathcal{#1}}}}}
\ddefloop ABCDEFGHIJKLMNOPQRSTUVWXYZ\ddefloop

\newcommand{\veps}{\varepsilon}

\newcommand{\rdef}{=\vcentcolon}

%% file: abstract.tex
We consider the task of estimating a conditional density using i.i.d.\ samples from a joint distribution, which is a fundamental problem with applications in both classification and uncertainty quantification for regression. For joint density estimation, minimax rates have been characterized for general density classes in terms of uniform (metric) entropy, a well-studied notion of statistical capacity. When applying these results to conditional density estimation, the use of uniform entropy---which is infinite when the covariate space is unbounded and suffers from the curse of dimensionality---can lead to suboptimal rates. Consequently, minimax rates for conditional density estimation 
cannot be characterized using these classical results.

We resolve this problem for well-specified models, obtaining matching (within logarithmic factors) upper and lower bounds on the minimax Kullback--Leibler risk in terms of
the \emph{empirical Hellinger entropy} for the conditional density class. 
The use of empirical entropy allows us to appeal to concentration arguments based on local Rademacher complexity, which---in contrast to uniform entropy---leads to matching rates for large, potentially nonparametric classes and captures the correct dependence on the complexity of the covariate space. Our results require only that the conditional densities are bounded above, and do not require that they are bounded below or otherwise satisfy any tail conditions.

%% file: section-files/introduction.tex
\section{Introduction}
\label{sec:introduction}

Estimating the density underlying an observed stochastic process is a central problem in statistics. We consider the task of \emph{conditional density estimation}, a fundamental and well-studied variant of the density estimation problem in which the observed stochastic process has a natural decomposition into \emph{covariates} (representing contextual information) and \emph{responses}; special cases include real-valued responses (e.g., regression tasks) and binary or categorical responses (e.g., classification tasks) \citep{vandegeer90regression,yang99classification,bartlett02complexity,rakhlin17minimax}. Compared to nonparametric regression, where the goal is to estimate the conditional mean, estimating the conditional density allows the user to quantify uncertainty, which can be used to reliably inform downstream decisions. For example, applications to online decision making may be found in \citet{agarwal20flambe} and \citet{foster2021efficient}. However, while the task of \emph{joint} density estimation (i.e., estimating the density for both covariates and responses jointly) has enjoyed extensive classical investigation, conditional density estimation presents unique technical challenges, and optimal estimators and rates of convergence for general conditional density classes are not well understood.

We focus on conditional density estimation under \emph{logarithmic loss}, and aim to minimize the associated \emph{Kullback--Leibler} (KL) risk
\*[
	\EE_{\context\sim\contextdistn}\KL{\truefunc(\context)}{\hat{f}(\context)},
\]
where $\contextdistn$ is the (unknown) marginal distribution over covariates, $\truefunc$ is the (unknown) true conditional density function, and $\hat{f}$ is the estimator. Classical results in (joint) density estimation and nonparametric regression \citep{lecam73convergence,birge83estimation,birge86hellinger,yang99information} have established that under fairly general conditions, the minimax rates of convergence for these tasks are determined (up to logarithmic factors) by a \emph{critical radius} $\eps_n$ satisfying the relationship
\[\label{eqn:entropy-relationship}
	\entsym(\functionset, \eps_n) \asymp n \eps_n^2,
\]
where $\entsym$ denotes the \emph{entropy} for the model class $\functionset$ under a suitable problem-specific notion of distance.
Existing density estimation results have two limitations from the perspective of the present work. 
First, they are tailored to joint density estimation, and consequently require estimation of the marginal covariate distribution. This leads to suboptimal results for conditional density estimation where---as we show---estimating the marginal distribution is not required for optimal performance. Second, joint density estimation results typically rely on a \emph{uniform} notion of entropy, which requires strong coverage over the entire covariate space; we show that for conditional density estimation, data-dependent (``empirical'')  notions of entropy are necessary to obtain tight rates, particularly when the covariate space is unbounded or high-dimensional. See \cref{fact:linear-example,fact:vc-example} and the discussion in \cref{sec:lit-uniform} for more details.

Empirical entropy is known to characterize the minimax rates for various nonparametric regression settings \citep{vandegeer90regression,bartlett02complexity,rakhlin17minimax}. However, the notion of distance (the $\funcnorm{p}$ metric) used to define entropy in these settings is tailored to estimating conditional means, and may not be suitable for the more general problem of estimating conditional densities. In fact, recent work on a sequential variant of the conditional density estimation problem \citep{bilodeau20logloss} shows that---in contrast to nonparametric regression---$\funcnorm{p}$ entropy is not sufficient to characterize minimax rates for all conditional density classes simultaneously, even when responses are binary-valued.
Consequently, determining the minimax rates for conditional density estimation with general conditional density classes has remained an open problem. The primary contribution of the present work is to provide a notion of entropy $\entsym(\functionset, \eps_n)$ such that the minimax rates for conditional density estimation are governed by the relationship in \cref{eqn:entropy-relationship} up to logarithmic factors.

\paragraph*{Contributions}
We characterize the minimax rates for conditional density estimation under Kullback--Leibler risk in terms of the \emph{empirical Hellinger entropy} for the conditional density class.
Our results establish matching (within logarithmic factors) upper and lower bounds on the KL risk for \emph{all} conditional density classes with either \emph{parametric} entropy ($\entsym(\functionset, \eps_n)$ of size $\funcdim\log(1/\eps_n)$) or \emph{nonparametric} entropy ($\entsym(\functionset, \eps_n)$ of size $\eps_n^{-\funcdim}$).
Our upper bounds are achieved using a novel estimator based on aggregation over an appropriate data-dependent discretization of $\functionset$ tailored to empirical Hellinger distance, and we complement the use of this estimator by highlighting natural settings in which the more standard maximum likelihood estimator attains suboptimal KL risk.

The key techniques that enable our results are a) the novel introduction of empirical Hellinger entropy for analyzing conditional density estimation, b) the use of concentration inequalities based on local Rademacher complexity, which is made possible by the aforementioned choice of empirical entropy, and c) a result of \citet{yang98model} that relates KL divergence to Hellinger distance whenever the \emph{log} density ratio is bounded. The latter technique is what motivates our use of empirical Hellinger distance, and allows us to avoid requiring any lower bounds on the tails for the conditional density class under consideration. All three of these aspects are crucial to obtaining tight rates for general, potentially nonparametric classes of conditional densities.

\aos{
\paragraph*{Organization}
The remainder of this work is organized as follows. \cref{sec:notation} gives the formal setup for the conditional density estimation problem.
In \cref{sec:results-summary} we state our main results; in particular, \cref{sec:minimax-estimator} describes an explicit estimator that achieves minimax rates, and \cref{sec:examples} provides six examples for which the rates from our main theorem improve upon the state of the art.
In \cref{sec:main} we state slightly more general variants of our primary theorems (upper and lower bounds for the minimax risk), and prove the main results.
In \cref{sec:mle} we discuss the performance of the conditional maximum likelihood estimator and prove that in general, it does not achieve minimax rates, and in \cref{sec:literature} we highlight relevant related work from the rich literature on density estimation and nonparametric regression. Finally, we conclude with a brief discussion in
\cref{sec:discussion}. We defer omitted proofs for the main results, as well as a straightforward extension to obtain regret bounds and an adaptive estimator, to the supplementary material.}

%% file: section-files/notation.tex
\section{Problem setting}
\label{sec:notation}

We consider a conditional density estimation setting in which the
statistician receives \iid{} samples $(\context_1,\data_1),\ldots,(\context_n,\data_n)\in\contextspace\times\dataspace$, where
$\context_t\sim\contextdistn$ is a covariate and $\data_t \negsetdelim \context_t\sim\truefunc(\context_t)$ is a response,
and where $\truefunc$ is the true underlying conditional density
function. Given such a sample, the goal of the statistician is to produce an
estimator $\predfunc{n}$ such that the average KL divergence
$\EE_{\context\sim\contextdistn}\KL{\truefunc(\context)}{\predfunc{n}(\context)}$ is
minimized. 
We next define the setting rigorously and state
our main statistical assumptions.

\subsection{Probability spaces and statistical assumptions}
Let $\contextMspace$ and $\dataMspace$ be 
measurable spaces
denoting the covariate space and response space respectively. Our results are stated in terms of a fixed reference measure $\datareldistn$ on
$\dataMspace$; 
let $\densityset(\dataspace)$ denote the set of probability densities with respect to $\datareldistn$ on $\dataMspace$. 
A \emph{(regular) conditional density} is a jointly measurable map $\func:\contextspace \times \dataspace \to \Reals$ such that for all $\context\in\contextspace$, 
$[\data \mapsto \func(\context,\data)] \in \densityset(\dataspace)$.
We denote the collection of regular conditional densities by $\kernelset(\contextspace, \dataspace)$.
The \emph{covariate distribution} $\contextdistn$ 
is an arbitrary probability measure
on $\contextMspace$, and 
observations are generated by the process $\context\sim\contextdistn$ and
$\data\negsetdelim\context \sim\truefunc(\context)$, where
$\truefunc\in\kernelset(\contextspace,\dataspace)$ is the true conditional density.

Let $\functionset\subseteq\kernelset(\contextspace,\dataspace)$ be the statistician's
\emph{conditional density class} (or, \emph{model}),
and for each $\context\in\contextspace$, $\data\in\dataspace$, and
$\func\in\functionset$, denote the conditional density of $\data$
given $\context$ by $[\func(\context)](\data)$. 
For our results, we require that the model is
\emph{well-specified} and bounded above.
\begin{assumption}
  \label{ass:well-specified}
  The true (data-generating) conditional density satisfies $\truefunc \in \functionset$.
\end{assumption}
\begin{assumption}
  \label{ass:bounded-density}
  There exist constants $\densbound,\datasize < \infty$ such that $\sup_{\func\in\functionset} \norm{\func}_\infty = \densbound$ and
  $\datareldistn(\dataspace)=\datasize$. 
  Without loss of generality we
  assume that $\datasize \geq 1/\densbound$.
\end{assumption}
While many of the quantities we consider (e.g., KL and Hellinger) are
independent of the choice of reference measure $\nu$, our quantitative
results depend (weakly) on this choice through the parameters $B$ and
$K$. Refer to \cref{sec:results-summary} for further discussion.

\subsection{Estimation and risk}
Let $\PosInts = \Nats \cup \{0\}$, and for any $n \in \Nats$ let $[n] = \{1,\dots,n\}$.
For any function $f$ defined on an arbitrary space $\arbspace$ and
vector $\arbvec_{1:n} \in \arbspace^n$, we use the vector notation
$\func(\arbvec_{1:n}) = (\func(\arbvec_1),\dots,\func(\arbvec_n))$.
By a sample of size $n$, we mean a random element
$\sample{n} = (\context_{t}, \data_{t})_{t \in [n]}$ taking values in
$\historyspace{n}$.
In this setting, an estimator (of $\truefunc$) is a sequence of regular-conditional-density-valued functions
$\predfunc{}= (\predfunc{n})_{n\in\PosInts}$ where, for each
$n$,
\*[
  \predfunc{n}: \historyspace{n} \to \kernelset(\contextspace, \dataspace).
\]
When the sample is clear from context, we %
use $\predfunc{n}$ to also denote the \emph{estimate} $\predfunc{n}(\sample{n})$.

For brevity, when $\context$ is a random element, we write $\data \sim \truefunc(\context)$ to denote $\data\negsetdelim\context \sim \truefunc(\context)$.
For any $n\geq1$ and measurable $\phi:\contextspace^n \to \Reals$, we
write 
\*[
  \EE_{\context_{1:n} \sim \contextdistn} [\phi(\context_{1:n})]
  = \int \phi(\context_{1:n}) \contextdistn^{\otimes n}(\dee \context_{1:n}),
\]
and similarly for any measurable $\psi:\dataspace^n \to \Reals$,
$\truefunc \in\kernelset(\contextspace, \dataspace)$, and
$\context_{1:n}\in\contextspace^n$, we write
\*[
\EE_{\data_{1:n} \sim \truefunc(\context_{1:n})} [\psi(\data_{1:n})]
  = \int \psi(\data_{1:n}) \left(\prod_{t=1}^n [\truefunc(\context_{t})](\data_t)\right) \datareldistn^{\otimes n}(\dee \data_{1:n}).
\]
When $n=1$ we drop the subscripts.

For any densities $p,q \in \densityset(\dataspace)$, define the
\emph{Kullback--Leibler (KL)
divergence} by
\*[
\KL{p}{q}= \int p(\data) \log\left(\frac{p(\data)}{q(\data)}\right) \datareldistn(\dee\data).
\]
It will also be convenient to define the so-called \emph{KL distance} by
\*[
  \distsym_{\klsym}(p,q)
  = \sqrt{\KL{p}{q}} \, .
\]
We define the \emph{KL loss} of a regular conditional density $\func \in\kernelset(\contextspace, \dataspace)$ as
\*[
  \loss{\func}{\truefunc, \contextdistn}
  = \EE_{\context \sim \contextdistn} \distsym_{\klsym}^2(\truefunc(\context), \func(\context)),
\]
and define the \emph{KL risk} of an estimator $\predfunc{}= (\predfunc{n})_{n\in\Nats}$ as
\*[
  \predrisk{\predfunc{}}{\truefunc, \contextdistn}{n}
  = \EE_{\context_{1:n} \sim \contextdistn} \EE_{\data_{1:n} \sim \truefunc(\context_{1:n})} \loss{\predfunc{n}(\sample{n})}{\truefunc, \contextdistn}.
\]
Finally, the \emph{minimax risk} of a conditional density class $\functionset \subseteq \kernelset(\contextspace, \dataspace)$ is 
\*[
  \minimaxrisk{\functionset}{n}
  = \inf_{\predfunc{}} \sup_{\contextdistn} \sup_{\truefunc \in \functionset} \predrisk{\predfunc{}}{\truefunc, \contextdistn}{n},
\]
where $\predfunc{}$ ranges over all estimators and $\contextdistn$
ranges over all probability measures on $\contextMspace$. Our main contribution is tight upper and lower bounds on the minimax risk.

\paragraph*{Additional notation}
For any two densities $f,g \in \densityset(\dataspace)$, we use the
following standard discrepancy measures. For $r>0$, the
\emph{$\funcnorm{r}$ distance} is
\*[
  \distsym_{r}(f,g)
  = \left(\int \abs{f(\data) - g(\data)}^r \datareldistn(\dee \data) \right)^{1/r}
\]
and the \emph{Hellinger distance} is
\*[
  \distsym_{\hellsym}(f,g)
  = \frac{1}{\sqrt{2}} \, \distsym_{2}(\sqrt{\!\smash[b]{f}\vphantom{a}},\sqrt{\!\smash[b]{g}\vphantom{a}} \,).
\]
Notice that the Hellinger distance is scaled to ensure that $\distsym_{\hellsym} \in [0,1]$.

If $\dataspace = \{0,1\}$, then for any function $\distsym:\densityset(\dataspace)\times\densityset(\dataspace)\to\Reals$ and $a,b\in[0,1]$, we overload notation to let $\distsym(a,b)$ denote $\distsym(\bernoullidist(a), \bernoullidist(b))$.

\paragraph*{Asymptotic notation}
To state results concisely, we use standard asymptotic notation.
In particular, for any two nonnegative sequences $(a_n)$ and $(b_n)$, we write $a_n \lesssim b_n$ if there exists a constant $C>0$ such that for all $n$, $a_n \leq C b_n$. 
If $a_n \lesssim b_n$ and $b_n \lesssim a_n$, we write $a_n \asymp b_n$.
All constants are independent of $n$, $\densbound$, $\datasize$, the
effective dimension (i.e., $\funcdim$ when
$\genent{\hellsym}{2}(\functionset, \eps, n) \asymp
\eps^{-\funcdim}$), and the scale for the covering/packing numbers we consider.

%% file: section-files/results-summary.tex
\section{Main results}
\label{sec:results-summary}
In order to characterize the minimax risk $\minimaxrisk{\functionset}{n}$ of conditional density estimation, we must first formally introduce the \emph{empirical Hellinger entropy}.
\subsection{Empirical metric entropy}
Fix a metric $\distsym: \densityset(\dataspace) \times \densityset(\dataspace) \to \Reals$ and normalized $\vecnorm{q}$ norm for $q \in [1,\infty)$. (The extension of the following definitions to $q=\infty$ is clear.)
For a class $\functionset \subseteq \kernelset(\contextspace, \dataspace)$, we say $\coverset \subseteq \functionset$ \emph{$(\distsym, q)$-covers $\functionset$ on $\context_{1:n} \subseteq \contextspace$ at scale $\eps>0$} if
\*[
  \sup_{\func\in\functionset}
  \inf_{\coverfunc\in\coverset}
  \Bigg(\frac{1}{n} \sum_{t=1}^n \distsym^q(\func(\context_t), \coverfunc(\context_t)) \Bigg)^{1/q}
  \leq \eps.
\]
We denote the cardinality of the smallest such cover by $\gencov{\distsym}{q}(\functionset, \eps, \context_{1:n})$. The \emph{empirical $(\distsym, q)$-entropy of $\functionset$ on samples of size $n$ at scale $\eps$} is
\*[
  \genent{\distsym}{q}(\functionset, \eps, n)
  = \sup_{\context_{1:n} \subseteq \contextspace} \log \,\gencov{\distsym}{q}(\functionset, \eps, \context_{1:n}).
\]
Similarly, we say $\packingset \subseteq \functionset$ is a \emph{$(\distsym, q)$-packing of $\functionset$ on $\context_{1:n} \subseteq \contextspace$ at scale $\eps>0$} if
\*[
  \inf_{\func \neq \coverfunc \in \packingset}
  \Bigg(\frac{1}{n}\sum_{t=1}^n \distsym^q(\func(\context_t), \coverfunc(\context_t)) \Bigg)^{1/q}
  > \eps.
\]
We denote the cardinality of the largest such packing $\packingset$ by $\genpack{\distsym}{q}(\functionset, \eps, \context_{1:n})$. The \emph{empirical $(\distsym, q)$-packing entropy of $\functionset$ on samples of size $n$ at scale $\eps$} is
\*[
  \genpackent{\distsym}{q}(\functionset, \eps, n)
  = \sup_{\context_{1:n} \subseteq \contextspace} \log \,\genpack{\distsym}{q}(\functionset, \eps, \context_{1:n}).
\] 

These quantities are closely related, as shown by the following well-known result.
\begin{lemma}\label{fact:cover-pack}
For all $\eps>0$ and $\context_{1:n} \subseteq \contextspace$,
\*[
  \genpack{\distsym}{q}(\functionset, 2\eps, \context_{1:n})
  \leq \gencov{\distsym}{q}(\functionset, \eps, \context_{1:n})
  \leq \genpack{\distsym}{q}(\functionset, \eps, \context_{1:n}).
\]
\end{lemma}

Our main results, presented below, characterize the minimax rates for conditional density estimation in terms of the \emph{empirical Hellinger entropy}, defined by
\*[
\genent{\hellsym}{2}(\functionset,\eps,n) =\genent{\distsym_{\hellsym}}{2}(\functionset,\eps,n).
\]
That is, for the metric on densities in the general definition above we take $\distsym = \distsym_{\hellsym}$, the Hellinger distance.
Our use of the $\vecnorm{q}$ norm (in particular, $\vecnorm{2}$) on the empirical observations is crucial for our results. For more detailed discussion of uniform alternatives and why they fail, see \cref{sec:lit-uniform}.

\subsection{Minimax upper and lower bounds}\label{sec:main-thms}

To state our results in an easily interpretable fashion, we present results for classes $\functionset$ of either \emph{parametric} or \emph{nonparametric} complexity (for completeness, we also recover results for finite classes). 
\edited{Informally, we say that $\functionset$ is parametric if $\genent{\hellsym}{2}(\functionset, \eps,n)$ grows logarithmically in $\eps^{-1}$ and $\functionset$ is nonparametric if $\genent{\hellsym}{2}(\functionset, \eps,n)$ grows polynomially in $\eps^{-1}$. 
This dichotomy is standard in the minimax literature, and captures most classes of interest.}

\cref{fact:risk-upper-rates,fact:risk-lower-rates} are
simplifications of more general results, \cref{fact:risk-upper,fact:risk-lower}, which we state and prove in \cref{sec:main}. \cref{fact:risk-upper,fact:risk-lower} provide upper and lower bounds that hold for \emph{any} class of conditional densities, with no assumption on the growth rate for the empirical entropy. \cref{fact:risk-upper-rates,fact:risk-lower-rates} follow by applying these results with the parametric and nonparametric growth rates assumed in the theorem statements; see \cref{sec:rates-proofs} for detailed calculations.

\edited{In all cases, we say a sequence $\rate_n$ (which may depend on quantities such as $\funcdim$, $\densbound$, and $\datasize$)
is \emph{the minimax rate for $\functionset$} if $\minimaxrisk{\functionset}{n} \asymp \rate_n$; as is standard in the literature, we often refer to the minimax rate as $\rate_n$ even if this is only satisfied up to logarithmic factors.
Our main results (\cref{fact:risk-upper-rates,fact:risk-lower-rates}) establish that, for general classes satisfying either the parametric or nonparametric growth assumptions above, the minimax KL risk is determined by the \emph{critical radius} $\eps_n^2$ satisfying (up to logarithmic factors)
\*[
  \genent{\hellsym}{2}(\functionset, \eps_n, n)
  \asymp n \eps_n^2.
\]
The precise dependence of $\eps_n$ on logarithmic factors will depend on $\functionset$; we discuss the necessity of these logarithmic factors in \cref{sec:lit-uniform}.}
\edited{
\begin{theorem}[Main upper bound]\label{fact:risk-upper-rates}
For all $\funcdim,\entropyconst>0$:\\
If $\genent{\hellsym}{2}(\functionset, \eps, n) \lesssim \entropyconst \cdot \eps^{-\funcdim}$ and $\entropyconst \lesssim n^{\frac{\funcdim}{\funcdim+2}} \cdot \ind{\funcdim \leq 2} + n^{\frac{4}{\funcdim(\funcdim+2)}} \cdot \ind{\funcdim > 2}$,
\*[
  \minimaxrisk{\functionset}{n}
  \lesssim 
  \entropyconst^{\frac{2}{\funcdim+2}} 
  \cdot
  [\log(n\densbound\datasize)]^{\frac{\funcdim}{\funcdim+2}} 
  \cdot
  n^{-\frac{2}{\funcdim+2}}.
\]
If $\genent{\hellsym}{2}(\functionset, \eps, n) \lesssim \funcdim \cdot \log(\entropyconst/\eps)$ and $n \cdot \entropyconst^2 > \funcdim$,
\*[
  \minimaxrisk{\functionset}{n}
  \lesssim \frac{\funcdim}{n} \cdot \Bigg(\log(\entropyconst \sqrt{n}) + \log(n\densbound\datasize) \cdot\log(\entropyconst \sqrt{n} / \sqrt{\funcdim}) \Bigg) + \frac{\log(n\densbound\datasize) \cdot \log (n)}{n}.
\]
\end{theorem}

\begin{theorem}[Main lower bound]\label{fact:risk-lower-rates}
For all $\funcdim,\entropyconst>0$ and $n$ sufficiently large:\\
If $\genent{\hellsym}{2}(\functionset, \eps, n) \asymp \entropyconst \cdot \eps^{-\funcdim}$,
\*[
  \minimaxrisk{\functionset}{n}
  \gtrsim
  \entropyconst^{\frac{2}{\funcdim+2}} 
  \cdot
  [\log(n\densbound\datasize)]^{-\frac{2}{\funcdim+2}} 
  \cdot
  n^{-\frac{2}{\funcdim+2}}.
\]
If $\genent{\hellsym}{2}(\functionset, \eps, n) \asymp \funcdim \cdot \log(\entropyconst/\eps)$,
\*[
  \minimaxrisk{\functionset}{n}
  \gtrsim
  \frac{\funcdim}{n} \cdot \Bigg(\frac{1}{\log(n\densbound\datasize)} \Bigg).
\]
\end{theorem}
}

\edited{
\begin{remark}
Since the Hellinger distance is upper bounded by KL, \cref{fact:risk-upper-rates} immediately implies a bound on Hellinger risk. Similarly, inspection of the proof of \cref{fact:risk-lower-rates} shows that we actually lower bound Hellinger risk, and hence our results imply that Hellinger and KL risk are equivalent up to logarithmic factors. It remains open to find a density class where they have differing dependence on logarithmic factors.
\end{remark}
}

\begin{remark}
To emphasize the rates as a function of $n$, \cref{fact:risk-upper-rates} suggests that $\entropyconst$ cannot depend on $\densbound$ and $\datasize$. In fact, all results still hold under a slightly more complex condition on $\entropyconst$ (with explicit dependence on $\densbound$ and $\datasize$, which can be extracted directly from the end of the proof of \cref{fact:risk-upper-rates} in \cref{sec:rates-proofs}). In our examples where $\densbound$ and $\datasize$ are functions of $n$, this condition reduces to requiring $\entropyconst$ (depending on $n$ through $\densbound$ and $\datasize$) to satisfy the bound that is stated in \cref{fact:risk-upper-rates}.
\end{remark}

\input{section-files/minimax-estimator}
\input{section-files/examples}
\subsection{Discussion}\label{sec:lit-uniform}
We close this section by discussing some of the key technical features of \pref{fact:risk-upper-rates,fact:risk-lower-rates}, as well as limitations and opportunities for improvement. %

\input{section-files/log-factors}

\paragraph*{Comparison with uniform joint entropy}
As discussed in the introduction, a key aspect that distinguishes our results from the classical literature on density estimation is that we work with the empirical Hellinger entropy of the conditional density class, which i) is data-dependent, and ii) captures only the difficulty of estimating the conditional distribution over responses, and \emph{not} the difficulty of estimating the marginal distribution over covariates.
To illustrate this improvement, we now highlight that a) without assumptions on the marginal distribution of the covariates, which is the setting we study, uniform joint entropy leads to vacuous rates, and b) even under some standard assumptions on the marginal distribution of the covariates, which is an easier problem than the one we consider, uniform joint entropy leads to suboptimal rates. 
Let us formally define uniform joint entropy.
Suppose there  exists a finite measure $\jointreldistn$ on $\contextspace \times \dataspace$, and let $\densityset(\contextspace \times \dataspace)$ denote the collection of densities on $\contextspace \times \dataspace$ with respect to $\jointreldistn$. 
For a class of densities $\jointfunctionset \subseteq \densityset(\contextspace \times \dataspace)$, we say $\jointcoverset \subseteq \jointfunctionset$ is a \emph{uniform Hellinger cover} at scale $\eps>0$ if
\*[
  \sup_{\jointfunc\in\jointfunctionset}
  \inf_{\jointcoverfunc\in\jointcoverset}
   \distsym_{\hellsym}(\jointfunc, \jointcoverfunc) 
   \leq \eps,
\] 
where the Hellinger distance for joint distributions is given by
\*[
  \distsym_{\hellsym}(\jointfunc, \jointcoverfunc)
  = \sqrt{\frac{1}{2}\int \Big(\sqrt{\jointfunc(\context,\data)} - \sqrt{\jointcoverfunc(\context,\data)} \, \Big)^2 \jointreldistn(\dee (\context, \data))}.
\]
We denote the cardinality of the smallest $\jointcoverset$ that achieves this by $\genjointcov{\hellsym}(\jointfunctionset,\eps)$. Existing minimax results for density estimation \citep{wong95inequalities,yang99information} bound the minimax risk of estimating the data-generating density $\truejointfunc\in\jointfunctionset$ in terms of the \emph{uniform joint entropy} $\genjointent{\hellsym}(\jointfunctionset,\eps) = \log \genjointcov{\hellsym}(\jointfunctionset,\eps)$. 

To compare the empirical conditional entropy to the uniform joint entropy, consider a set $\covfunctionset$ of marginal densities on $\contextspace$ and class $\functionset\subseteq\kernelset(\contextspace, \dataspace)$, and define
\*[
  \jointfunctionset_{\covfunctionset,\functionset}
  = \Big\{(\context,\data) \mapsto \jointfunc_{p,\func}(\context,\data) = p(\context)[\func(\context)](\data) \setdelim p \in \covfunctionset, \func \in \functionset \Big\}.
\]
For any $p,p'\in\covfunctionset$ and $\func\in\functionset$, $\distsym_{\hellsym}(\jointfunc_{p,\func}, \jointfunc_{p',\func}) = \distsym_{\hellsym}(p, p')$ since $\data \mapsto [\func(\context)](\data)$ integrates to one for almost all $\context$.
Thus, $\genjointent{\hellsym}(\jointfunctionset_{\covfunctionset,\{\func\}}, \eps) = \genjointent{\hellsym}(\covfunctionset, \eps)$ for all $\func\in\functionset$, so by monotonicity of metric entropy we conclude $\genjointent{\hellsym}(\jointfunctionset_{\covfunctionset,\functionset}, \eps) \geq \genjointent{\hellsym}(\covfunctionset, 2\eps)$ for all $\functionset$.
Trivially, if $\genjointent{\hellsym}(\covfunctionset, 2\eps) = \infty$, then
$\genjointent{\hellsym}(\jointfunctionset_{\covfunctionset,\functionset}, \eps) = \infty$. This is the case when, for example, 
$\contextspace=\Reals$ and $\covfunctionset$ is the set of all densities with respect to Lebesgue measure. In contrast, our results provide tight rates for this same setting.

The previous example has already ruled out the suitability of uniform joint entropy for conditional density estimation.
However, even if $\genjointent{\hellsym}(\jointfunctionset_{\covfunctionset,\functionset}, \eps) < \infty$, uniform joint entropy may still lead to suboptimal rates under ``natural'' assumptions on $\covfunctionset$.
For example, if $\contextspace$ is a compact, 
convex subset of $\Reals^d$, the marginal entropy of $\covfunctionset$ grows as $\eps^{-d}$ for Lipschitz densities \citep{dudley74entropy} or as $\eps^{-d-1}$ for almost-isotropic log-concave densities \citep{kur19logconcave}. Since we generally expect that the dimension of $\contextspace$ is large 
compared to the dimension of $\dataspace$ (commonly $1$, as in regression and classification), this dependence is prohibitively suboptimal even when $\cF$ has moderate nonparametric complexity.
By characterizing performance in terms of the entropy for the conditional densities, our results only depend on the complexity of $\contextspace$ through the complexity of $\functionset$ itself. For example, if $\functionset$ is parametric, our dependence on the dimension of $\contextspace$ will be linear at worst, while the previous joint examples will have an exponential dependence.

These examples demonstrate why the entropy for the conditional density class $\functionset$ is the correct object to consider in our setting, rather than the joint entropy of $\jointfunctionset_{\covfunctionset,\functionset}$.  If one does not use the \emph{empirical} entropy for $\cF$, however, one may obtain suboptimal rates. For example, a commonly used alternative notion of covering is to say that $\coverset$ uniformly covers $\functionset$ at scale $\eps$ if
\*[
  \sup_{\func\in\functionset} \inf_{\coverfunc\in\coverset} \sup_{\context\in\contextspace} \distsym_{\hellsym}(\func(\context), \coverfunc(\context))\leq\veps.
\]
For large (continuous) $\contextspace$, this may be vacuously large; this phenomenon is well-known within the literature on regression with the square loss \citep{lugosi99empirical}.

Finally, we briefly remark that if the covariate density $\contextdensity$ is known in advance, the conditional density estimation problem reduces to joint density estimation over the class $\{(\context, \data) \mapsto \contextdensity(\context) [\func(\context)](\data) \setdelim \func \in \functionset\}$. However, the covariate density is rarely known in practice, which necessitates the use of empirical entropy.

\paragraph*{On the bounded density assumption}
To handle all classes simultaneously, we require the conditional densities to be bounded above, since otherwise (without additional assumptions) the KL risk may be infinite. For certain classes (e.g., exponential families), a more specialized analysis could likely replace this assumption with weaker tail conditions, or remove it altogether. 
Crucially, unlike related work \citep{zhang06information,grunwald20fastrates} that obtains bounds for classes satisfying certain tail conditions on the logarithmic loss composed with the density class, we do not require the true conditional density to be bounded below, which would rule out many classes of interest, particularly for classification applications. For a more detailed discussion, see \cref{sec:lit-modern-density}.

The requirement that the reference distribution $\nu$ over $\dataspace$ admits a finite bound $\datasize$ appears as an artifact of the analysis of our minimax estimator (\cref{sec:minimax-estimator}). While such an assumption is common in the minimax estimation literature, it is unclear whether it is necessary to obtain rates for all classes simultaneously.
Ultimately, however, the logarithmic dependence on $\densbound$ and $\datasize$ in our results is fairly mild.

In more detail, observe that in many real statistical applications, the response space $\dataspace$ is often fairly simple. At the same time, the covariate space $\contextspace$ can be anything but simple. For classification (finite $\dataspace$), $\densbound=1$ and $\datasize = \abs{\dataspace}$, and consequently the dependence on $\dataspace$ in our bound is only of size $\log\abs{\dataspace}$. Another well-studied setting is where $\dataspace$ is a ball of radius $r$ in $d$ dimensions, in which case $\datasize = r^d$, so the dependence on $\dataspace$ in our bound is $d \log r$. For the commonly encountered situation in which $\contextspace \subseteq \Reals^p$ for some $d \ll p$, this dependence on $d$ is of lower order. (Specifically for \cref{fact:linear-example}, one can take $d=1$ and $r=\datasize$.)

A natural choice for $\datareldistn$ 
is the uniform distribution, if such a distribution exists. 
When no such distribution exists or such a choice does not permit a bound on the densities,
one may design $\nu$ for the model class of interest (e.g.,
by picking $\datareldistn$ to have slightly heavier tails).
See \cref{fact:linear-example} for specific realizations of $\datareldistn$ that admit unbounded $\dataspace$.

%% file: section-files/minimax-estimator.tex
\subsection{Minimax estimator}
\label{sec:minimax-estimator}
We now describe the estimator that achieves the minimax upper bound in
\pref{fact:risk-upper-rates} and provide some intuition behind its
performance. The main difficulty in controlling performance under the
KL risk is in handling the likelihood ratio, which a-priori is
potentially unbounded. The obvious way to remove this problem is to require that all densities in $\functionset$ are bounded away from zero, but this would eliminate many classes of interest (e.g., shape-constrained densities). 
A first observation is that, under \pref{ass:bounded-density}, we can restrict our
estimator to be bounded away from zero without losing any generality
in the conditional density class $\functionset$ via \emph{smoothing}. In particular, for every $q \in \densityset(\dataspace)$ and $\smoothparam > 0$, define the \emph{$\smoothparam$-smoothed density} 
\*[
	\smooth{\smoothparam}{q}(\data)
	= \frac{q(\data) + \smoothparam/\datasize}{1+\smoothparam},
	\qquad \data\in\dataspace.
\]
Note that defining such a uniformly smoothed estimator necessitates
the existence of a reference measure $\datareldistn$ that admits
finite $\datasize$ for $\dataspace$. However, this does not
necessarily require that $\dataspace$ itself be bounded; for instance,
in many standard settings it suffices to take $\datareldistn$ to be another probability distribution with slightly heavier tails (see \cref{fact:linear-example}).

Our estimator is parametrized by scalars $\eps,\smoothparam>0$. Given a sample
$(\context_1,\data_1),\ldots,(\context_{2n},\data_{2n})$, the estimator splits the sample
into two halves $(\context'_{1:n}, \data'_{1:n})$ and $(\context_{1:n}, \data_{1:n})$.
The main idea behind the estimator
is to use the first half of the sample to form an empirical
Hellinger cover, then use the second half of the sample to perform
aggregation over a smoothing of the cover.

In more detail, let $\coverset$ be an $(\hellsym, 2)$-cover of $\functionset$ on $\dumcontext_{1:n}$ at scale $\eps$. 
Then, for each $t\in[n]$, define the mixture forecaster
\*[
  [\predfunc{t}(\context)](\data)
  = \sum_{\coverfunc \in \coverset} [\smooth{\smoothparam}{\coverfunc}(\context)](\data) 
  \left[ 
  \frac{\prod_{s=1}^{t} [\smooth{\smoothparam}{\coverfunc}(\context_s)](\data_s)}{\sum_{\coverfuncdumm \in \coverset} \prod_{s=1}^{t} [\smooth{\smoothparam}{\coverfuncdumm}(\context_s)](\data_s)}
  \right],
  \qquad \context\in\contextspace,\data\in\dataspace.
\]
The estimated conditional density is given by the Cesaro average of the mixture forecasters:
\*[
  [\riskpredfunc{n}(\context)](\data)
  = \frac{1}{n+1}\sum_{t=0}^n [\predfunc{t}(\context)](\data),
  \qquad \context\in\contextspace,\data\in\dataspace.
\]

The weights of the mixture forecaster correspond to the posterior probability that each smoothed conditional density generated the observed data under a uniform prior.
Such posteriors over a cover for the density class have previously been used to obtain minimax rates for joint density estimation (see, for example, \citep{yang99information}). 
Crucially for this work, however, our estimator uses the posterior over an
\emph{empirical cover} that must be estimated from data, and we explicitly smooth the estimator in order to avoid unbounded likelihood ratios.

%% file: section-files/examples.tex
\subsection{Examples}
\label{sec:examples}
We now instantiate \pref{fact:risk-upper-rates,fact:risk-lower-rates}
for concrete function classes of interest. 
The main contribution of this section is to provide the minimax rates for three models that existing joint density estimation bounds \emph{fail} to capture. Our examples include a high-dimensional covariate space, an unbounded covariate space, and multiple conditional analogues of smoothness classes. In each example, we obtain upper bounds using \cref{fact:risk-upper-rates} (and, in most cases, matching lower bounds using \cref{fact:risk-lower-rates}), and the rates are achieved by the minimax estimator introduced in \cref{sec:minimax-estimator}. Proofs for examples are deferred to \pref{app:examples-proofs}.

\edited{
To instantiate our main theorem, all that is required is to compute the empirical Hellinger entropy for the conditional density class of interest. 
While the empirical $\funcnorm{2}$ entropy for regression functions has been extensively studied, empirical Hellinger entropy cannot be tightly bounded by $\funcnorm{r}$ entropy for any $r$.
This is because the tightest generic inequality possible for the Hellinger distance is
\[\label{eqn:hellinger-tv-inequality}
  \distsym_{1} \leq \distsym_{\hellsym} \leq \sqrt{\distsym_{1}},
\]
and which side of the inequality is tight will depend on the geometry of the densities.
In our examples, for some classes (e.g., smooth densities) we find that $\onecov{\hellsym} \asymp \onecov{\norm{\cdot}^{1/2}_1}$, while for others (e.g., linear functions) we find that $\onecov{\hellsym} \asymp \onecov{\norm{\cdot}_1}$.
In the context of joint density estimation, \citet{birge86hellinger} makes a similar observation, noting that this adaptive behavior is what allows $\distsym_\hellsym$ to correctly capture the complexity of (joint) density estimation. 
This contrasts with the problem of conditional mean estimation under square loss, 
where $\funcnorm{2}$ entropy is sufficient to completely characterize the minimax rates \citep{rakhlin17minimax}.

\paragraph*{Dimension-free generalized linear models}
Consider univariate generalized linear models (GLMs) with exponential family responses and canonical link functions \citep[c.f.\ ][]{glmbook}. 
Classical asymptotic rates for consistency and normality are well-understood for GLMs, where the maximum likelihood error scales with the Fisher information \citep{fahrmeir85generalized}.
We now use our main results to extend the study of GLMs beyond this classical setting.

Parameter estimation for GLMs is \emph{exactly} a regression problem, yet in this work we are interested in a full characterization of the uncertainty of the response variable through the conditional density. 
Hence, we depart from the classical theory by studying the KL error of the estimated conditional density rather than error on the parameter scale, which \emph{a priori} may be unbounded even for bounded parameters.
For exponential families satisfying certain tail conditions (e.g., logistic regression with probabilities restricted to some fixed interval $[c,1-c]$ for $c > 0$),
earlier work \citep{zhang06information,deheide20generalized,grunwald20fastrates} obtains the parametric KL rate $\paramdim \cdot (\log n) \cdot n^{-1}$ (where $\paramdim$ is the covariate dimension).
Without any tail assumptions (although we generally require bounds on the scale parameter to ensure the density is bounded above), our \cref{fact:risk-upper-rates} recovers this rate up to logarithmic factors for generic exponential family responses, since the entropy is parametric.

In the regime where the parameter dimension scales with $n$, this parametric rate is no longer useful. Under a sparsity assumption that only $\sparsedim$ entries of the linear regressor are non-zero, it has been shown in a variety of cases that minimax risk of $\sparsedim \cdot (\log \paramdim) \cdot n^{-1}$ is possible \citep{vandegeer90regression,raskutti11minimax}.
In the case of a dense linear predictor or exponentially large dimension, these rates become vacuous.
In the linear regression setting---where similar parametric dependence on dimension arises---it is known that, in fact, dimension-free rates are possible \citep{bartlett98neuralnets}.
Finally, for Bernoulli responses with a linear link function, Lemma~8 of \citet{rakhlin15binary} (combined with \cref{fact:regret-lower}) implies that the minimax risk is no larger than $\sqrt{(\log n) \cdot n^{-1}}$.
Our contribution in this section is to show that the same dimension-free rates are possible for the KL risk  of conditional density estimation of generalized linear models on $\Reals$.

\begin{example}\label{fact:linear-example}
Let $\hilbertspace$ be a Hilbert space, and for any
radius $\linearrad > 0$ let
$\Bb(\linearrad) = \{v\in\hilbertspace \setdelim \norm{v} \leq \linearrad\}$.
For each of the following, the polynomial dependence on $n$ is tight.

\begin{itemize}
\item[{\upshape (a)}]
Fix $\sigma, \xbound, \wbound > 0$. Let $\dataspace = \Reals$ and set $\contextspace = \Bb(\xbound)$ and $\weightspace = \Bb(\wbound)$. 
If
\*[
  \functionset_{\weightspace} = \Big\{\context \mapsto \normaldist\Big(\inner{\context}{\weight}, \sigma^2\Big) \setdelim \weight \in \weightspace\Big\},
\]
then
\*[
  \minimaxrisk{\functionset_{\weightspace}}{n} \lesssim \, n^{-1/2} \cdot
  \frac{\xbound\wbound\cdot \sqrt{\log n}\cdot \sqrt{\log(n \xbound^2 \wbound^2 / \sigma)}}{\sigma}.
\]

\item[{\upshape (b)}]
Fix $\xbound, \wbound > 0$. Let $\dataspace = \Nats$ and set $\contextspace = \Bb(\xbound)$ and $\weightspace = \Bb(\wbound)$. 
If
\*[
  \functionset_{\weightspace} = \Big\{\context \mapsto \poissondist\Big(\exp\{\inner{\context}{\weight}\}\Big) \setdelim \weight \in \weightspace\Big\},
\]
then
\*[
  \minimaxrisk{\functionset_{\weightspace}}{n} \lesssim \, n^{-1/2} \cdot
  \xbound\wbound e^{\xbound\wbound/2} \cdot (\log n).
\]

\item[{\upshape (c)}]
Fix $\alpha, \gamma, \xbound, \wbound > 0$. Let $\dataspace = [0,\infty)$ and set $\contextspace = \Bb(\xbound)$ and $\weightspace = \Bb(\wbound)$. 
If
\*[
  \functionset_{\weightspace} = \Big\{\context \mapsto \gammadist\Big(\alpha, -\frac{1}{\alpha(\inner{\context}{\weight}-\xbound\wbound-\gamma)}\Big) \setdelim \weight \in \weightspace\Big\},
\]
then
\*[
  \minimaxrisk{\functionset_\weightspace}{n}
  \lesssim
  n^{-1/2} \cdot (\log n) \cdot \sqrt{\frac{\xbound\wbound\max\{\alpha,1\}}{\gamma}\log\Bigg(n \cdot \Big( \frac{\xbound\wbound+\gamma}{\gamma}\Big)^\alpha \Bigg)}.
\]
\end{itemize}

\end{example}

\paragraph*{Vapnik--Chervonenkis (VC) classes}
While the above example has removed all dependence on the dimension of the covariate space, the rate still depends on the norm of the linear regressor, which is generally unavoidable without assumptions on $\contextdistn$ \citep{lee20minimax}.
However, if the conditional density map is further restricted, we now show that it is possible to consider unbounded covariate spaces without any assumptions on $\contextdistn$.

For the next example, we consider a classification setting
where $\dataspace = \{0,1\}$, so that any density in
$\densityset(\dataspace)$ can represented by a scalar in $[0,1]$
parameterizing the mean of a Bernoulli random variable. Taking $\datareldistn$ to be uniform on $\{0,1\}$ gives $\datasize = 1$ and $\densbound = 2$.
This setting has been extensively studied, particularly in the
sequential, non-\iid{} setting (see \cref{sec:lit-sequential} for a survey), and is
closely connected to the problem of universal compression.
A common measure of complexity for classification is the \emph{VC-dimension}, and the natural extension to real-valued functions is the \emph{pseudo-dimension}. 
Our example concerns the case where the conditional density class has a finite pseudo-dimension; for completeness, we first restate the relevant definitions.

Let $\powerset(\contextspace)$ denote the power set of $\contextspace$.
For any $\vcclass \subseteq \powerset(\contextspace)$, 
its \emph{VC-dimension} is 
\*[
  \vcdim(\vcclass) = \sup\Big\{n \in \Nats \Bigsetdelim \exists \, \context_{1:n} \subseteq \contextspace \stT \abs{\{\vcfunc \cap \context_{1:n} \setdelim \vcfunc \in \vcclass\}} = 2^n\Big\}.
\]
For any $\functionset \subseteq \contextspace \to \Reals$,
its \emph{pseudo-dimension} is
\*[
  \vcind(\functionset)
  = 
  \vcdim\Big(\Big\{\{(\context,\vcval) \in \contextspace\times\Reals \setdelim \func(\context) \geq \vcval\} \Bigsetdelim \func\in\functionset \Big\}\Big).
\]

Then, we have the following result.
\begin{example}\label{fact:vc-example}
If $\contextspace$ is arbitrary, $\dataspace = \{0,1\}$,
and $\functionset \subseteq \contextspace \to [0,1]$,
then
\*[
  \minimaxrisk{\functionset}{n}
  \lesssim \, \frac{\vcind(\functionset) \, (\log n)^2}{n}.
\]
\end{example}
\citet{bhatt21logloss} consider a special case of this problem in the sequential (but still \iid{}) setting when $\functionset = \{\context \mapsto \vcsmooth_{\ind{\context\in\vcfunc}} \setdelim \vcfunc \in \vcclass, \vcsmooth_0, \vcsmooth_1 \in [0,1] \}$ for some fixed $\vcclass \subseteq \powerset(\contextspace)$. 
They provide a risk upper bound $\minimaxrisk{\functionset}{n}
\lesssim \sqrt{\vcdim(\vcclass) \cdot n^{-1}}$ (see \cref{fact:regret-lower} to
relate their result---which is stated in terms of regret---to a risk bound), which has a gap from their lower bound, $\vcdim(\vcclass) \cdot n^{-1}$. 
To resolve this gap, we require the following lemma, which we also prove in \cref{app:examples-proofs}.
\begin{lemma}\label{fact:vcind-vcdim}
For any $\vcclass \subseteq \powerset(\contextspace)$, if $\functionset = \{\context \mapsto \vcsmooth_{\ind{\context\in\vcfunc}} \setdelim \vcfunc \in \vcclass, \vcsmooth_0, \vcsmooth_1 \in [0,1] \}$ then
\*[
  \vcind(\functionset) \leq 6 \cdot \vcdim(\vcclass).
\]
\end{lemma}
Combining \cref{fact:vc-example,fact:vcind-vcdim}, our minimax estimator achieves an upper
bound that significantly improves on the $n^{-1/2}$ rate, matching the lower bound of \citet{bhatt21logloss} within $\log(n)$ factors.

So-called ``fast rates'' of the form $\vcind(\functionset)\cdot n^{-1}$ for VC classes have been extensively studied in the literature. 
For binary outcomes (i.e., \cref{fact:vc-example}), constraining the probabilities away from the boundary of $[0,1]$ ensures that the logarithmic loss is bounded above and below by square loss (similarly for Hellinger distance or total variation), and hence classical results for regression apply (e.g., Theorem~8 of \citep{haussler92vc}).
Without such boundedness assumptions, the discussion in Sections~2.3 and 7 of \citet{yang99information} implies fast rates for the KL risk of joint density estimation.
However, even classes with $\vcdim(\vcclass)=1$ may not admit a uniform cover (e.g., halfspaces on $\Reals$), and consequently our use of empirical entropy is crucial to obtain non-vacuous rates.
Further, it is possible to have $\vcdim(\vcclass) = 1$ yet fail to be \emph{online learnable} (again, one-dimensional halfspaces witness this; see Example~21.4 of \citep{UML14}), in which case the upper bounds provided by worst-case sequential analyses (e.g., Theorem~7 of \citep{foster18logistic}) will be vacuous. 
For the KL risk of conditional density estimation, \citet{zhang06information} and \citet{grunwald20fastrates} obtain fast rates under boundedness assumptions on the densities. While these assumptions are weaker than strictly constraining the densities away from the boundary of $[0,1]$, they fail for simple classes (including the $\functionset$ in \cref{fact:vcind-vcdim}).
In summary, to the best of our knowledge, \cref{fact:vc-example} provides the first fast rates for KL risk of VC conditional density classes without any boundedness assumptions.

\paragraph*{Smoothness classes}

Smooth densities are one of the most fundamental classes studied in
joint density estimation, and their entropy is well-understood (e.g.,
Theorem~3 of \citep{ibragimov83density}). 
For conditional densities, there are multiple possible notions of smoothness, and these have been studied on a case-by-case basis in the literature with a new (kernel-based) estimator for each setting.
We now consider three possible notions, providing novel rates for the KL risk in each case.

First, consider the case when the response space is discrete; that is, conditional multinomial estimation.
\begin{example}\label{fact:multinomial-smooth-example}
Let $\contextspace = [0,1]^\funcdim$, $\dataspace = [\multidim]$, $\holderparam\in(0,1]$, and $\numderivs\in\PosInts$ such that
\*[
  \functionset 
  =
  \left\{\context \mapsto (\holderfunc_1(\context),\dots,\holderfunc_\multidim(\context)) \
  \middle\vert
  \begin{array}{l}
  \ \forall \context,\dumcontext\in\contextspace, j\in[\multidim],
  \ \sup_{\derivvec: \norm{\derivvec}_1\leq\numderivs}\abs[0]{D^\derivvec \holderfunc_j(\context)} \leq 1
  \ \text{ and } \ \\
  \sup_{\derivvec: \norm{\derivvec}_1=\numderivs}\abs[0]{D^\derivvec\holderfunc_j(\context) - D^\derivvec\holderfunc_j(\dumcontext)}
  \leq \norm{\context - \dumcontext}_\infty^\holderparam
  \end{array}\right\}.
\]
Then, up to logarithmic factors, $\minimaxrisk{\functionset}{n} \asymp \multidim^{\frac{2\paramdim}{\paramdim+\numderivs+\holderparam}} \cdot n^{-\frac{\numderivs+\holderparam}{\paramdim+\numderivs+\holderparam}}$.
\end{example}
In this setting with $\holderparam=1$ and $\numderivs=0$, Theorem~2 of \citet{gyorfi07conditional} shows that the total variation risk is no larger than $n^{-\frac{1}{\paramdim+2}}$. The minimax KL rate of \cref{fact:multinomial-smooth-example}, which is $n^{-\frac{1}{\paramdim+1}}$, cannot be recovered from this result (and similarly the total variation risk cannot be recovered from our rate) using either of the inequalities in \cref{eqn:hellinger-tv-inequality}. Once again, this demonstrates how the KL risk is qualitatively different due to the curvature of the loss near the boundary.

For a continuous response space, more structure is required. Theorem~2.2 of \citet{li21holder} shows that, in general, the conditional density must be smooth in \emph{both} the covariates and the responses in order for the model to be estimable. One way to enforce this is to consider sufficiently smooth parametric densities with the parameters themselves generated by smooth maps; a concrete example of such parametric densities are exponential family models (see the proofs of \cref{fact:linear-example}). In particular, we have the following result, which follows immediately from \cref{fact:multinomial-smooth-example}.
\begin{example}\label{fact:parameter-smooth-example}
Let $\contextspace = [0,1]^\funcdim$, $\dataspace = \Reals$, $\parambound > 0$, and $\holderparam\in(0,1]$.
Suppose $\paramdensityset = \{\paramdensity{\param} \setdelim \param \in [0,\parambound]\}$ satisfies
$ \distsym^2_{\hellsym}(\paramdensity{\param}, \paramdensity{\paramdum})
  \leq \abs{\param - \paramdum}$
for all $\param,\paramdum \in [0,\parambound]$, and define
\*[
  \functionset_{\paramdensityset} 
  =
  \left\{\context \mapsto \paramdensity{\holderfunc(\context)} \
  \middle\vert
  \begin{array}{l}
  \ \forall \context,\dumcontext\in\contextspace,
  \ \holderfunc(\context) \in [0,\parambound]
  \ \text{ and } \
  \abs[0]{\holderfunc(\context) - \holderfunc(\dumcontext)}
  \leq \norm{\context - \dumcontext}_\infty^\holderparam
  \end{array}\right\}.
\]
Then, up to logarithmic factors, $\minimaxrisk{\functionset}{n} \lesssim \parambound^{\frac{2\paramdim}{\paramdim+\holderparam}} \cdot n^{-\frac{\holderparam}{\paramdim+\holderparam}}$.
\end{example}

More generally, using recent \emph{functional} metric entropy results from \citet{park22entropy}, we obtain upper bounds (packing numbers, and hence lower bounds, remain an open problem) on the risk for models with smoothness constraints on both the covariates and the responses.
\begin{example}\label{fact:general-smooth-example}
Fix $\contextspace = [0,1]^\funcdim$, $\dataspace = [0,1]^\datadim$, $\holderparam\in(0,1]$, $C,L>0$, and $\numderivs,\nummderivs\in\PosInts$. If
\*[
  \functionset 
  =
  \left\{\context \mapsto \Big(\data \mapsto [\func(\context)](\data) \Big)\
  \middle\vert
  \begin{array}{l}
  \ \forall 
  \context\in\contextspace, 
  \data,\dumdata\in\dataspace,
  \ \derivvec\in\{0,\dots,\numderivs\}^\funcdim \text{ s.t. } \norm{\derivvec}_1 \leq \numderivs, \\
  \ \derivvecc\in\{0,\dots,\nummderivs\}^\datadim \text{ s.t. } \norm{\derivvecc}_1 \leq \nummderivs, \\
  \ \abs[0]{D^{\derivvecc}[D^\derivvec \func(\context)](\data)} \leq C \ \text{ and } \ \\
  \ \abs[0]{D^{\derivvecc}[D^\derivvec \func(\context)](\data) - D^{\derivvecc}[D^\derivvec \func(\context)](\dumdata)}
  \leq L \norm{\data-\dumdata}_\infty^\holderparam.
  \end{array}\right\},
\]
then up to factors in $L$ and $C$ and logarithmic factors, $\minimaxrisk{\functionset}{n} \lesssim n^{-\frac{1}{1 + \funcdim/\numderivs + \datadim/(\nummderivs+\holderparam)}}$.
\end{example}
}

\paragraph*{Conditional mean estimation}
In the context of nonparametric regression, \citet{rakhlin17minimax}
characterized the minimax rates for conditional mean estimation with
square loss, showing that if $\genent{2}{2}(\functionset, \eps, n)
\asymp \eps^{-\funcdim}$, the minimax rate for $L^{2}$ risk has $\EE (\truefunc(\context) - \predfunc{n}(\context))^2 \asymp n^{-2/(2+\funcdim)}$, where $\functionset \subseteq \contextspace \to \dataspace$ is a class of conditional \emph{means}.
In particular, their Theorem 2 provides an upper bound that holds for
every function class and noise distribution (with $\dataspace
\subseteq [0,1]$), while their Theorem 7 provides a specific function
class and noise distribution ($\dataspace = \{0,1\}$) for which the
lower bound holds. It appears to be folklore that their Theorem 7 can
be upgraded to show that for \emph{every} function class there exists
a noise distribution such that the lower bound holds. 
Here we derive this result as a corollary of our lower bound to facilitate ease of use for both experts and non-experts.

\begin{example}\label{fact:mean-example}
Let $\contextspace$ be arbitrary, $\dataspace = \Reals$, and
$\functionset \subseteq (\contextspace \to [0,1])$ satisfy
$\genent{2}{2}(\functionset, \eps, n) \asymp \eps^{-\funcdim}$ for
some $\funcdim>0$. Then, 
\*[
  \inf_{\hat \normalmean} \sup_{\contextdistn} \sup_{\truefunc\in\functionset}
  \EE_{\context_{1:n} \sim \contextdistn}
  \EE_{\data_{1:n} \sim \normaldist(\truefunc(\context_{1:n}),1)}
  \EE_{\context\sim\contextdistn}
  \Big(\truefunc(\context) - \hat \normalmean_n(\context)\Big)^2
  \,
  \gtrsim \, n^{-\frac{2}{2+\funcdim}}
\]
up to logarithmic factors, where $\hat \normalmean$ ranges over conditional mean estimators.
\end{example}

%% file: section-files/log-factors.tex
\edited{
\paragraph*{On logarithmic factors}
\cref{fact:risk-upper-rates,fact:risk-lower-rates} combined characterize the polynomial dependence on $n$ for \emph{all} function classes with parametric or nonparametric empirical Hellinger entropy.
For marginal density estimation (no covariates), the dependence on logarithmic factors is unresolved for general classes; for example, the logarithmic factors do not match in \citet{yang99information}, \citet{zhang06information}, or \citet{rigollet12glm}, and are ignored in more recent work such as \citet{grunwald20fastrates}.
Further, information-theoretic lower bounds are generally suboptimal with regards to logarithmic factors, which is a consequence of them holding for \emph{all} function classes simultaneously rather than a specific function class constructed for a counterexample.
For a related discussion of logarithmic factors in risk bounds, we direct the reader to Section~6.2.2 of \citet{birge98mce}.
We now briefly discuss more concretely the tightness of logarithmic factors for our setting.

When $\functionset$ is an arbitrary finite class, our analysis recovers $\minimaxrisk{\functionset}{n} \lesssim (\log n) \cdot \log(n\densbound\datasize) / n$ (follows from part iii) of \cref{fact:loc-radius}), yet the correct rate is well-known to be $(\log\abs{\functionset})/n$.
This can be obtained with no additional assumptions and with no dependence on $\densbound$ and $\datasize$ by using, for example, Vovk's aggregating algorithm. A simple proof can be found by combining online-to-batch (\cref{fact:regret-lower}) with the argument in Section 9.5 of \citet{plg06book}, and a matching lower bound is provided by Theorem~3 of \citet{lecue06lower}.
The looseness of our rate for finite classes is a consequence of our method and analysis applying to \emph{every} function class simultaneously.

Beyond finite classes, it is impossible for the logarithmic factors to be tight in all cases without further assumptions beyond entropy conditions. Gaussian mean estimation with a known variance has a known minimax rate of $1/n$ (no logarithmic factors), and has entropy that grows like $\log(1/\eps)$. However, Example~6 of \citet{rakhlin17minimax}---which applies in our setting since KL is lower bounded by square loss for Bernoulli conditional mean estimation---provides a function class that also has entropy that grows like $\log(1/\eps)$, yet has minimax risk lower bounded by $(\log n)/n$. 

We now demonstrate that our dependence on logarithmic factors can also not be improved in general without further assumptions, since they are tight for \emph{some} function class.
As a consequence, we conclude that while Hellinger empirical entropy correctly captures the polynomial dependence on $n$ for the minimax rates, it is insufficient to fully characterize the logarithmic factors. It remains an open question to provide a more refined characterization of function classes that does capture this dependence.

\begin{fact}\label{fact:tight-logBK}
There exists an $\functionset$ for which the dependence on $\densbound$ and $\datasize$ is tight. 
\end{fact}
\begin{proof}[Proof of \cref{fact:tight-logBK}]
We realize this lower bound for the unconditional density estimation setting, which is a subset of conditional density estimation. In particular, let $\contextspace = \emptyset$, $\datareldistn$ be Lebesgue measure, and $\functionset$ be all monotone densities on $[0,\datasize]$ bounded above by $\densbound$.
Since metric entropy is no larger than \emph{bracketing} entropy for any metric\footnote{One can always construct a cover by selecting $\coverfunc\in B \cap \functionset$ for each bracket $B$.}, Theorem~3 of \citet{gao09monotone} (instantiated with $k=1$) implies that
\*[
	\genent{\hellsym}{2}(\functionset, \eps, n) \lesssim \sqrt{\log(\densbound\datasize)} \cdot \eps^{-1}.
\]
Thus, \cref{fact:risk-upper-rates} implies that
\*[
	\minimaxrisk{\functionset}{n}
	\lesssim [\log(\densbound\datasize)]^{1/3} \cdot [\log(n\densbound\datasize)]^{1/3} \cdot n^{-2/3}.
\]
However, by Pinsker's and Eq.~(1.3) of \citet{birge87density},
\*[
	\minimaxrisk{\functionset}{n}
	\gtrsim [\log(\densbound\datasize)]^{2/3} \cdot n^{-2/3}.
\]
\end{proof}

\begin{fact}\label{fact:tight-logn}
There exists an $\functionset$ for which the dependence on $\log(n)$ is tight. 
\end{fact}
\begin{proof}[Proof of \cref{fact:tight-logn}]
We again consider a collection of monotone densities, but this time in the regime where the support scales with $n$. Taking $\contextspace = \emptyset$ and $\datareldistn$ to be counting measure, we wish to again apply Theorem~3 of \citet{gao09monotone}.

To do so, let $\functionset_\datasize^{\text{Count}}$ be the set of non-increasing probability mass functions (i.e., densities with respect to counting measure) on $[\datasize]$ and $\functionset_\datasize^{\text{Leb}}$ be the set of non-increasing probability densities on $[0,\datasize]$ bounded by $1$. For any $\func\in\functionset_\datasize^{\text{Count}}$, define the piecewise-constant extension $\func^{\text{Leb}}: [0,\datasize] \to [0,1]$ by
\*[
	\func^{\text{Leb}}(\data) = \sum_{j=1}^\datasize \func(j) \ind{\data\in[j-1,j)}.
\]
Clearly, $\func^{\text{Leb}}$ is non-increasing and bounded by $1$, and
\*[
	\int_0^\datasize \func^{\text{Leb}}(\data) \dee \data
	= \sum_{j=1}^\datasize \int_{j-1}^j \func(j) \dee \data
	= \sum_{j=1}^\datasize \func(j)
	= 1,
\]
so $\func^{\text{Leb}} \in \functionset_\datasize^{\text{Leb}}$. Moreover, if $\func,\coverfunc \in \functionset_\datasize^{\text{Count}}$, then
\*[
	\distsym_{\hellsym}^2(\func^{\text{Leb}},\coverfunc^{\text{Leb}})
	&= \frac{1}{\sqrt{2}}\int_0^\datasize (\sqrt{\func^{\text{Leb}}(\data)} - \sqrt{\coverfunc^{\text{Leb}}(\data)})^2 \dee \data \\
	&= \frac{1}{\sqrt{2}}\sum_{j=1}^\datasize \int_{j-1}^j (\sqrt{\func(j)} - \sqrt{\coverfunc(j)})^2 \dee \data \\
	&= \frac{1}{\sqrt{2}} \sum_{j=1}^\datasize (\sqrt{\func(j)} - \sqrt{\coverfunc(j)})^2 \\
	&= \distsym_{\hellsym}^2(\func,\coverfunc).
\]
This connection between piecewise-constant densities and probability mass functions implies that $\genent{\hellsym}{2}(\functionset_\datasize^{\text{Count}}, 2\eps, n) \leq \genent{\hellsym}{2}(\functionset_\datasize^{\text{Leb}}, \eps, n)$.
Thus, the same argument as \cref{fact:tight-logBK} (taking $\densbound = 1$ and $\datasize = 2 n^{1/3}$) implies
\*[
	\minimaxrisk{\functionset_\datasize^{\text{Count}}}{n}
	\lesssim \Bigg(\frac{\log n}{n} \Bigg)^{2/3}.
\]
However, by Pinsker's and Theorem~2.1 of \citet{devroye19trees},
\*[
	\minimaxrisk{\functionset_\datasize^{\text{Count}}}{n}
	\gtrsim \Bigg(\frac{\log n}{n} \Bigg)^{2/3}.
\]
\end{proof}

Finally, for extremely rich classes with $\genent{\hellsym}{2}(\functionset, \eps,n)$ that grows exponentially in $\eps^{-1}$, our results imply that the risk decays polynomially in $(\log n)^{-1}$. A more refined analysis is needed to understand the precise role of the logarithmic factors, which become crucially important when all dependence on $n$ is logarithmic.}

%% file: section-files/local-complexity.tex
We now prove
\pref{fact:risk-upper-rates,fact:risk-lower-rates}. Before proceeding
to the proofs, we introduce some additional technical tools.
\subsection{Technical preliminaries}
\label{sec:technical-prelims}
\paragraph*{Hellinger distance}
We begin by stating two technical lemmas that allow one to relate the KL
divergence to the Hellinger distance, which are crucial to our
arguments. The first lemma shows that the two divergences are
equivalent up to logarithmic factors whenever the density ratio for
the arguments is bounded.
\begin{lemma}[Lemma~4 of \citet{yang98model}]
\label{fact:yang98}
For all $p,q \in \densityset(\dataspace)$,
\*[
  \distsym_{\klsym}^2(p, q)
  \leq 2 \bigg[2 + \sup_{y\in\dataspace}\log\left(\frac{p(\data)}{q(\data)}\right) \bigg] \distsym^2_{\hellsym}(p,q).
\]
\end{lemma}
The second lemma controls the approximation error incurred by the
smoothing used in our estimator (which controls the aforementioned
density ratio).
\begin{lemma}\label{fact:smoothed-hellinger}
For all $\smoothparam \geq 0$ and $p,q \in \densityset(\dataspace)$,
\*[
  \distsym^2_{\hellsym}(p, \smooth{\smoothparam}{q})
  \leq \distsym^2_{\hellsym}(p, q) + \smoothparam.
\]
\end{lemma}
\begin{proof}[Proof of \cref{fact:smoothed-hellinger}]
By convexity of squared Hellinger distance,
\*[
  \distsym^2_{\hellsym}(p, \smooth{\smoothparam}{q})
  &\leq \frac{1}{1+\smoothparam} \distsym^2_{\hellsym}(p, q) + \frac{\smoothparam}{1+\smoothparam} \distsym^2_{\hellsym}(p, 1/\datasize)
  \leq \distsym^2_{\hellsym}(p,q) + \smoothparam.
\]
\end{proof}

\paragraph*{Local Rademacher complexity}

An important step in the analysis for our upper bound is to derive
``fast'' (i.e., $1/n$-type for parametric classes) uniform concentration guarantees for Hellinger distance, which we obtain using local Rademacher complexities \citep{koltchinskii00local,bousquet02thesis,bartlett05localrademacher}.

We define the \emph{(empirical) Rademacher complexity} for a class $\dumfunctionset \subseteq \contextspace \to \Reals$ on a sample $\context_{1:n} \subseteq \contextspace$ as
\*[
	\rad_n(\dumfunctionset, \context_{1:n}) = \EE_{\sigma_{1:n}} \sup_{\dumfunc \in \dumfunctionset} \frac{1}{n} \sum_{t=1}^n \sigma_t \dumfunc(\context_t),
\]
where each $\sigma_t$ is an independent Rademacher random variable, taking values $1$ and $-1$ with equal probability.
The \emph{worst-case Rademacher complexity} over samples is given by
\*[
	\rad_n(\dumfunctionset) = \sup_{\context_{1:n}} \, \rad_n(\dumfunctionset, \context_{1:n}).
\]
The \emph{localization} of $\dumfunctionset$ at scale $r>0$ on a
sample $\context_{1:n}$ is defined as
\*[
  \dumfunctionset[r,\context_{1:n}] = \Big\{\dumfunc \in \dumfunctionset \Bigsetdelim \frac{1}{n} \sum_{t=1}^n \dumfunc(\context_t) \leq r \Big\},
\]
and the \emph{local Rademacher complexity} is 
\*[
	\rad_n(\dumfunctionset, r)
	= \sup_{\context_{1:n}} \, \rad_n(\dumfunctionset[r,\context_{1:n}], \context_{1:n}).
\]

For fixed $n\in\bbN$,
we define a \emph{\uppfuncname{}} for $\dumfunctionset$ as any function $\uppfunc_n:[0,\infty) \to
\Reals$ satisfying the following conditions:
\begin{enumerate}
\item $\uppfunc_n(r)$ is non-negative;
\item $\uppfunc_n(r)$ is non-decreasing;
\item $\uppfunc_n(r)$ is strictly positive for some $r$;
\item $\uppfunc_n(r) / \sqrt{r}$ is non-increasing;
\end{enumerate}
along with, for all $r \geq 0$,
\*[
  \rad_n(\dumfunctionset, r) \leq \uppfunc_n(r).
\]
A size-$n$ \emph{\uppradname{}} of $\dumfunctionset$ with respect to a \uppfuncname{} $\uppfunc_n$ is the largest
real number $\locradius{n} \geq 0$ satisfying $\locradius{n} =
\uppfunc_n(\locradius{n})$.

%% file: section-files/risk-upper.tex
\subsection{Proof for risk upper bound (\pref{fact:risk-upper-rates})}\label{sec:risk-upper}

The main result for this section is the following theorem, which is a
generalization of \pref{fact:risk-upper-rates} stated in terms of
fixed points for local Rademacher complexity.

\begin{theorem}[Quantitative version of \cref{fact:risk-upper-rates}]\label{fact:risk-upper}
For any $\functionset\subseteq \kernelset(\contextspace,\dataspace)$,
$n >1$,
$\eps>1/\sqrt{n}$, 
and \uppfuncname{} $\uppfunc_n$ for $\hellingerset = \{\context
\mapsto \distsym^2_{\hellsym}(\func(\context), \coverfunc(\context))
\setdelim \func,\coverfunc \in \functionset \}$,
the corresponding size-$n/2$ \uppradname{} $\locradius{n/2}$ satisfies
\*[ 
  &\sup_{\contextdistn}
  \sup_{\truefunc\in \functionset}
  \predrisk{\riskpredfunc{}}{\truefunc, \contextdistn}{n} \\
  &\leq \frac{2\genent{\hellsym}{2}(\functionset, \eps, n)}{n+2}
  +
  2\Big[2 + \log(n\densbound\datasize)\Big]
  \bigg[2\eps^2 + 106 \locradius{n/2} + \frac{96\log n + 576\log\log n + 4}{n} \bigg],
\]
where $\riskpredfunc{}$ is the estimator defined in \cref{sec:minimax-estimator}.
\end{theorem}

In order to derive \pref{fact:risk-upper-rates} from this result, it
remains to control $\locradius{n}$. This is achieved using the
following lemma.

\begin{lemma}\label{fact:loc-radius}
In each of the following three cases, one may choose a \uppfuncname{} $\uppfunc_n$ for $\hellingerset = \{\context \mapsto
\distsym^2_{\hellsym}(\func(\context), \coverfunc(\context)) \setdelim
\func,\coverfunc \in \functionset \}$ such that for all $n\geq{}1$, the corresponding
size-$n$ \uppradname{} $\locradius{n}$ satisfies:
\begin{enumerate}
\item[i)] For any $\functionset \subseteq \kernelset(\contextspace,\dataspace)$,
\*[
  \locradius{n} \leq 972 \, (\log 2n)^{3} \, \bigg(\inf_{\gamma > 0} \bigg\{4\gamma + \frac{17}{\sqrt{n}}\int_\gamma^1 \sqrt{\genent{\hellsym}{2}(\functionset,\rho/2,n)} \dee \rho \bigg\} \bigg)^2.
\]
\item[ii)] If there exist $\funcdim, \entropyconst > 0$ such that $\genent{\hellsym}{2}(\functionset, \eps, n) \leq \funcdim \log(\entropyconst/\eps)$, then for 
\preprint{$n > (12/\entropyconst)^{2}\funcdim$,}
\*[
  \locradius{n} \leq \frac{289\funcdim}{n} \log\bigg(\frac{4\entropyconst \sqrt{n}}{17\sqrt{\funcdim}} \bigg).
\]
\item[iii)] If $\abs{\functionset} < \infty$,
\*[
  \locradius{n} \leq \frac{289\log\abs{\functionset}}{n}.
\]
\end{enumerate}
\end{lemma}

We now prove \cref{fact:risk-upper}, deferring the proof of \cref{fact:loc-radius} to \cref{sec:rates-proofs}.

\begin{proof}[Proof of \cref{fact:risk-upper}]
Fix $\contextdistn$ and $\truefunc\in\functionset$.
First, we observe that conditioned on the (independent) sample $(\dumcontext_{1:n}, \dumdata_{1:n})$ used to
define the empirical Hellinger cover $\coverset$, Jensen's inequality gives
\[\label{eqn:risk-upper-1}
  &\hspace{-2em}
  \EE_{\context_{1:n}\sim\contextdistn}
  \EE_{\data_{1:n} \sim \truefunc(\context_{1:n})}
  \EE_{\context \sim \contextdistn} 
  \distsym_{\klsym}^2(\truefunc(\context), \riskpredfunc{n}(\context)) \\
  &= 
  \EE_{\context_{1:n}\sim\contextdistn}
  \EE_{\data_{1:n} \sim \truefunc(\context_{1:n})}
  \EE_{\context \sim \contextdistn} 
  \distsym^2_{\klsym}\Big(\truefunc(\context), \frac{1}{n+1}\sum_{t=0}^n \predfunc{t}(\context)\Big) \\
  &\leq
  \EE_{\context_{1:n}\sim\contextdistn}
  \EE_{\data_{1:n} \sim \truefunc(\context_{1:n})}
  \EE_{\context \sim \contextdistn} 
  \frac{1}{n+1}\sum_{t=0}^n \distsym^2_{\klsym}(\truefunc(\context), \predfunc{t}(\context)) \\
  &=
  \EE_{\context_{1:n+1}\sim\contextdistn}
  \EE_{\data_{1:n+1} \sim \truefunc(\context_{1:n+1})}
  \frac{1}{n+1} \log\left(\frac{\prod_{t=0}^n [\truefunc(\context_{t+1})](\data_{t+1})}{\prod_{t=0}^n [\predfunc{t}(\context_{t+1})](\data_{t+1})}\right).
\]

Now, observe that the denominator $\prod_{t=0}^n [\predfunc{t}(\context_{t+1})](\data_{t+1})$ is equal to
\[\label{eqn:risk-upper-2}
  &\hspace{-1em}\Bigg(\sum_{\coverfunc \in \coverset} \frac{[\smooth{\smoothparam}{\coverfunc}(\context_1)](\data_1)}{\abs{\coverset}}\Bigg)
  \left(\frac{\sum_{\coverfunc \in \coverset} \prod_{t=1}^{2} [\smooth{\smoothparam}{\coverfunc}(\context_t)](\data_t)}{\sum_{\coverfunc \in \coverset} [\smooth{\smoothparam}{\coverfunc}(\context_{1})](\data_{1})} \right)
  \cdots
  \left(\frac{\sum_{\coverfunc \in \coverset} \prod_{t=1}^{n+1} [\smooth{\smoothparam}{\coverfunc}(\context_t)](\data_t)}{\sum_{\coverfunc \in \coverset} \prod_{t=1}^{n} [\smooth{\smoothparam}{\coverfunc}(\context_t)](\data_t)} \right) \\
  &= \frac{1}{\abs{\coverset}} \sum_{\coverfunc \in \coverset} \prod_{t=1}^{n+1}[\smooth{\smoothparam}{\coverfunc}(\context_t)](\data_t)
  \geq \frac{1}{\abs{\coverset}} \prod_{t=0}^{n}[\smooth{\smoothparam}{\coverfunc}_{\truefunc}(\context_{t+1})](\data_{t+1}),
\]
where we define
\*[
  \coverfunc_{\truefunc} = \argmin_{\coverfunc \in \coverset} \frac{1}{n}\sum_{t=1}^n \distsym^2_{\hellsym}(\truefunc(\dumcontext_t), \coverfunc(\dumcontext_t)).
\]

Thus, since $\smooth{\smoothparam}{\coverfunc}_{\truefunc}$ only depends on $\dumcontext_{1:n}$, $\predrisk{\riskpredfunc{}}{\truefunc, \contextdistn}{n}$ is bounded above by
\*[
  &\EE_{\dumcontext_{1:n}\sim\contextdistn}
  \frac{\log\abs{\coverset}}{n+1}
  +
  \EE_{\dumcontext_{1:n}\sim\contextdistn}
  \EE_{\context_{1:n+1}\sim\contextdistn}
  \EE_{\data_{1:n+1} \sim \truefunc(\context_{1:n+1})}
  \frac{1}{n+1}\sum_{t=1}^{n+1} 
  \log\left(\frac{[\truefunc(\context_t)](\data_t)}{[\smooth{\smoothparam}{g}_{\truefunc}(\context_t)](\data_t)} \right) \\
  &\hspace{2em}\leq 
  \frac{\genent{\hellsym}{2}(\functionset, \eps, n)}{n+1} 
  +
  \EE_{\dumcontext_{1:n}\sim\contextdistn}
  \EE_{\context\sim\contextdistn}
  \distsym^2_{\klsym}\Big(\truefunc(\context), \smooth{\smoothparam}{g}_{\truefunc}(\context)\Big).
\]
Further, by \cref{fact:smoothed-hellinger,fact:yang98}, for all $\context\in\contextspace$
\*[
  \distsym^2_{\klsym}(\truefunc(\context), \smooth{\smoothparam}{g}_{\truefunc}(\context))
  &\leq 2\Big[2 + \log(2\densbound\datasize/\smoothparam)\Big] \Big[\distsym_{\hellsym}^2(\truefunc(\context), g_{\truefunc}(\context)) + \smoothparam\Big].
\]

We now make use of a uniform concentration guarantee for the Hellinger
distance, which follows by applying Theorem~6.1 of
\citet{bousquet02thesis} (a generic concentration guarantee based on
local Rademacher complexity) to
$\hellingerset$ and using that $\sup_{\hellfunc \in \hellingerset} \sup_{\context\in\contextspace} \hellfunc(\context) \leq 1$.
\begin{lemma}\label{fact:uniform-hellinger}
For any \uppfuncname{} $\uppfunc_n$ for $\hellingerset$ with corresponding \uppradname{} $\locradius{n}$,
with probability at least $1-\delta$
over $\dumcontext_{1:n}$, all $\func,\coverfunc \in
\functionset$ satisfy
\[
  \EE_{\context\sim\contextdistn} \distsym_{\hellsym}^2(\func(\context), \coverfunc(\context))
  \leq \frac{2}{n} \sum_{t=1}^n \distsym^2_{\hellsym}(\func(\dumcontext_t), \coverfunc(\dumcontext_t))
  + 106\locradius{n} + 48\frac{\log(1/\delta) + 6\log\log n}{n}.
\label{eq:localization}
\]
\end{lemma}

In particular, since 
\cref{fact:uniform-hellinger}
holds uniformly for all
$\func,\coverfunc\in\functionset$, it can be applied to $\func = \truefunc$ and $\coverfunc = \coverfunc_{\truefunc}$.
Converting the resulting guarantee to 
a moment
bound by setting $\delta =
1/n$ and using boundedness of
 $\hellingerset$, we find
\*[
  &\hspace{-2em}\EE_{\dumcontext_{1:n}\sim\contextdistn}
  \EE_{\context\sim\contextdistn}
  \distsym^2_{\hellsym}(\truefunc(\context), \coverfunc_{\truefunc}(\context))
  \leq 2\eps^2 + 106 \locradius{n} + 48 \frac{\log n + 6\log\log n}{n} + \frac{1}{n}.
\]

We obtain the desired theorem statement by taking $\smoothparam = 1/n$ and recalling that each $n$ actually
corresponds to ``$n/2$'', since we split the sample in half.
\end{proof}

%% file: section-files/risk-lower.tex
\subsection{Proof for risk lower bound (\pref{fact:risk-lower-rates})}\label{sec:risk-lower}

We now proceed to state and prove \pref{fact:risk-lower}, which
is a generalization of our main lower bound, \pref{fact:risk-lower-rates}.
Our general result is most directly stated in terms of
\emph{local empirical entropy}. Informally, the local entropy for a
given function class captures the complexity of the class when we
restrict to an $\veps$-ball (with respect to empirical Hellinger distance) around a given
target function.
\begin{definition}
A set $\locpackingset \subseteq
\functionset$ is said to be a \emph{local $(\distsym, q)$-packing of $\functionset$ at $\reffunc \in \functionset$ on $\context_{1:n} \subseteq \contextspace$ at scale $\eps$} if
\*[
  \inf_{\coverfuncdum \neq \coverfunc \in \locpackingset}
  \Bigg(\frac{1}{n}\sum_{t=1}^n\distsym^q(\coverfuncdum(\context_t), \coverfunc(\context_t))\Bigg)^{1/q} 
  \geq \eps/2
\]
and
\*[
  \sup_{\coverfunc \in \locpackingset}
  \Bigg(\frac{1}{n}\sum_{t=1}^n\distsym^q(\reffunc(\context_t), \coverfunc(\context_t))\Bigg)^{1/q} 
   \leq \eps.
\]
We denote the cardinality of the largest local packing
$\locpackingset$ by
$\locgenpack{\distsym}{q}{\reffunc}(\functionset, \eps,
\context_{1:n})$. The \emph{local empirical $(\distsym,q)$-entropy of $\functionset$ for samples of length $n$
at scale $\eps$} is then defined as
\*[
  \locgenpackent{\distsym}{q}(\functionset, \eps, n)
  = \sup_{\reffunc\in\functionset} \sup_{\context_{1:n} \subseteq \contextspace} \log \locgenpack{\distsym}{q}{\reffunc}(\functionset, \eps, \context_{1:n}).
\]
\end{definition}
Equipped with this definition, we state the main lower bound.
\begin{theorem}[Quantitative version of \cref{fact:risk-lower-rates}]\label{fact:risk-lower}
For any $\functionset$ and $n > 1$,
\*[
  \minimaxrisk{\functionset}{n}
  \geq \sup_{\eps > n^{-1/2} (\log n)^{-1}} \frac{\eps^2}{4} \Bigg[1 - \frac{12n[1 + \log(2\densbound\datasize n \log n)]\eps^2 + \log(2)}{\locgenpackent{\hellsym}{2}(\functionset,\eps,n)} \Bigg].
\]
\end{theorem}
\pref{fact:risk-lower-rates} readily follows by combining this result
with the following empirical variant of a result of
\citet[Lemma~3]{yang99information}, which allows us to pass from local to global entropy.
\begin{lemma}\label{fact:loc-entropy}
For all $q\geq1$, $\eps>0$, and $\context_{1:n} \subseteq \contextspace$,
\*[
  \genpackent{\distsym}{q}(\functionset, \eps/2, \context_{1:n})
  - \genpackent{\distsym}{q}(\functionset, \eps, \context_{1:n})
  \leq \locgenpackent{\distsym}{q}(\functionset, \eps, \context_{1:n})
  \leq \genpackent{\distsym}{q}(\functionset, \eps/2, \context_{1:n}).
\]
\end{lemma}
The remainder of this section is spent proving \pref{fact:risk-lower};
we defer the detailed derivation of \pref{fact:risk-lower-rates} to \pref{sec:rates-proofs}.

\begin{proof}[Proof of \cref{fact:risk-lower}]

Let $\packcontext_{1:n} \subseteq \contextspace$ witness a maximal
local packing for $\cF$.
That is, there exists
$\reffunc\in\functionset$ and $\locpackingset \subseteq \functionset$
such that $\locpackingset$ is a local $(\hellsym,2)$-packing of $\functionset$ at $\reffunc \in
\functionset$ on $\packcontext_{1:n} \subseteq \contextspace$ at scale
$\eps$, and $\log \abs[0]{\locpackingset} =
\locgenpackent{\hellsym}{2}(\functionset,\eps,n)$. (If the maximal packing size is not obtained, we can repeat the following argument with a limit sequence.)

Taking $\contextdistn$ to be uniform on the witness set, we lower bound the minimax risk using Pinsker's inequality and Markov's inequality. Specifically,
\*[
  \minimaxrisk{\functionset}{n}
  &=
  \inf_{\predfunc{}} 
  \sup_{\contextdistn} 
  \sup_{\truefunc \in\functionset} 
  \EE_{\context_{1:n} \sim \contextdistn}
  \EE_{\data_{1:n} \sim \truefunc(\context_{1:n})}
  \EE_{\context \sim \contextdistn}
  \distsym^{2}_{\klsym}(\truefunc(\context),\predfunc{n}(\context)) \\
  &\geq 
  \inf_{\predfunc{}} 
  \sup_{\truefunc \in \locpackingset} 
  \EE_{\context_{1:n} \sim \uniformdist(\packcontext_{1:n})} 
  \EE_{\data_{1:n} \sim \truefunc(\context_{1:n})}
  \EE_{\context \sim \uniformdist(\packcontext_{1:n})}
  \distsym^2_{\hellsym}(\truefunc(\context), \predfunc{n}(\context)) \\
  &\geq 
  \frac{\eps^{2}}{16}\cdot
  \inf_{\predfunc{}} 
  \sup_{\truefunc \in \locpackingset} 
  \PP_{\context_{1:n} \sim \uniformdist(\packcontext_{1:n}), \data_{1:n} \sim \truefunc(\context_{1:n})} 
  \left[ 
  \frac{1}{n}\sum_{t=1}^n
  \distsym^2_{\hellsym}(\truefunc(\packcontext_t), \predfunc{n}(\packcontext_t)) 
  \geq \frac{\eps^2}{16} \right].
\]

Fix an estimator $\predfunc{}$ and
define
\*[
  \coverfunc_{\predfunc{n}} = \argmin_{\coverfunc \in \locpackingset} \frac{1}{n}\sum_{t=1}^n \distsym_{\hellsym}^2(\coverfunc(\packcontext_t), \predfunc{n}(\packcontext_t)).
\]
We claim that, if the estimator $\predfunc{}$ has low risk,
then $\coverfunc_{\predfunc{n}}$ identifies the true conditional density $\truefunc \in \locpackingset$. Let $\truefunc \in \locpackingset$ be arbitrary, and 
\aos{consider the case where}
\preprint{consider when}
$\sqrt{\frac{1}{n}\sum_{t=1}^n
  \distsym_{\hellsym}^2(\truefunc(\packcontext_t),
  \predfunc{n}(\packcontext_t))} < \eps/4$ and $\truefunc \neq
\coverfunc_{\predfunc{n}}$. Then, by triangle inequality,
\*[
  \sqrt{\frac{1}{n}\smash[b]{\sum_{t=1}^n} \distsym_{\hellsym}^2(\truefunc(\packcontext_t), \coverfunc_{\predfunc{n}}(\packcontext_t))}
  &\leq \sqrt{\frac{1}{n}\smash[b]{\sum_{t=1}^n} \distsym_{\hellsym}^2(\truefunc(\packcontext_t), \predfunc{n}(\packcontext_t))} + \sqrt{\frac{1}{n}\sum_{t=1}^n \distsym_{\hellsym}^2(\predfunc{n}(\packcontext_t), \coverfunc_{\predfunc{n}}(\packcontext_t))} \\
  &\leq 2 \sqrt{\frac{1}{n}\smash[b]{\sum_{t=1}^n} \distsym_{\hellsym}^2(\truefunc(\packcontext_t), \predfunc{n}(\packcontext_t))} 
  < \eps/2.
\]
This inequality contradicts the assumption that $\locpackingset$ is a
local packing on $\packcontext_{1:n}$, since that requires $\sqrt{\frac{1}{n}\sum_{t=1}^n \distsym_{\hellsym}^2(\truefunc(\packcontext_t), \coverfunc_{\predfunc{n}}(\packcontext_t))} \geq \eps/2$. Thus, 
\*[
  \truefunc \neq \coverfunc_{\predfunc{n}}
  \quad \implies \quad
  \sqrt{\frac{1}{n}\smash[b]{\sum_{t=1}^n} \distsym_{\hellsym}^2(\truefunc(\packcontext_t), \predfunc{n}(\packcontext_t))} \geq \eps/4.
\]
Consequently, we have
\*[
  &\hspace{-2em}\sup_{\truefunc \in \locpackingset}  \PP_{\context_{1:n}\sim \uniformdist(\packcontext_{1:n}), \data_{1:n} \sim \truefunc(\context_{1:n})} 
  \left[ 
  \frac{1}{n}\sum_{t=1}^n
  \distsym^2_{\hellsym}(\truefunc(\packcontext_t), \predfunc{n}(\packcontext_t)) 
  \geq \eps^2/16 \right] \\
  &\geq 
  \sup_{\truefunc \in \locpackingset}  \PP_{\context_{1:n}\sim \uniformdist(\packcontext_{1:n}), \data_{1:n} \sim \truefunc(\context_{1:n})}\left[\truefunc \neq \coverfunc_{\predfunc{n}} 
  \right]  \\
  &\geq \PP_{\context_{1:n}\sim \uniformdist(\packcontext_{1:n}), \truefunc \sim \uniformdist(\locpackingset), \data_{1:n} \sim \truefunc(\context_{1:n})}\left[\truefunc \neq \coverfunc_{\predfunc{n}} \right].
\]

To proceed, let us recall some standard notation (c.f. Section~2 of \citet{cover91information}). 
For any random variables $X,Y$, the \emph{entropy} of $X$ is $\entropy(X) = -\EE_X \log \gendens(X)$, the \emph{conditional entropy} is $\entropy(X\setdelim Y) = -\EE_{X,Y} \log \gendens(X\setdelim Y)$, and the \emph{mutual information} is $\mutualinfo(X;Y) = \entropy(X) - \entropy(X \setdelim Y)$.
For another random variable $Z$, the \emph{conditional mutual information} is $\mutualinfo(X; Y \setdelim Z) = \entropy(X \setdelim Z) - \entropy(X \setdelim Y, Z)$.

Since $\context_{1:n}$ and $\truefunc$ are independent with discrete supports, the joint
density $\gendens(\truefunc, \context_{1:n}, \data_{1:n}) = \gendens(\data_{1:n} \setdelim \context_{1:n}, \truefunc) \gendens(\context_{1:n}) \gendens(\truefunc)$ is
well-defined.
Finally, note that if $\truefunc \sim \uniformdist(\locpackingset)$, $\context_{1:n} \sim \uniformdist(\packcontext_{1:n})$, and $\data_{1:n} \sim \truefunc(\context_{1:n})$, then
$\truefunc$ and $\coverfunc_{\predfunc{n}}$ are conditionally independent given $(\context_{1:n}, \data_{1:n})$.
Thus, applying Fano's inequality (e.g., Theorem~2.10.1 of \citep{cover91information}),
we have
\[\label{eqn:fano-bound} 
  &\hspace{-2em}\PP_{\context_{1:n} \sim \uniformdist(\packcontext_{1:n}), \truefunc \sim \uniformdist(\locpackingset), \data_{1:n} \sim \truefunc(\context_{1:n})}\left[\truefunc \neq \coverfunc_{\predfunc{n}}\right] \\
  &\geq  
  \frac{\entropy(\truefunc \setdelim \context_{1:n}, \data_{1:n}) - \log(2)}{\log\abs[0]{\locpackingset}}
  =1 - \frac{\mutualinfo(\truefunc; \context_{1:n}, \data_{1:n}) + \log(2)}{\log\abs[0]{\locpackingset}}.
\]

Next, we bound the mutual information $\mutualinfo(\truefunc; \context_{1:n}, \data_{1:n})$ using
\[\label{eq:lower-MI-bound1}
  &\hspace{-1em}\EE_{\context_{1:n} \sim \uniformdist(\packcontext_{1:n}), \truefunc \sim \uniformdist(\locpackingset), \data_{1:n} \sim \truefunc(\context_{1:n})} \log \frac{\gendens(\truefunc, \context_{1:n}, \data_{1:n})}{\gendens(\truefunc)\gendens(\context_{1:n}, \data_{1:n})} \\
  &=
  \EE_{\context_{1:n} \sim \uniformdist(\packcontext_{1:n}), \truefunc \sim \uniformdist(\locpackingset), \data_{1:n} \sim \truefunc(\context_{1:n})} \log \frac{\gendens(\data_{1:n} \setdelim \truefunc, \context_{1:n})}{\EE_{\coverfunc \sim \uniformdist(\locpackingset)} \gendens(\data_{1:n} \setdelim \coverfunc, \context_{1:n})} \\
  &=
  \EE_{\context_{1:n} \sim \uniformdist(\packcontext_{1:n})}
  \EE_{\truefunc \sim \uniformdist(\locpackingset)} \distsym^2_{\klsym}\left(\prod_{t=1}^n \truefunc(\context_t), \EE_{\coverfunc \sim \uniformdist(\locpackingset)} \prod_{t=1}^n \coverfunc(\context_t) \right) \\
  &\leq
  \EE_{\context_{1:n} \sim \uniformdist(\packcontext_{1:n})}
  \inf_{q\in\kernelset(\contextspace,\dataspace)}
  \EE_{\truefunc \sim \uniformdist(\locpackingset)} \distsym^2_{\klsym}\left(\prod_{t=1}^n \truefunc(\context_t), \prod_{t=1}^n q(\context_t) \right).
\]
The
last step holds because
for any $q\in\kernelset(\contextspace,\dataspace)$ and $\context_{1:n}$,
\*[
  &\hspace{-1em}\EE_{\truefunc \sim \uniformdist(\locpackingset)} \Bigg[\distsym^2_{\klsym}\left(\prod_{t=1}^n \truefunc(\context_t), \prod_{t=1}^n q(\context_t) \right)
  - \distsym^2_{\klsym}\left(\prod_{t=1}^n \truefunc(\context_t), \EE_{\coverfunc \sim \uniformdist(\locpackingset)} \prod_{t=1}^n \coverfunc(\context_t) \right)\Bigg] \\
  &= \EE_{\truefunc \sim \uniformdist(\locpackingset)} \EE_{\data_{1:n} \sim \truefunc(\context_{1:n})} \log \bigg(\frac{\EE_{\coverfunc \sim \uniformdist(\locpackingset)} \prod_{t=1}^n [\coverfunc(\context_t)](\data_t)}{\prod_{t=1}^n [q(\context_t)](\data_t)}\bigg) \\
  &= \distsym^2_{\klsym}\left(\EE_{\coverfunc \sim \uniformdist(\locpackingset)} \prod_{t=1}^n \coverfunc(\context_t), \prod_{t=1}^n q(\context_t) \right)
  \geq 0.
\]

Then, we bound the last line of \cref{eq:lower-MI-bound1} using chain rule for KL, Jensen's inequality, and $\inf$ bounded by $\sup$. That is, for any $\smoothparam > 0$,
\[\label{eq:lower-MI-bound2}
  &\hspace{-1em}\EE_{\context_{1:n} \sim \uniformdist(\packcontext_{1:n})}
  \inf_{q\in\kernelset(\contextspace,\dataspace)}
  \EE_{\truefunc \sim \uniformdist(\locpackingset)} \distsym^2_{\klsym}\left(\prod_{t=1}^n \truefunc(\context_t), \prod_{t=1}^n q(\context_t) \right) \\
  &\leq
  \inf_{q\in\kernelset(\contextspace,\dataspace)}
  \EE_{\truefunc \sim \uniformdist(\locpackingset)} 
  \EE_{\context_{1:n} \sim \uniformdist(\packcontext_{1:n})} 
  \sum_{t=1}^n \distsym^2_{\klsym}\left(\truefunc(\context_t), q(\context_t) \right) \\
  &\leq
  \sup_{\coverfunc, \coverfuncdumm \in \locpackingset}
  \sum_{t=1}^n
  \EE_{\context_{t} \sim \uniformdist(\packcontext_{1:n})} 
   \distsym^2_{\klsym}\left(\coverfunc(\context_t), \smooth{\smoothparam}{\coverfuncdumm}(\context_t) \right) \\
  &=
  \sup_{\coverfunc, \coverfuncdumm \in \locpackingset}
  \sum_{t=1}^n 
   \distsym^2_{\klsym}\left(\coverfunc(\packcontext_t), \smooth{\smoothparam}{\coverfuncdumm}(\packcontext_t) \right).
\]
Finally, recall that for any $\coverfunc, \coverfuncdumm \in \locpackingset$,
\[\label{eqn:loc-reffunc}
  \sqrt{\sum_{t=1}^n \distsym^2_{\hellsym}\left(\coverfunc(\packcontext_t), \coverfuncdumm(\packcontext_t) \right)} 
  \leq
  \sqrt{\sum_{t=1}^n\distsym^2_{\hellsym}\left(\coverfunc(\packcontext_t), \reffunc(\packcontext_t) \right)} + \sqrt{\sum_{t=1}^n \distsym^2_{\hellsym}\left(\reffunc(\packcontext_t), \coverfuncdumm(\packcontext_t) \right)}
  \leq 2\sqrt{n} \, \eps,
\]
where $\reffunc$ is the reference function defining the local packing $\locpackingset$.
Thus, applying \cref{eqn:loc-reffunc} to \cref{eq:lower-MI-bound2} along with \cref{fact:yang98,fact:smoothed-hellinger} gives
\*[
  \mutualinfo(\truefunc; \context_{1:n}, \data_{1:n})
  \leq 2n[2 + \log(2\densbound\datasize/\smoothparam)][2\eps^2 + \smoothparam].
\]
Substituting this bound back into \cref{eqn:fano-bound} with $\smoothparam = 1/[n(\log n)^2]$ gives the result.
\end{proof}

%% file: section-files/mle.tex
\section{On the performance of maximum likelihood}
\label{sec:mle}

A natural question is whether the estimator introduced in
\cref{sec:minimax-estimator} is required to attain minimax optimal
performance, or whether a simpler, more computationally efficient
estimator will suffice. In particular, the maximum likelihood estimator (MLE) is simpler both conceptually and computationally.
In this section, we highlight the following properties of the MLE for
conditional density estimation:
\begin{enumerate}
\item For sufficiently small classes---namely, those that satisfy
  the \emph{Donsker} property (c.f.
  \citet{vandervaart96weakconvergence}) with respect to Hellinger entropy---the MLE obtains the minimax optimal
  rate. This follows by extending the fundamental results of
  \citet{shen94sieve} and \citet{wong95inequalities}.
\item For richer, nonparametric classes, the MLE can fail to attain the
  minimax rate. This follows by extending classical results from joint
  density estimation \citep{birge93mce} and nonparametric regression
  \citep{shen97sieves}. 
  We also discuss \emph{sieve MLEs} \citep{geman82sieve}, arguing that they are no more efficient than our estimator and their optimality is unknown.
\end{enumerate}

Formally, the maximum likelihood estimator for a sample
$(\context_{1:n}, \data_{1:n})$ is defined as
\*[
	\mlefunc{n}
	= \argmax_{f \in \functionset} \sum_{t=1}^n \log\Big([f(\context_t)](\data_t) \Big).
\]
Our observations in this section also concern the
$\alpha$-smoothed (as defined in \cref{sec:minimax-estimator}) MLE, which we denote by
$\smoothmle{\smoothparam}_n$ for notational convenience.

\subsection{Upper bound on the risk of maximum likelihood}\label{sec:mle-upper}

We first give an upper bound on the performance of the
smoothed maximum likelihood estimator, with the proof deferred to \cref{sec:mle-upper-proofs}.
To do so, we require a slightly different notion of complexity, the \emph{expected bracketing entropy}, denoted by $\brackent{\distsym}{q}(\functionset, \eps, \contextdistn)$ (see \cref{sec:mle-upper-proofs}).

It is not hard to see that, in general, $\genent{\distsym}{q}(\functionset, \eps, \contextdistn) \lesssim \brackent{\distsym}{q}(\functionset, \eps, \contextdistn)$, and there are cases where the LHS is bounded by a constant yet the RHS tends to infinity.
Consequently, our main results in \cref{sec:main-thms} are generally tighter than the following upper bound.
However, in the case when the metric entropy and bracketing entropy agree \emph{and} the model class is sufficiently small, \cref{fact:mle} shows that the MLE is optimal.

\begin{theorem}[MLE upper bound]\label{fact:mle}
For all $\funcdim>0$,
\*[
  &\hspace{-0.5em}\predrisk{\smoothmle{1/n}}{\truefunc, \contextdistn}{n} \\
  &\quad\lesssim
  \, \log(n\densbound\datasize)
  \cdot
  \begin{cases}
  n^{-1/\funcdim}
  &\text{ if } \ \brackent{\hellsym}{2}(\functionset, \eps, \contextdistn) \lesssim
  \eps^{-\funcdim} \text{ for } \funcdim > 2 \\
  n^{-2/(2+\funcdim)}
  &\text{ if } \ \brackent{\hellsym}{2}(\functionset, \eps, \contextdistn) \lesssim
  \eps^{-\funcdim} \text{ for } \funcdim \leq 2 \\
  (\funcdim \log n)/n
  &\text{ if } \ \brackent{\hellsym}{2}(\functionset, \eps, \contextdistn) \lesssim \, \funcdim \log(1/\eps).
  \end{cases}
\]
\end{theorem}
Even for \emph{unconditional} density estimation, the tightest results for MLE estimation \citep{wong95inequalities} require conditions on bracketing entropy rather than metric entropy; resolving this gap remains an open problem.
Moreover, while it is possible (using \cref{fact:uniform-hellinger}) to substitute $\genent{\hellsym}{2}(\functionset, \eps, n)$ in the above theorem when the bracketing and metric entropies agree, it is unclear how to use a general \emph{empirical} bracketing entropy for all classes.
Finally, even in the advantageous setting when empirical Hellinger entropy can be used to characterize the rates, when this scales with
$\eps^{-\funcdim}$ for $\funcdim > 2$ the upper bound is significantly worse than the optimal rate of
$n^{-2/(2+\funcdim)}$.

\subsection{The suboptimality of maximum likelihood}\label{sec:mle-lower}

\input{section-files/mle-lower-quant}
\subsection{Sieve estimators}\label{sec:sieve-mle}

A traditional approach to sidestep limitations of the MLE is to use a
\emph{sieve} maximum likelihood estimator, which corresponds to the
MLE over a smaller subclass of the full model $\functionset$.
Although it is classically known that the sieve MLE can be consistent when the MLE is not (e.g., \citet{geman82sieve}),
we now briefly argue that it cannot obtain the minimax rates without sacrificing the computational efficiencies that make the MLE desirable (i.e., the standard MLE does not require construction of a cover).
Since the estimate output by the sieve MLE by definition belongs to the subclass used to approximate $\functionset$, the worst-case risk of the sieve MLE is lower bounded by the scale at which the subclass approximates $\functionset$ (in expected Hellinger distance).
Thus, to achieve the minimax risk, the approximation scale must be chosen at least as small as the minimax risk, so by \cref{fact:risk-lower-rates} the size of the subclass must be at least as large as the cover used by our estimator (described in \cref{sec:minimax-estimator}).
Since this requires computing an empirical cover to construct the subclass for the optimal sieve MLE, it is at least as computationally expensive as our estimator.

%

%
%
%
%

%
%
%
%

%% file: section-files/mle-lower-quant.tex
As mentioned above, \cref{fact:mle} provides a suboptimal upper bound for large classes.
This suboptimality  is not simply an artifact of the analysis, but rather an inherent limitation of the method. For example, Theorem~3 of \citet{birge93mce} gives an example in which the MLE is suboptimal for joint density estimation, while Example~3 of \citet{shen97sieves} gives examples in which the MLE is suboptimal for regression with square loss.
Recent results have also established the suboptimality of the (closely related) least squares estimator for square loss regression with Gaussian errors and convex function classes \citep{kur20convex}. Since conditional density estimation contains both joint density estimation (by setting the covariate space to a singleton) and conditional mean estimation (by setting the errors to be Gaussian) as special cases, these results immediately imply the suboptimality of the MLE for conditional density estimation.

We now present another instance of the suboptimality of the MLE for conditional density estimation. Our example concerns classification with binary responses and covariates, and the purpose is to highlight another natural setting in which the MLE is suboptimal. This construction differs qualitatively from previous results: compared to nonparametric regression the errors are non-Gaussian, and compared to joint density estimation we make non-trivial use of the covariate space. Readers familiar with the construction of \citet{birge93mce} will recognize the similarities with our analysis, which we defer to \cref{sec:mle-lower-proofs}.

\begin{theorem}[MLE lower bound]\label{fact:mle-lower-conditional}
For every $\holderparam \in (0,1/2)$, there exists a conditional density class $\convexmleclass$ that is convex, contains only $\holderparam$-H\"{o}lder continuous maps from $[-1/2,1/2]$ to $[7/16, 9/16]$, and has the following property: If the data-generating process is $\context_{1:n} \sim \uniformdist[-1/2,1/2]^{\otimes n}$ and $\data_{1:n} \sim \bernoullidist(1/2)^{\otimes n}$ independently for every $n \in \Nats$, then for every $\delta>0$, there exists $C>0$ such that
\*[
	\liminf_{n \to \infty} \PP_{\context_{1:n}, \data_{1:n}}\Big[\EE_{\context \sim \uniformdist[-1/2,1/2]} \distsym_{\hellsym}^2(1/2, \holdermle{n}(\context)) \geq C (n\log n)^{-\holderparam} \Big] \geq 1-\delta,
      \]
      where $\holdermle{n}$ denotes the maximum likelihood estimator for $\convexmleclass$.

      In particular, we have a well-specified model, yet the risk of the MLE over $\convexmleclass$ does not achieve the minimax optimal rate.
\end{theorem}

Let $\functionset_\holderparam$ denote the set of all $\holderparam$-H\"{o}lder continuous maps from $[-1/2,1/2]$ to $[7/16, 9/16]$. Then, since $\convexmleclass \subseteq \functionset_\holderparam$,\footnote{For any $p,q \in [7/16, 9/16]$, $(p-q)^2/8 \leq \distsym^2_{\hellsym}(p,q) \leq (p-q)^2$, so that
  $\genent{\hellsym}{2}(\functionset_{\holderparam}, \eps, n)
	\asymp \genent{2}{2}(\functionset_{\holderparam}, \eps, n)
	\asymp \eps^{-1/\holderparam}$,
where the last step follows from, e.g., Example 5.11 of \citet{wainwright19book}. } \cref{fact:risk-upper-rates} gives
\[\label{eq:holder-small-upper-rate}
	\minimaxrisk{\convexmleclass}{n} \leq \minimaxrisk{\functionset_{\holderparam}}{n} \lesssim n^{-\frac{2\holderparam}{2\holderparam+1}}.
\]
We see that for any $\holderparam\in(0,1/2)$, $n^{-\holderparam} > n^{-2\holderparam/(2\holderparam+1)}$, and hence the maximum likelihood estimator fails to achieve the minimax rate.

Note that in one dimension (of the covariate space), $\holderparam = 1/2$ marks the transition point at which $\convexmleclass$ ceases to be Donsker with respect to empirical Hellinger entropy. That is, we match the transition point for suboptimality of the MLE that has previously been observed for joint density estimation and square loss regression \citep{birge93mce,shen97sieves}.

We remark that the rate in \cref{eq:holder-small-upper-rate} does not match the rate we derive in \cref{fact:multinomial-smooth-example} (when $\funcdim=1$, $\numderivs=0$, and $\multidim=2$), as there is a meaningful difference between the set of H\"{o}lder continuous maps to all of $[0,1]$ and H\"{o}lder continuous maps with output bounded away from the boundary.
In particular, Hellinger distance behaves linearly rather than quadratically near the boundary of $[0,1]$, which affects the minimax rate.

%% file: section-files/literature.tex
%

%
%
%
\section{Related work}
\label{sec:literature}
There is a vast literature on both density estimation and nonparametric regression. 
The perspective of the present work is to combine (relatively) recent advances in nonparametric regression with classical density estimation techniques to derive minimax rates.
In this section, we discuss the most relevant work from both of these lines of research and connect their results to our own.
In addition, we discuss connections between our work and contemporary advances in density estimation.%

\subsection{Nonparametric density estimation}\label{sec:lit-estimation}

The use of metric entropy to study density estimation was pioneered by \citet{lecam73convergence}, who derived minimax rates for estimation of joint densities with general nonparametric classes using uniform Hellinger entropy; the rates here satisfy the relationship in \cref{eqn:entropy-relationship}. 
\citet{birge83estimation} used a similar approach to obtain minimax rates for bounded metrics on the parameter space, and presented (now classical) examples of the entropy for specific classes of interest. 
\citet{stone80nonparametric} and \citet{ibragimov83density} provided minimax rates for smooth nonparametric classes under the $\funcnorm{r}$ loss, and \citet{birge86hellinger} unified these results and extended them to the Hellinger loss using metric entropy.
More recently, entropy conditions for joint density estimation have been used to derive minimax rates for log-concave density estimation \citep{doss16logconcave,kim16logconcave,samworth18logconcave}.
In particular, \citet{kur19logconcave} and \citet{han19erm} showed that for this setting, maximum likelihood estimation can achieve optimal rates despite its known suboptimality for certain large classes.

Early lower bounds to complement these results include those of \citet{boyd78lower}, who showed that estimation under the $\funcnorm{r}$ loss cannot have a faster rate than $1/n$, and \citet{devroye83density}, who showed that for very rich classes (i.e., without regularity and smoothness assumptions) all estimators can be forced to experience an arbitrarily slow rate.
Information-theoretic lower bounds based on the fixed point in \cref{eqn:entropy-relationship} were given by \citet{birge83estimation,birge86hellinger}.
\citet{yang99information} observed that these lower bounds require construction of \emph{local} packings for each class, and instead present a method to derive lower bounds in terms of global entropy. %
They also provide a simple estimator that obtains the minimax rate for joint density estimation with uniform entropy, which inspired our estimator for conditional density estimation.

While some of the work above instantiates results for conditional density estimation as a special case, there is either the assumption that the conditional distribution follows a simple parametric form, thereby reducing density estimation to conditional mean estimation (e.g., Model~1 of \citep{stone80nonparametric}), or the assumption that the covariate distribution is known (e.g., Theorem~6 of \citep{yang99information}).
These examples consequently fall under the broader subject of nonparametric regression, which focuses on finding an estimate of the conditional mean $\regfunc(X) = \EE[Y \negsetdelim X]$ rather than the entire conditional density.

\subsection{Fast rates for nonparametric regression}\label{sec:lit-fast-rates}

The task of estimating the conditional mean (regression function) without specific distributional assumptions is well-studied, and comprehensive discussions are given by \citet{gyorfi02nonparametric} and \citet{tsybakov09nonparametric}. 
We wish to highlight a specific line of research into ``fast rates'' for nonparametric regression that provided significant inspiration for the present work. 

Under a sub-Gaussian assumption on the additive regression noise, \citet{vandegeer90regression} obtained upper bounds for estimation under square loss using empirical entropy conditions.
More generally,
for bounded, Lipschitz losses (e.g., square loss on a bounded domain), \citet{bartlett02complexity} provided tighter bounds for the minimax risk of regression using empirical Rademacher complexity \citep{koltchinskii01rademacher} 
rather than
the empirical entropy, leading to $n^{-1/2}$ bounds for Donsker classes and $n^{-1/\funcdim}$ bounds for larger classes.

Another important advance in nonparametric regression---which enables the tight rates achieved in the present work---is to move from global complexities to localized complexities by adapting modulus-of-continuity techniques (previously used to bound empirical processes \citep{vandervaart96weakconvergence}) to regression.
\citet{koltchinskii00local} and \citet{bousquet02local} used this strategy to obtain generalization bounds (and consequently risk bounds for empirical risk minimization) of rate $1/n$ whenever the regression class contains a function that achieves zero empirical error.
\citet{bartlett05localrademacher} extended these results to noisy regression, where the minimum achievable risk may always be strictly positive.
Under a convexity assumption on the regression class, their approach obtains the minimax rate of $n^{-2/(\funcdim+2)}$ for regression with Donsker classes under square loss; see Theorem~5.1 of \citet{koltchinskii11erm} for more detail.
\citet{rakhlin17minimax} extended this line of inquiry, using local Rademacher complexities to obtain matching upper and lower bounds for \emph{misspecified} regression under (bounded) square loss for general function classes.
Notably, their optimal estimator combines aggregation (like our procedure) with empirical risk minimization (the analogue of maximum likelihood for square loss), and they show that neither procedure can be optimal for all classes on its own in the misspecified setting.

\subsection{Sequential prediction under logarithmic loss}\label{sec:lit-sequential}

Conditional density estimation is closely related to the problem of sequential prediction under the logarithmic loss. In this setting, the objective is the same, but the statistician must make predictions one-by-one for individual sequences of data rather than \iid{} samples. For binary responses and parametric conditional density classes, \citet{rissanen84coding}, \citet{cover91universal}, and \citet{barron98coding} obtained nearly exact (including constants) minimax rates.
For richer nonparametric classes, \citet{opper99logloss} and \citet{cesabianchi99logloss} obtained upper bounds on the minimax rates using 
uniform metric entropies.
However, these rates are tight only under strong assumptions on the underlying data generating process (e.g., when densities are bounded uniformly away from zero). 
As we have discussed for conditional density estimation, the correct notion of complexity in the sequential setting must be \emph{empirical} in nature to avoid penalizing the estimator for the complexity of the entire covariate space simultaneously.
For this reason, \citet{rakhlin15binary} and \citet{foster18logistic} use a sequential notion of empirical entropy, introduced by \citet{rakhlin10combinatorial}, to obtain upper bounds on minimax performance. 
\citet{bilodeau20logloss} refined these arguments to obtain matching upper and lower bounds in terms of the empirical $\funcnorm{\infty}$ entropy. However, their bounds are not tight for \emph{all} possible function classes, which implies that the $\funcnorm{\infty}$ metric is insufficient to completely characterize the minimax rates for logarithmic loss.
\citet{bilodeau20logloss} highlighted the setting of \cref{fact:linear-example} as an instance for which their rates are not tight; in contrast, our use of the Hellinger distance rather than $\funcnorm{\infty}$ allows us to obtain matching upper and lower bounds for this setting.

\subsection{Contemporary results on density estimation}\label{sec:lit-modern-density}
While our results mainly build on the classical nonparametric density estimation literature, density estimation is still an active field of study; we discuss a few of the most relevant recent works here.
For conditional density estimation, estimation error has generally only been studied for specific cases. 
For the $\funcnorm{1}$ and $\funcnorm{2}$ losses, this includes kernel-based estimates for certain smoothness classes \citep{hall04crossvalidation,efromovich07conditional,gyorfi07conditional,holmes07conditional,bott17conditional,li21holder}, while for Hellinger risk, this includes exponential family regression \citep{baraud20robust} and the countable model selection problem for smoothness classes with the assumption that the conditional densities are lower bounded \citep{birge13robust,sart17conditional}.

For KL risk, a natural question is whether our results can easily be extended to the misspecified setting. This appears nontrivial due to the difficulty of controlling likelihood ratios when the measure generating the data does not belong to $\functionset$, but recent work makes progress toward this goal under additional assumptions. \citet{zhang06information} showed that under a Bernstein condition on the tails of logarithmic loss, 
the Gibbs posterior (maximum likelihood regularized with the KL divergence) 
achieves the optimal rates in terms of uniform entropy; for the (bounded) Hellinger risk, no tail conditions are necessary.
\citet{cohen11conditional} studied penalized maximum likelihood for conditional densities, but only bounded the (nontrivially smaller) Jenson-Kullback-Leibler divergence.
\citet{grunwald20fastrates} unified a significant line of work \citep{grunwald11safe,grunwald12learning,grunwald19unified} to provide a generalization of the Bernstein condition that leads to risk bounds for the Gibbs posterior in misspecified density estimation. Finally, \citet{foster18logistic} exploited properties of \emph{improper estimators} for conditional density estimation with the logistic loss. \citet{mourtada19smp} extended this work to remove $\log(n)$ factors for misspecified conditional density estimation in the parametric setting and obtain improved rates for misspecified Gaussian linear regression. 

All existing results for the logarithmic loss or KL risk (and some recent results for Hellinger risk) require tail conditions that hold uniformly over the density class, which immediately rules out a number of basic settings (e.g., univariate, compact response spaces in which densities are arbitrarily close to zero). Such tail conditions are not required to obtain non-trivial worst-case rates: as \citet{cesabianchi99logloss} noted and exploited, one may always artificially truncate the class of densities to make the logarithmic loss Lipschitz and bounded, then pay an additional penalty for the approximation error incurred by truncation. However, truncation is known to lead to suboptimal rates \citep{rakhlin15binary}, so it remains an open problem to characterize the optimal rates for misspecified (joint or conditional) density estimation without extra tail conditions.

%% file: section-files/conclusion.tex
\section{Conclusion}\label{sec:discussion}

We have characterized the minimax rates for conditional density estimation in terms of empirical Hellinger entropy for the conditional density class. Our results show that the classical fixed point relationship in \cref{eqn:entropy-relationship} governs the minimax rates for conditional density estimation, provided that one adopts empirical Hellinger entropy as the notion of complexity. Importantly, the empirical Hellinger entropy captures the correct dependence on the complexity of the covariate space, and leads to optimal rates, even for high-dimensional, potentially unbounded, covariate spaces. We leverage new techniques in both density estimation and nonparametric regression to make these new advances.

A natural direction for future work is to extend our results beyond the well-specified setting. While recent work \citep{mourtada19smp,grunwald20fastrates} has made progress on providing (excess) risk bounds in the presence of misspecification, it remains an open problem to characterize the minimax rates for conditional density estimation with misspecified models without tail assumptions on the densities.

%% file: acknowledgements.tex
\section*{Acknowledgements}\label{sec:acknowledgements}
BB acknowledges support from an NSERC Canada Graduate Scholarship and the Vector Institute. DMR is supported in part by an NSERC Discovery Grant and an Ontario Early Researcher Award.
This material is based also upon work supported by the United States Air Force under Contract No. FA850-19-C-0511. Any opinions, findings and conclusions or recommendations expressed in this material are those of the author(s) and do not necessarily reflect the views of the United States Air Force.
The authors thank Jeffrey Negrea, Yanbo Tang, and Yuhong Yang for helpful discussions and comments, and Abhishek Shetty for pointing out that Theorem~5 should use bracketing entropy rather than metric entropy.

%% file: section-files/rates-proofs.tex
\section{Proofs for specific rates}\label{sec:rates-proofs}

First, we prove \pref{fact:loc-radius}, which controls the \uppradname{} that appears in \cref{fact:risk-upper}.
The arguments proceed similarly to those in Lemma~8 of \citet{rakhlin17minimax}, with square loss replaced by Hellinger distance.

\begin{proof}[Proof of \cref{fact:loc-radius}]
Notice that when $\abs{\functionset}=1$, the result holds trivially with $\locradius{n}=0$. Thus, for the remainder of the proof we suppose $\abs{\functionset} \geq 2$.

First, we prove point i).
To start, we apply Lemma~2.2 of \citet{srebro10fastrates} to (in the
language of their lemma statement) the ``$2$-smooth'' loss $\smoothloss(t,y) = t^2$ composed with the class $\sqrt{\hellingerset}$. 
This gives, for all $r>0$, 
\*[
  \rad_n(\hellingerset,r)
  \leq 18 \sqrt{24 r} \, \rad_n(\sqrt{\hellingerset}) \, \log^{3/2}\Bigg(\frac{n \sup_{\hellfunc \in \sqrt{\hellingerset}}
\sup_{\context\in\contextspace} \hellfunc(\context)}{\rad_n(\sqrt{\hellingerset})} \Bigg).
\]
Further, it can be easily seen (using, for example, Khintchine's
inequality; see Lemma~A.9 of \citet{plg06book}) that
\*[
  \rad_n(\sqrt{\hellingerset})
  &\geq \sup_{\hellfunc \in \sqrt{\hellingerset}}
\sup_{\context\in\contextspace} \hellfunc(\context) /\sqrt{2n}.
\]
Thus,
\*[
  \rad_n(\hellingerset,r)
  \leq 9 \sqrt{12 r} \, \rad_n(\sqrt{\hellingerset}) \, \log^{3/2}(2n).
\]
Clearly, the \rhs satisfies the properties of a \uppfuncname{}
(scaling as $\sqrt{r}$), so we can take it as our
choice for $\uppfunc_n(r)$ and obtain
\*[
  \locradius{n} = 972 \, (\log 2n)^3 \, \rad_n^2(\sqrt{\hellingerset}).
\]
To proceed, recall that by the usual Dudley integral argument (e.g., Lemma~10 of \citet{rakhlin17minimax}),
\*[
  \rad_n(\sqrt{\hellingerset})
  \leq \inf_{\gamma > 0} \left\{4\gamma + \frac{12}{\sqrt{n}}\int_\gamma^1 \sqrt{\genent{1}{2}(\sqrt{\hellingerset},\rho,n) \dee \rho} \right\}.
\]
Next, using the triangle inequality, for any $\func,\coverfunc,\func',\coverfunc'\in\functionset$,
\*[
  &\hspace{-2em}\sqrt{\frac{1}{n}\sum_{t=1}^n \Big(\distsym_{\hellsym}(\func(\context_t), \coverfunc(\context_t)) - \distsym_{\hellsym}(\func'(\context_t), \coverfunc'(\context_t)) \Big)^2} \\
  &\leq 
  \sqrt{\frac{1}{n}\sum_{t=1}^n \Big(\distsym_{\hellsym}(\func(\context_t), \func'(\context_t)) + \distsym_{\hellsym}(\coverfunc(\context_t), \coverfunc'(\context_t)) \Big)^2} \\
  &\leq
  \sqrt{\frac{1}{n}\sum_{t=1}^n \distsym^2_{\hellsym}(\func(\context_t), \func'(\context_t))}
  + 
  \sqrt{\frac{1}{n}\sum_{t=1}^n \distsym^2_{\hellsym}(\coverfunc(\context_t), \coverfunc'(\context_t))}.
\]
That is, for all $\rho>0$ and $n$
\*[
  \gencov{1}{2}(\sqrt{\hellingerset},\rho,n)
  \leq \gencov{\hellsym}{2}^2(\functionset,\rho/2,n),
\]
so that
\*[
  \rad_n(\sqrt{\hellingerset})
  \leq \inf_{\gamma > 0} \left\{4\gamma + \frac{17}{\sqrt{n}}\int_\gamma^1 \sqrt{\genent{\hellsym}{2}(\functionset,\rho/2,n)} \dee \rho \right\}.
\]

Next, we prove point ii). First, since $\distsym_{\hellsym} \in [0,1]$, for every $\context_{1:n} \subseteq \contextspace$, 
\*[
  \sup_{h \in \hellingerset[r, \context_{1:n}]} \frac{1}{n} \sum_{t=1}^n h^2(\context_t)
  = \sup_{\func,\coverfunc\in\functionset: \frac{1}{n}\sum_{t=1}^n \distsym^2_{\hellsym}(\func(\context_t), \coverfunc(\context_t)) \leq r} 
  \frac{1}{n}\sum_{t=1}^n \distsym^4_{\hellsym}(\func(\context_t), \coverfunc(\context_t))
  \leq r.
\]
Further, for any $x,y \in [0,1]$,
\*[
  (x^2 - y^2)^2
  = \Big((x-y)(x+y) \Big)^2
  \leq 4(x-y)^2,
\]
so that $\gencov{1}{2}(\hellingerset,\rho,n) \leq
\gencov{1}{2}(\sqrt{\hellingerset},\rho/2,n)$ for all $\rho>0$. 
Thus,
by Lemma~A.3 of \citet{srebro10fastrates} (taking $\alpha=0$ in their
lemma statement), for any $r>0$,
\*[
  \rad_n(\hellingerset,r)
  &\leq \frac{12}{\sqrt{n}} \int_0^{\sqrt{r}} \sqrt{\genent{1}{2}(\hellingerset,\rho,n)} \dee \rho \\
  &\leq \frac{17}{\sqrt{n}} \int_0^{\sqrt{r}} \sqrt{\genent{\hellsym}{2}(\functionset,\rho/4,n)} \dee \rho \\
  &\leq \frac{17}{\sqrt{n}} \int_0^{\sqrt{r} \wedge 4\entropyconst} \sqrt{\funcdim\log(4\entropyconst/\rho)} \dee \rho \\
  &= 68 \entropyconst \sqrt{\frac{\funcdim}{n}}  \int_0^{(\sqrt{r}/4\entropyconst) \wedge 1} \sqrt{\log(1/\rho)} \dee \rho.
\]

Using integration by parts, we obtain (letting $\erf$ denote the errror function, which is non-negative on the positive reals):
\*[
  \rad_n(\hellingerset, r)
  &\leq 68 \entropyconst \sqrt{\frac{\funcdim}{n}} 
  \bigg[\bigg(\frac{\sqrt{r}}{4\entropyconst} \wedge 1 \bigg) \sqrt{-\log\bigg(\frac{\sqrt{r}}{4\entropyconst} \wedge 1\bigg)} - \frac{1}{2} \sqrt{\pi} \erf\Big(\sqrt{\log(1/\rho)} \Big) \Big\rvert_0^{(\sqrt{r}/4\entropyconst) \wedge 1} \bigg] \\
  &\leq 17 \sqrt{\frac{r \funcdim}{n}\log\Bigg(\frac{4\entropyconst}{\sqrt{r}}\Bigg)} \ \ind{r \leq 16\entropyconst^2}.
\]
To solve for the \uppradname{}, we need to find $r \in (0, 16\entropyconst^2)$ such that
\*[
	1 = \frac{289 \funcdim}{n r} \log\Bigg(\frac{4\entropyconst}{\sqrt{r}}\Bigg) \rdef f_n(r).
\]

Observe that $f_n(r)$ is decreasing for $r>0$, so any $r$ such that
$f_n(r) \leq 1$ gives an upper bound on $\locradius{n}$. To this end,
note that for any $Z > e$,
\*[
	f_n\bigg(\frac{289 \funcdim}{n} \log(Z) \bigg)
	= \frac{ \log\Big(\frac{4\entropyconst \sqrt{n}}{17\sqrt{\funcdim}} \Big) - \frac{1}{2}\log \log (Z)}{\log(Z)}
	\leq \frac{ \log\Big(\frac{4\entropyconst \sqrt{n}}{17\sqrt{\funcdim}} \Big)}{\log(Z)}.
\]
Hence, taking $Z = (4\entropyconst \sqrt{n})/(17\sqrt{\funcdim})$ and
gives the result (whenever this quantity is larger than $e$).

Finally, for the proof of point iii), we use the same initial steps as
ii) to show that
\*[
	\rad_n(\hellingerset,r)
 	&\leq \frac{12}{\sqrt{n}} \int_0^{\sqrt{r}} \sqrt{\genent{1}{2}(\hellingerset,\rho,n)} \dee \rho 
 	&\leq  \frac{17\sqrt{r}}{\sqrt{n}} \sqrt{\log\abs{\functionset}},
\]
so that
\*[
	\locradius{n} \leq \frac{289 \log\abs{\functionset}}{n}.
\]

\end{proof}

Equipped with this lemma, we can now prove \cref{fact:risk-upper-rates}.

\begin{proof}[Proof of \cref{fact:risk-upper-rates}]
For the cases $\genent{\hellsym}{2}(\functionset, \eps, n) \lesssim \funcdim\log(\entropyconst/\eps)$ and $\abs{\functionset}<\infty$, \pref{fact:risk-upper-rates} follows directly from combining \cref{fact:risk-upper,fact:loc-radius} and taking $\eps = 1/\sqrt{n}$.
The remainder of this proof concerns the case where
$\genent{\hellsym}{2}(\functionset, \eps, n) \lesssim
\entropyconst \eps^{-\funcdim}$. Let $\funcdim > 0$ be fixed, and consider the bound
on $\locradius{n}$ from \cref{fact:loc-radius}, part i):
\*[
	\locradius{n}
	\lesssim (\log n)^3 \, \inf_{\gamma > 0} \left(\gamma + \frac{1}{\sqrt{n}}\int_\gamma^1 \sqrt{\genent{\hellsym}{2}(\functionset,\rho/2,n)} \dee \rho \right)^2.
\]
For $\funcdim \neq 2$, we calculate
\*[
	&\hspace{-2em}\inf_{\gamma > 0} \left\{\gamma + \frac{1}{\sqrt{n}}\int_\gamma^1 \sqrt{\genent{\hellsym}{2}(\functionset,\rho/\sqrt{2},n)} \dee \rho \right\} \\
	&\lesssim \inf_{\gamma > 0} \left\{\gamma + \sqrt{\frac{\entropyconst}{n}}\int_\gamma^1 \rho^{-\funcdim/2}  \dee \rho \right\} \\
	&= \inf_{\gamma > 0} \left\{\gamma + \sqrt{\frac{\entropyconst}{n}}\bigg(\frac{2}{2-\funcdim} \bigg) \Big[1 - \gamma^{\frac{2-\funcdim}{2}} \Big] \right\}. 
\]
For $\funcdim < 2$, we may take $\gamma = 1/\sqrt{n}$, and observe
that for all $n\geq{}1$,
\*[	
	\frac{2}{2-\funcdim} \Big(1 - n^{\frac{\funcdim-2}{4}} \Big)
	\leq \frac{1}{2}\log n.
\] 
For $\funcdim > 2$, we may simply take $\gamma = n^{-1/\funcdim}$. Finally, in the case where $\funcdim = 2$, we have
\*[
	&\hspace{-2em}\inf_{\gamma > 0} \left\{\gamma + \frac{1}{\sqrt{n}}\int_\gamma^1 \sqrt{\genent{\hellsym}{2}(\functionset,\rho/\sqrt{2},n)} \dee \rho \right\} \\
	&\lesssim \inf_{\gamma > 0} \left\{\gamma + \sqrt{\frac{\entropyconst}{n}}\int_\gamma^1 \rho^{-1}  \dee \rho \right\} \\
	&= \inf_{\gamma > 0} \left\{\gamma - \sqrt{\frac{\entropyconst}{n}} \cdot\log\gamma \right\}, 
\]
so it suffices to choose $\gamma = 1/\sqrt{n}$.
Combining these cases, we have
\[\label{eq:loc-radius-bound}
	\locradius{n}
	\lesssim \entropyconst \cdot (\log n)^3 \cdot
	\begin{cases}
	(\log n)^2/n & \funcdim \leq 2 \\
	n^{-2/\funcdim} & \funcdim > 2.
	\end{cases}
\]

It remains to control the following quantity appearing in \cref{fact:risk-upper} (ignoring universal constants)
\*[
  \frac{\entropyconst \eps^{-\funcdim}}{n} + \log(n \densbound \datasize) \eps^2,
\]
which we approximately minimize by taking 
\*[
	\eps = \Big(\frac{\entropyconst}{n\log(n\densbound\datasize)}\Big)^{\frac{1}{\funcdim+2}}.
\]

Combining these choices of $\eps$ and $\gamma$ gives
\*[
	\minimaxrisk{\functionset}{n}
	\lesssim 
	\entropyconst^{\frac{2}{\funcdim+2}} 
	\cdot
	[\log(n\densbound\datasize)]^{\frac{\funcdim}{\funcdim+2}} 
	\cdot
	n^{-\frac{2}{\funcdim+2}}
	+ 
	\entropyconst \cdot \log(n\densbound\datasize) \cdot
	(\log n)^3 \cdot
	\begin{cases}
	(\log n)^2/n & \funcdim \leq 2 \\
	n^{-2/\funcdim} & \funcdim > 2.
	\end{cases}
\]
To simplify this result, our conditions on $\entropyconst$ imply that when $\funcdim \leq 2$, $\entropyconst/n < n^{-\frac{2}{\funcdim+2}}$, and for $\funcdim > 2$, $\entropyconst \cdot n^{-\frac{2}{\funcdim}} < n^{-\frac{2}{\funcdim+2}}$, which means that the additional polylogarithmic factors only appear in lower-order terms.
\end{proof}

We conclude this section of the appendix by proving \cref{fact:risk-lower-rates}.

\begin{proof}[Proof of \cref{fact:risk-lower-rates}]
We begin from the statement of \cref{fact:risk-lower}. Before
proceeding, let us remark that we
make no effort to optimize the numerical constants in this analysis.%

First, we consider the case where $\genent{\hellsym}{2}(\functionset, \eps, n) = \entropyconst \eps^{-\funcdim}$ for some $\funcdim > 0$ and $\entropyconst > 0$. By \cref{fact:loc-entropy}, 
\*[
	\locgenpackent{\hellsym}{2}(\functionset, \eps, n)
	\geq \entropyconst (2^\funcdim - 1) \eps^{-\funcdim}.
\]

Thus, for any $n>2$, the negative term from the \rhs of \cref{fact:risk-lower} can be bounded as
\[\label{eqn:lower-subtract-lesssim}
	\frac{12n[1 + \log(2\densbound\datasize n \log n)]\eps^2 + \log(2)}{\locgenpackent{\hellsym}{2}(\functionset,\eps,n)}
	\leq
	\frac{24n\eps^2\log(2\densbound \datasize n \log n) + 1}{\entropyconst (2^\funcdim - 1)\eps^{-\funcdim}}.
\]
By choosing
\*[
	\eps = \bigg(\frac{192n\log(2n\densbound\datasize)}{\entropyconst (2^\funcdim - 1)} \bigg)^{-\frac{1}{2+\funcdim}}
\]
and using that $2\densbound \datasize n \log n < (2\densbound
\datasize n)^2$, we may further upper bound the \rhs of \cref{eqn:lower-subtract-lesssim} by
\*[
	\frac{24n\log(2\densbound\datasize n \log n)}{\entropyconst (2^\funcdim - 1)} \bigg(\frac{\entropyconst (2^\funcdim - 1)}{192n \log(2n \densbound \datasize)} \bigg) + \frac{1}{\entropyconst (2^\funcdim - 1)}\bigg(\frac{\entropyconst (2^\funcdim - 1)}{192n\log(2n\densbound\datasize)} \bigg)^{\frac{\funcdim}{2+\funcdim}}
	\leq 1/2,
\]
as long as
\*[
	n \geq \frac{4^{\frac{2+\funcdim}{\funcdim}}}{192[\entropyconst(2^\funcdim-1)]^{2/\funcdim}}.
\]
That is, for any $n$ satisfying this condition and $\eps >
n^{-1/2} (\log n)^{-1}$ (which holds for all $n$ sufficiently large by
our choice of $\eps$),
\*[
	\minimaxrisk{\functionset}{n}
	\geq \eps^2/8
	= \frac{1}{8} \Bigg(\frac{192 \cdot n \cdot \log(2n\densbound\datasize)}{\entropyconst (2^\funcdim-1)} \Bigg)^{-\frac{2}{2+\funcdim}}.
\]

For the next case, suppose that $\genent{\hellsym}{2}(\functionset, \eps, n) = \funcdim \log(\entropyconst/\eps)$ for some $\entropyconst,\funcdim > 0$. By \cref{fact:loc-entropy}, 
\*[
	\locgenpackent{\hellsym}{2}(\functionset, \eps, n)
	\geq \funcdim \log 2.
\]
By a similar argument to the previous case, 
taking 
\*[
	\eps = \sqrt{\frac{\funcdim\log2}{192n\log(2\densbound\datasize n)}}
\]
gives that for $n$ sufficiently large and $\funcdim > 7$,
\*[
	\minimaxrisk{\functionset}{n}
	\geq \eps^2/8
	= \frac{\funcdim}{1536 \cdot n \log(2\densbound\datasize n)}.
\]

\end{proof}

%% file: section-files/examples-proofs.tex
\section{Proofs for examples}
\label{app:examples-proofs}
\begin{proof}[Proof of \cref{fact:linear-example}]$ $

\emph{Gaussian mean with linear link.} 
Let $p_\theta$ denote the density of $\normaldist(\theta,\sigma^2)$. 
For any $\theta,\theta' \in \Reals$,
\*[
  \distsym^2_{\hellsym}(p_{\theta}, p_{\theta'})
  &= 1 - \exp\Bigg\{-\frac{(\theta-\theta')^2}{8\sigma^2}\Bigg\}
  \leq \frac{(\theta-\theta')^2}{8\sigma^2}.
\]
Thus, by Corollary~3 of \citet{zhang02covering},
\*[
  \genent{\hellsym}{2}(\functionset_\weightspace, \eps, n)
  \leq \genent{2}{2}(\weightspace, \sqrt{8} \cdot \eps\sigma, n)
  \lesssim \frac{\xbound^2 \wbound^2}{\eps^2 \sigma^2} \log(n).
\]
Using $\datareldistn(\data) = [\pi(1+\data^2)]^{-1}$ (i.e., a $t$-distribution with 1 degree of freedom), we obtain $\datasize = 1$ and $\densbound = \sqrt{\pi} \cdot (1+\xbound^2\wbound^2)/\sqrt{2\sigma^2}$.
So, by \cref{fact:risk-upper-rates},
\*[
  \minimaxrisk{\functionset_{\weightspace}}{n}
  \lesssim
  n^{-1/2} \cdot
  \frac{\xbound\wbound\cdot \sqrt{\log n}\cdot \sqrt{\log(n \xbound^2 \wbound^2 / \sigma)}}{\sigma}.
\]

Further, if $(\theta-\theta')^2 \leq 8\sigma$,
\*[
  \distsym^2_{\hellsym}(p_{\theta}, p_{\theta'})
  \geq \frac{(\theta-\theta')^2}{16\sigma^2}.
\]
That is, for sufficiently small $\eps$,
\*[
  \genent{\hellsym}{2}(\functionset_\weightspace, \eps, n)
  \geq \genent{2}{2}(\weightspace, 4\eps\sigma, n).
\]

Finally, by combining a shattering argument of \citet{mendelson04linear} (specifically, Part 1 of Theorem~4.11 combined with Lemma~4.5) and Theorem~2 of \citet{bartlett97shattering}, it follows that
\*[
  \genent{2}{2}(\weightspace, \eps, n)
  \gtrsim \eps^{-2},
\]
so the polynomial dependence on $n$ is tight.

\emph{Poisson with logarithmic link.}
Let $p_\theta$ denote the density of $\poissondist(e^\theta)$.
By, for example, Eq.~(15) of \citet{nielsen20cumulant} and Table~1 of \citet{rigollet12glm}, for any $\theta,\theta' \in \Reals$
\*[
  \distsym_{\hellsym}^2(p_\theta,p_{\theta'})
  = 1 - \exp\Big\{\exp\Big((\theta+\theta')/2)\Big) - (1/2)\Big[e^\theta + e^{\theta'}\Big] \Big\}
  \leq \exp\Big(\max\{\theta,\theta'\}\Big) \cdot\frac{(\theta-\theta')^2}{4},
\]
where the inequality holds by applying $1-x \leq e^{-x} \leq 1-x+x^2/2$ for all $x>0$.
Thus, again by Corollary~3 of \citet{zhang02covering},
\*[
  \genent{\hellsym}{2}(\functionset_\weightspace, \eps, n)
  \lesssim \frac{\xbound^2 \wbound^2 e^{\xbound\wbound}}{\eps^2} \log(n).
\]
Using $\datareldistn(\data) = 2\cdot[\pi(1+\data^2)]^{-1}$ for $\data\in\Nats$, we obtain
\*[
  \datasize 
  = \sum_{\data=0}^\infty \datareldistn(\data) 
  = 2/\pi + \sum_{\data=1}^\infty \datareldistn(\data)
  \leq 2/\pi + \int_0^\infty \datareldistn(\data) \dee \data
  = 1 + 2/\pi
\]
 and 
\*[
  \densbound
  = \sup_{\lambda > 0} \sup_{\data\in\Nats}
  \frac{\lambda^\data e^{-\lambda} \pi (1+\data^2)}{2\data!}
  \leq \pi/2.
\]
So, by \cref{fact:risk-upper-rates},
\*[
  \minimaxrisk{\functionset_{\weightspace}}{n}
  \lesssim
  n^{-1/2} \cdot
  \xbound\wbound e^{\xbound\wbound/2} \cdot (\log n).
\]

Since the inequalities on $e^{-x}$ are nearly tight, it can similarly be shown that for sufficiently close $\theta$ and $\theta'$, for large enough $c$
\*[
  \distsym_{\hellsym}^2(p_\theta,p_{\theta'})
  \geq c \cdot \exp\Big(\max\{\theta,\theta'\}\Big) \cdot(\theta-\theta')^2.
\]
The lower bound thus follows from the same arguments as the Gaussian case.

\emph{Gamma scale with negative inverse link.}
To avoid confusion of which convention is being used, let
\*[
  \gammadist(\alpha,\beta)[\data]
  = \frac{1}{\Gamma(\alpha)\beta^\alpha} \data^{\alpha-1} e^{-\data/\beta}.
\]
Let $p_\theta$ denote the density of $\gammadist(\alpha, -1/(\alpha\theta))$.
Again by Eq.~(15) of \citet{nielsen20cumulant} and Table~1 of \citet{rigollet12glm}, for any $\theta,\theta' \in (-\infty, 0)$
\[\label{eqn:gamma-helper}
  \distsym_{\hellsym}^2(p_\theta,p_{\theta'})
  &= 1 - \exp\Big\{\alpha\log(-2/(\theta+\theta')) - (\alpha/2)\Big[\log(-1/\theta) + \log(-1/\theta')\Big] \Big\} \\
  &\leq
  \Big(\sqrt{-\theta} - \sqrt{-\theta'}\Big)^2 \cdot \frac{\max\{\alpha, 1\}}{\max\{-\theta,-\theta'\}},
\]
where the inequality holds since $1-x^a \leq \max\{a,1\}(1-x)$ for $x\in(0,1)$ and $a>0$.

Now, define
\*[
  \functionset_{\mathrm{Prob}} = \Big\{\context \mapsto \tfrac{1}{\xbound\wbound}(\xbound\wbound + \inner{\context}{\weight}) \setdelim \weight \in \weightspace \Big\}.
\] 
By \cref{eqn:gamma-helper}, 
\*[
  \genent{\hellsym}{2}(\functionset_{\weightspace}, \eps, n)
  \leq \genent{\hellsym}{2}\Bigg(\functionset_{\mathrm{Prob}}, \eps \sqrt{\frac{\alpha\gamma}{\xbound\wbound\max\{\alpha,1\}}}, n\Bigg)
\]

Since Lemma~8 of \citet{rakhlin15binary} combined with \cref{fact:regret-lower} implies that $\minimaxrisk{\functionset_{\mathrm{Prob}}}{n} \lesssim \sqrt{(\log n) / n}$, it must hold that
\*[
  \genent{\hellsym}{2}(\functionset_{\mathrm{Prob}}, \eps, n)
  \lesssim \eps^{-2} \cdot (\log n)^2.
\]

Using $\datareldistn = \gammadist(\alpha, 1/(\alpha\gamma))$, we obtain $\datasize = 1$ and
\*[
  \densbound
  = 
  \sup_{\data\in[0,\infty)}
  \sup_{\theta \in [-\xbound\wbound-\gamma,-\gamma]}
  \Big(\frac{-\theta}{\gamma} \Big)^{\alpha}
  \exp\Big\{-\data\alpha[-\theta - \gamma] \Big\}
  \leq \Big( \frac{\xbound\wbound+\gamma}{\gamma}\Big)^\alpha.
\]
Thus, by \cref{fact:risk-upper-rates},
\*[
  \minimaxrisk{\functionset_\weightspace}{n}
  \lesssim
  n^{-1/2} \cdot (\log n) \cdot \sqrt{\frac{\xbound\wbound\max\{\alpha,1\}}{\gamma}\log\Bigg(n \cdot \Big( \frac{\xbound\wbound+\gamma}{\gamma}\Big)^\alpha \Bigg)}.
\]

For a lower bound, a similar argument can be made to show 
\*[
  \genent{\hellsym}{2}(\functionset_{\weightspace}, \eps, n)
  \geq \genent{\hellsym}{2}(\functionset_{\mathrm{Prob}}, \eps \cdot c, n)
\]
for an appropriate $c$, since, for example, $1-x^a > \min\{a,1\}(1-x)$ for $x \in (0,1)$.
Then,
note that for any $a,b\in[0,1]$, it is easy to check that 
\*[
  \distsym^2_{\hellsym}(a,b) 
  = \frac{1}{2} \Big[(\sqrt{a} - \sqrt{b})^2 + (\sqrt{1-a} - \sqrt{1-b})^2 \Big] 
  \geq \frac{1}{2}(a-b)^2.
\]
The argument then proceeds with the same lower bounds as in the Gaussian case for square loss of linear functionals.
\end{proof}

\begin{proof}[Proof of \cref{fact:vc-example}]
We prove this example by showing that $\genent{\hellsym}{2}(\functionset, \eps, n) \lesssim \vcind(\functionset) \cdot \log(1/\eps)$, and then applying \cref{fact:risk-upper-rates}.
Fix $\eps > 0$ and $\context_{1:n} \subseteq \contextspace$. 
By, for example, Theorem~2.6.7 of \citet{vandervaart96weakconvergence}, there exists $\coverset$ that is a $(1,2)$-cover of $\functionset$ at scale $\eps^2$ such that $\abs{\coverset} \lesssim \vcind(\functionset) \cdot \log(1/\eps)$. 

Consider any $\func \in \functionset$, and let $\coverfunc\in\coverset$ be the corresponding covering element. 
Then, since $\distsym_\hellsym^2(a,b) \leq \abs{a-b}$ for all $a,b\in[0,1]$,
\*[
  \frac{1}{n}\smash[b]{\sum_{t=1}^n} \distsym^2_{\hellsym}(\func(\context_t), \coverfunc(\context_t))
  \leq 
  \frac{1}{n}\smash[b]{\sum_{t=1}^n} \abs[0]{\func(\context_t) - \coverfunc(\context_t)}
  \leq \eps^2.
\]
That is, $\coverset$ is a $(\hellsym,2)$-cover for $\functionset$ at scale $\eps$.

\end{proof}

\begin{proof}[Proof of \cref{fact:vcind-vcdim}]
By \citet[][Lemma~6.11]{egloff05statistical}, $\Hh_{\vcclass} = \{\context \mapsto \vcsmooth \cdot \ind{\context\in\vcfunc} \setdelim \vcfunc \in \vcclass, \vcsmooth \in [0,1] \}$ satisfies $\vcind(\Hh_\vcclass) \leq 3\cdot \vcdim(\vcclass)$. Then, since 
\*[
  \functionset
  &= \{\context \mapsto \vcsmooth_0 \cdot \ind{\context\in\vcfunc^\compsym} + \vcsmooth_1 \cdot \ind{\context\in\vcfunc}  \setdelim \vcfunc \in \vcclass, \vcsmooth_0, \vcsmooth_1 \in [0,1] \} \\
  &= \{\context \mapsto \vcsmooth_0 \cdot \ind{\context\in\vcfunc^\compsym} \vee \vcsmooth_1 \cdot \ind{\context\in\vcfunc}  \setdelim \vcfunc \in \vcclass, \vcsmooth_0, \vcsmooth_1 \in [0,1] \} \\
  &\subseteq \Hh_{\vcclass} \vee \Hh_{\vcclass^\compsym},
\]
by \citet[][Lemma~9.9]{kosorok08empirical},
\*[
  \vcind(\functionset)
  \leq \vcind(\Hh_\vcclass) + \vcind(\Hh_{\vcclass^\compsym}) - 1
  \leq 6\cdot \vcdim(\vcclass).
\]
\end{proof}

\begin{proof}[Proof of \cref{fact:multinomial-smooth-example}]
For the upper bound, for any $\func,\func' \subseteq \{\contextspace \to [\multidim]\}$,
\*[
  \distsym_\hellsym^2(\func(\context), \func'(\context))
  \leq \distsym_{\tvsym}(\func(\context), \func'(\context))
  \leq \multidim \norm{\func(\context) - \func'(\context)}_{\infty},
\]
so
\*[
  \genent{\hellsym}{2}(\functionset,\eps,n)
  \leq \genent{\infty}{\infty}(\functionset,(\eps/\multidim)^2,n)
  \lesssim (\eps/\multidim)^{-\frac{2\funcdim}{\numderivs+\holderparam}},
\]
where we appeal to Theorem~2.7.1 of \citet{vandervaart96weakconvergence} for the last step.

For simplicity, we prove the lower bound for $\numderivs=0$ and $\multidim=2$ using standard techniques, noting that the argument proceeds similarly for the higher-order smoothness (see Section~5.1 of \citep{wainwright19book}).
Fix $\eps>0$ and $n$. Let $\context_{1:M}$ be a
uniform discretization of $[0,1]^\funcdim$ at scale
$\eps^{2\funcdim/\holderparam}$, so that $M =
\eps^{-2\funcdim/\holderparam}$ (we assume without loss of generality
that $M$ is an integer). 
Define $\phi(y) = 2^{2\holderparam} y^\holderparam (1-y)^\holderparam \ind{y\in[0,1]}$.
For each $\beta \in \{0,1\}^M$, let
\*[
  \coverfunc_{\beta}(\context) = \sum_{i=1}^M \beta_i \eps^2 \, \phi\Bigg(\frac{\norm{\context - \context_i}_\infty + \eps^{2/\holderparam}/2}{\eps^{2/\holderparam}} \Bigg).
\]
Define a set $\coverset = \{\coverfunc_\beta = z \setdelim \beta \in
\{0,1\}^M\}$, which has $\abs{\coverset} = 2^M$. 
As noted in Example 5.11 of \citet{wainwright19book}, $\phi$ satisfies
$\abs{\phi(a) - \phi(b)} \leq \abs{a - b}^\holderparam$, so that by
the triangle inequality, $\cG\subseteq\cF$.

Let $\delta(\beta,\beta') = \sum_{j=1}^M \ind{\beta_j \neq
\beta'_j}$ denote the Hamming distance. By the Gilbert-Varshamov
lemma (e.g., Lemma 4.14 of \citet{rigollet17highdim}), there exists
$\coverset' \subseteq \coverset$ of size $\log\abs{\coverset'} \geq
M/8$ with $\delta(\beta,\beta') \geq M/4$ for all $\beta,\beta'$
satisfying $\coverfunc_\beta, \coverfunc_{\beta'}\in\coverset'$.

We next we observe that for Bernoulli distributions $u = (0,1)$ and $v = (q, 1-q)$,
\*[
  \distsym_{\hellsym}^2(u,v)
  = 1 - \sqrt{1-q}
  \geq 1 - \sqrt{1 - q + q^2/4}
  = 1 - \sqrt{(1 - q/2)^2}
  = q/2.
\]
It follows that over a sample $\dumcontext_{1:n} \sim \uniformdist(\context_{1:M})$, for every $\coverfunc_\beta,\coverfunc_{\beta'} \in \coverset'$,
\*[
  \EE_{\dumcontext_{1:n}}
  \frac{1}{n}\sum_{t=1}^n \distsym_{\hellsym}^2(\coverfunc_\beta(\dumcontext_t),\coverfunc_{\beta'}(\dumcontext_t))
  &= 
  \frac{1}{n}\sum_{t=1}^n 
  \frac{1}{M}\sum_{j=1}^M \distsym_{\hellsym}^2(\coverfunc_\beta(\context_j),\coverfunc_{\beta'}(\context_j)) \\
  &= \frac{1}{M} \sum_{j=1}^M \distsym_{\hellsym}^2(\eps^2\phi(1/2), 0) \ind{\beta_j \neq \beta'_j} \\
  &\geq 
  \frac{1}{M}(\eps^2/2) \delta(\beta,\beta') \\
  &\geq \eps^2/8.
\]
This argument shows that there exists a distribution on
$\contextspace$ such that $\coverset'$ is a Hellinger packing in
expectation. It follows that there exists a particular choice for
$\dumcontext_{1:n}$ for which $\coverset'$ an empirical Hellinger packing of size $\log\abs{\coverset'} \gtrsim \eps^{-2\funcdim/\holderparam}$.
\end{proof}

\begin{proof}[Proof of \cref{fact:general-smooth-example}]
First, observe that for any $\context_{1:n} \in [0,1]^\funcdim$ and $\func,\coverfunc \in \{[0,1]^\datadim \to \Reals\}^{[0,1]^\funcdim}$,
\*[
  \frac{1}{n} \sum_{t=1}^n \distsym_\hellsym^2(\func(\context_t), \coverfunc(\context_t))
  \leq \sup_{\context\in[0,1]^\funcdim} \norm{\func(\context) - \coverfunc(\context)}_\infty.
\]  
That is, $\genent{\hellsym}{2}(\functionset,\eps,n) \leq \genent{\infty}{\infty}(\functionset, \eps^2)$.
Let
\*[
  \smoothspace
  = \left\{
    \psi: [0,1]^\datadim \to \Reals\ 
    \middle\vert
    \begin{array}{l}
  \ \forall \data,\dumdata\in[0,1]^\datadim, \ell\in\{0,\dots,\nummderivs\}, \\
  \ \sup_{\derivvecc: \norm{\derivvecc}_1\leq\ell}\abs[0]{D^{\derivvecc} \psi(\data)} \leq 1
  \ \text{ and } \ \\
  \sup_{\derivvecc: \norm{\derivvec}_1=\nummderivs}\abs[0]{D^{\derivvecc}\psi(\data) - D^{\derivvecc}\psi(\dumdata)}
  \leq \norm{\data - \dumdata}_\infty^\holderparam
  \end{array}
    \right\}.
\]
Again by Theorem~2.7.1 of \citet{vandervaart96weakconvergence}, $\entsym_{\infty}(\smoothspace,\eps) \lesssim \eps^{-\frac{\datadim}{\nummderivs+\holderparam}}$ (using simplified entropy notation for real-valued functions). Further, in the notation of Theorem~3.3 of \citet{park22entropy}, $\functionset = \Gg^\numderivs_{\smoothspace}$, and thus
\*[
  \genent{\infty}{\infty}(\functionset, \eps)
  \lesssim \eps^{-\Big(\frac{\funcdim}{\numderivs} +  \frac{\datadim}{\nummderivs+\holderparam}\Big)}.
\]
The result then follows from \cref{fact:risk-upper-rates}.
\end{proof}

\begin{proof}[Proof of \cref{fact:mean-example}]
Let $\shortnormaldist_{\normalmean} =
\normaldist(\normalmean,1)$ for any $\normalmean\in[0,1]$, and define
\*[
  \functionset_{\shortnormaldist} = \{\context \mapsto \shortnormaldist_{\func(\context)} \setdelim \func\in\functionset\}.
\]
Note that
$\KL{\shortnormaldist_\normalmean}{\shortnormaldist_{\normalmeandum}}
= (\normalmean-\normalmeandum)^2/2$,
$\distsym_{\hellsym}^2(\shortnormaldist_\normalmean,
\shortnormaldist_{\normalmeandum}) = 1 -
e^{-(\normalmean-\normalmeandum)^2/8}$, and for all $x\in(-1,1)$, $x^2/10 \leq 1-e^{-x^2/8} x^2/8$. 
Thus,
$\genent{\hellsym}{2}(\functionset_{\shortnormaldist}, \eps, n) \asymp
\genent{\distsym_\klsym}{2}(\functionset_{\shortnormaldist}, \eps, n)
\asymp \eps^{-\funcdim}$. The result now follows by using that
\*[
  &\hspace{-1em}
  \inf_{\hat \normalmean} \sup_{\contextdistn} \sup_{\truefunc\in\functionset}
  \EE_{\context_{1:n} \sim \contextdistn}
  \EE_{\data_{1:n} \sim \shortnormaldist_{\truefunc(\context_{1:n})}}
  \EE_{\context\sim\contextdistn}
  \Big(\truefunc(\context) - \predfunc{n}(\context)\Big)^2 \\
  &=
  2
  \inf_{\hat \normalmean} \sup_{\contextdistn} \sup_{\truefunc\in\functionset}
  \EE_{\context_{1:n} \sim \contextdistn}
  \EE_{\data_{1:n} \sim \shortnormaldist_{\truefunc(\context_{1:n})}}
  \EE_{\context\sim\contextdistn}
  \KL{\shortnormaldist_{\truefunc(\context)}}{\shortnormaldist_{\hat \normalmean_n(\context)}} \\
  &= 
  2
  \inf_{\hat \normalmean} \sup_{\contextdistn} \sup_{\truefunc_{\shortnormaldist}\in\functionset_{\shortnormaldist}}
  \EE_{\context_{1:n} \sim \contextdistn}
  \EE_{\data_{1:n} \sim \truefunc_{\shortnormaldist}(\context_{1:n})}
  \EE_{\context\sim\contextdistn}
  \KL{\truefunc_{\shortnormaldist}(\context)}{\shortnormaldist_{\hat \normalmean_n(\context)}} \\
  &\geq
  2
  \inf_{\predfunc{}} \sup_{\contextdistn} \sup_{\truefunc_{\shortnormaldist}\in\functionset_{\shortnormaldist}}
  \EE_{\context_{1:n} \sim \contextdistn}
  \EE_{\data_{1:n} \sim \truefunc_{\shortnormaldist}(\context_{1:n})}
  \EE_{\context\sim\contextdistn}
  \KL{\truefunc_{\shortnormaldist}(\context)}{\predfunc{n}(\context)} \\
  &\gtrsim n^{-\frac{\funcdim}{2+\funcdim}},
\]
where $\hat \normalmean$ ranges over all conditional mean estimators,
$\predfunc{}$ ranges over all conditional density estimators, and the last step follows from \cref{fact:risk-lower-rates} (ignoring logarithmic factors).
\end{proof}

%% file: section-files/mle-proofs.tex
\section{Proofs for the Performance of the MLE}\label{sec:mle-proofs}

\subsection{Proof of the Upper Bound}\label{sec:mle-upper-proofs}

We first define bracketing entropy.
\begin{definition}[Expected Hellinger Bracketing Entropy]
  A class of measurable functions $\coverset \subseteq \{\contextspace\times\dataspace\to\Reals\}$ is said to \emph{$(\hellsym, 2)$-bracket-cover} $\functionset$ with
  respect to a density $\contextdistn$ at scale
  $\eps>0$ if
 for all $\func\in\functionset$, there exists $\coverfunc_1,\coverfunc_2\in\coverset$ such that
\*[
  \sqrt{\EE_{\context \sim
      \contextdistn} \distsym_{\hellsym}^2(\coverfunc_1(\context),
    \coverfunc_2(\context))} \leq \eps
\]
and, for all $\context\in\contextspace$ and $\data\in\dataspace$, $[\coverfunc_1(\context)](\data) \leq [\func(\context)](\data) \leq  [\coverfunc_2(\context)](\data)$.
We denote the cardinality of the smallest $\coverset$ that achieves
this by $\brackcov{\hellsym}{2}(\functionset, \eps, \contextdistn)$, and finally define the entropy to be $\brackent{\hellsym}{2}(\functionset, \eps, \contextdistn) = \log \brackcov{\hellsym}{2}(\functionset, \eps, \contextdistn)$.
\end{definition}

Note that constructing an expected cover typically requires knowledge of
the covariate distribution $\contextdistn$, which we have avoided
assuming throughout this work. However, the maximum likelihood
estimator does not require constructing such a cover; this notion is
only used to analyze the estimator.

\begin{proof}[Proof of \cref{fact:mle}]

Fix $\contextdistn$ and $\truefunc\in\functionset$. Once again applying \cref{fact:yang98,fact:smoothed-hellinger} and taking $\smoothparam = 1/n$ gives
\*[
	\predrisk{\smoothmle{1/n}}{\truefunc, \contextdistn}{n}
	&= \EE_{\context_{1:n} \sim \contextdistn} \EE_{\data_{1:n} \sim \truefunc(\context_{1:n})} \EE_{\context \sim \contextdistn} \distsym_{\klsym}^2(\truefunc(\context), \smoothmle{1/n}_n(\context)) \\
	&\leq 2 \Big[2 + \log(n\densbound\datasize) \Big] \Big[\EE_{(\context_{1:n},\data_{1:n})} \EE_{\context} \distsym_{\hellsym}^2(\truefunc(\context), \mlefunc{n}(\context)) + \frac{2}{n} \Big].
\]
Next, for every $\eps>0$ it holds that
\*[
	\EE_{(\context_{1:n},\data_{1:n})} \EE_{\context} \distsym_{\hellsym}^2(\truefunc(\context), \smoothmle{1/n}_n(\context))
	&\leq \PP_{(\context_{1:n},\data_{1:n})}\Big[\EE_{\context} \distsym_{\hellsym}^2(\truefunc(\context), \mlefunc{n}(\context)) \geq \eps^2 \Big]
	+ \eps^2.
\]

It remains to bound this tail probability.
Suppose for ease of notation that $\contextdistn$ has a density $\contextdensity$ (if this does not hold, the argument can proceed the same with Radon-Nikodym derivatives) and define $\func_{\contextdistn}(\context,\data) = [\func(\context)](\data) \contextdensity(\context)$ for all $\func\in\functionset$, $\context\in\contextspace$, and $\data\in\dataspace$.
Observe that for any $\func, \func' \in \functionset$,
\*[
	\distsym^2_{\hellsym}(\func_{\contextdistn}, \func'_{\contextdistn})
	&= \int_{\contextspace} \int_{\dataspace} \Big(\sqrt{[\func(\context)](\data) \contextdensity(\context)} - \sqrt{[\func'(\context)](\data) \contextdensity(\context)} \Big)^2 \datareldistn(\dee \data) \dee \context \\
	&= \int_{\contextspace} \contextdensity(\context) \int_{\dataspace} \Big(\sqrt{[\func(\context)](\data)} - \sqrt{[\func'(\context)](\data) } \Big)^2 \datareldistn(\dee \data) \dee \context \\
	&= \EE_{\context\sim\contextdistn} \distsym^2_{\hellsym}(\func(\context), \func'(\context)).
\]

Now, for all $\eps>0$, define $\truefunctionset_\eps = \{\func \in \functionset \setdelim \EE_{\context\sim\contextdistn} \distsym^2_{\hellsym}(\func(\context), \truefunc(\context)) \geq \eps^2\}$. Then, applying Theorem~1 of \citet{wong95inequalities} with $\functionset_n \equiv \{\func_{\contextdistn} \setdelim \func\in\functionset\}$ gives
\*[
	\PP_{(\context_{1:n},\data_{1:n})}
	\Bigg[\sup_{\func\in\truefunctionset_{\eps}} \prod_{t=1}^n \frac{[\func(\context_t)](\data_t)}{[\truefunc(\context_t)](\data_t)} \geq 1 \Bigg]
	\leq
	4e^{-\mlec_2 n \eps^2}
\]
whenever
\[\label{eq:mle-requirement}
	\int_{\mlec_0 \eps^2}^{\mlec_1 \eps} \sqrt{\brackent{\hellsym}{2}(\functionset, \rho/{\mlec_3}, \contextdistn)} \, \dee \rho
	\leq \mlec_4 \sqrt{n} \, \eps^2.
\]

Since $\EE_{\context\sim\contextdistn} \distsym^2_{\hellsym}(\truefunc(\context), \mlefunc{n}(\context)) \geq \eps^2$ implies that there exists $\func\in\truefunctionset_\eps$ with $\prod_{t=1}^n [\func(\context_t)](\data_t) \geq  \prod_{t=1}^n [\truefunc(\context_t)](\data_t)$, we have that for all $\eps$ satisfying \cref{eq:mle-requirement},
\[\label{eq:mle-risk-bound}
	\predrisk{\smoothmle{1/n}}{\truefunc, \contextdistn}{n}
	&\leq 2 \Big[2 + \log(n\densbound\datasize) \Big] \Big[4e^{-\mlec_2 n \eps^2} + \eps^2 + 2/n \Big].
\]

To finish the proof, we consider three cases based on the growth of $\brackent{\hellsym}{2}(\functionset, \eps, \contextdistn)$.

First, if $\brackent{\hellsym}{2}(\functionset, \eps, \contextdistn) \lesssim \funcdim \log(1/\eps)$, then \cref{eq:mle-requirement} is satisfied by $\eps \gtrsim \sqrt{(\funcdim \log n)/n}$, and thus this can be substituted into \cref{eq:mle-risk-bound} to obtain
\*[
	\predrisk{\smoothmle{1/n}}{\truefunc, \contextdistn}{n}
	&\lesssim \, \frac{(\funcdim\log n) \log(n\densbound\datasize)}{n}.
\]

Second, if $\funcdim \leq 2$ and $\brackent{\hellsym}{2}(\functionset, \eps, \contextdistn) \lesssim \eps^{-\funcdim}$, then \cref{eq:mle-requirement} is satisfied by $\eps \gtrsim n^{-1/(2+\funcdim)}$. Thus, $\eps \asymp n^{-1/(2+\funcdim)}$ can be substituted into \cref{eq:mle-risk-bound} to obtain
\*[
	\predrisk{\smoothmle{1/n}}{\truefunc, \contextdistn}{n}
	&\lesssim \, n^{-2/(2+\funcdim)} \log(n\densbound\datasize).
\]

Otherwise, if $\funcdim > 2$ and $\brackent{\hellsym}{2}(\functionset, \eps, \contextdistn) \lesssim \eps^{-\funcdim}$, then \cref{eq:mle-requirement} is satisfied by $\eps \gtrsim n^{-1/(2\funcdim)}$. Thus, $\eps \asymp n^{-1/(2\funcdim)}$ can be substituted into \cref{eq:mle-risk-bound} to obtain
\*[
	\predrisk{\smoothmle{1/n}}{\truefunc, \contextdistn}{n}
	&\lesssim \, n^{-1/\funcdim} \log(n\densbound\datasize).
\]

\end{proof}

\input{section-files/mle-lower-quant-proof}

%% file: section-files/mle-lower-quant-proof.tex
\subsection{Proofs for the lower bound}\label{sec:mle-lower-proofs}

\paragraph*{Preliminaries}
Before proving \pref{fact:mle-lower-conditional}, we first formally
describe the construction of the conditional density class $\convexmleclass$ used in the
theorem and derive some basic properties of this class.

For $\context \in [-1/2, 1/2]$, define
\*[
	\basefunc(\context) = 
	\begin{cases}
	-2\context - 1
	& \context \in [-1/2, -3/8] \\
	-1/4
	& \context \in [-3/8,-1/8] \\
	2\context
	& \context \in [-1/8, 1/8] \\
	1 / 4
	& \context \in [1/8, 3/8] \\
	-2 \context + 1
	& \context \in [3/8, 1/2].
	\end{cases}
\]
Next, for any $\mlewidth > 0$ and $\context \in [-\mlewidth, \mlewidth]$, define
\*[
	\bumpfunc_{\mlewidth}(\context) =
	\begin{cases}
	1 + \context / \mlewidth
	& \context \in [-\mlewidth, -\mlewidth/2] \\
	-\context/\mlewidth
	& \context \in [-\mlewidth/2, \mlewidth/2] \\
	\context / \mlewidth-1
	& \context \in [\mlewidth/2, \mlewidth].
	\end{cases}
\]

\begin{figure}[H]
\centering \makebox[\textwidth]{
\begin{tikzpicture}[scale=0.4]
\footnotesize
  \draw [<-] (0,12) -- (0,0);
  \draw [->] (0,0) -- (14,0);
  \draw[shift={(0,3.5)}] (4pt,0pt) -- (-4pt,0pt) node[left] {$-1/4$};
  \draw[shift={(0,6)}] (4pt,0pt) -- (-4pt,0pt) node[left] {$0$};
  \draw[shift={(0,8.5)}] (4pt,0pt) -- (-4pt,0pt) node[left] {$1/4$};
  \draw[shift={(1,0)}] (0pt,4pt) -- (0pt,-4pt) node[below] {$-\frac{1}{2}$}; 
  \draw[shift={(2.5,0)}] (0pt,4pt) -- (0pt,-4pt) node[below] {$-\frac{3}{8}$};
  \draw[shift={(5.5,0)}] (0pt,4pt) -- (0pt,-4pt) node[below] {$-\frac{1}{8}$};
  \draw[shift={(7,0)}] (0pt,4pt) -- (0pt,-4pt) node[below] {$0$};
  \draw[shift={(8.5,0)}] (0pt,4pt) -- (0pt,-4pt) node[below] {$\frac{1}{8}$};
  \draw[shift={(11.5,0)}] (0pt,4pt) -- (0pt,-4pt) node[below] {$\frac{3}{8}$};
  \draw[shift={(13,0)}] (0pt,4pt) -- (0pt,-4pt) node[below] {$\frac{1}{2}$}; 
  \draw[color=black,very thick] (0,6) -- (1,6);
  \draw[color=black,very thick] (1,6) -- (2.5,3.5);
  \draw[color=black,very thick] (2.5,3.5) -- (5.5,3.5);
  \draw[color=black,very thick] (5.5,3.5) -- (8.5,8.5);
  \draw[color=black,very thick] (8.5,8.5) -- (11.5,8.5);
  \draw[color=black,very thick] (11.5,8.5) -- (13,6);
  \draw[color=black,very thick] (13,6) -- (14,6);
  \node[below] at (7,-1.5) {$\context$};
  \node[left] at (-1,12) {$\basefunc(\context)$};
\end{tikzpicture}
\begin{tikzpicture}[scale=0.4]
\footnotesize
  \draw [<-] (0,12) -- (0,0);
  \draw [->] (0,0) -- (14,0);
  \draw[shift={(0,2)}] (4pt,0pt) -- (-4pt,0pt) node[left] {$-1/2$};
  \draw[shift={(0,6)}] (4pt,0pt) -- (-4pt,0pt) node[left] {$0$};
  \draw[shift={(0,10)}] (4pt,0pt) -- (-4pt,0pt) node[left] {$1/2$};
  \draw[shift={(3,0)}] (0pt,4pt) -- (0pt,-4pt) node[below] {$-\mlewidth$}; 
  \draw[shift={(5,0)}] (0pt,4pt) -- (0pt,-4pt) node[below] {$-\mlewidth/2$};
  \draw[shift={(7,0)}] (0pt,4pt) -- (0pt,-4pt) node[below] {$0$};
  \draw[shift={(9,0)}] (0pt,4pt) -- (0pt,-4pt) node[below] {$\mlewidth/2$};
  \draw[shift={(11,0)}] (0pt,4pt) -- (0pt,-4pt) node[below] {$\mlewidth$}; 
  \draw[color=black,very thick] (0,6) -- (3,6);
  \draw[color=black,very thick] (3,6) -- (5,10);
  \draw[color=black,very thick] (5,10) -- (9,2);
  \draw[color=black,very thick] (9,2) -- (11,6);
  \draw[color=black,very thick] (11,6) -- (14,6);
  \node[below] at (7,-1.2) {$\context$};
  \node[left] at (-1,12) {$\bumpfunc_{\mlewidth}(\context)$};
\end{tikzpicture}}
\caption{Visualizations of $\basefunc(\context)$ and $\bumpfunc_{\mlewidth}(\context)$.}
\label{fig:mle-funcs}
\end{figure}

Finally, for any $\mlewidth>0$ let
\*[
  \sepset(\mlewidth) =
   \left\{ (\sepcontext_1,\dots,\sepcontext_\sepsize)\ \middle\vert \begin{array}{l}
    N \in \Nats, \quad
	\forall i,j\in[\sepsize] \quad
	(\sepcontext_i - \mlewidth, \sepcontext_i+\mlewidth) \subseteq [-3/8, -1/8] \cup [1/8, 3/8] \\
    \qquad\qquad\qquad\quad \ \text{ and }
	(\sepcontext_i - \mlewidth, \sepcontext_i+\mlewidth) \cap (\sepcontext_j - \mlewidth, \sepcontext_j+\mlewidth) = \emptyset
  \end{array}\right\}.
\]

Given $\mleheight\in(0,1)$, $\holderparam \in (0,1]$,
$\sepcontext_{1:\sepsize} \in \sepset(\mlewidth)$, and
$\sepsign_{1:\sepsize}\in\{\pm1\}^\sepsize$ define a collection of functions
\[
  \label{eq:g_function}
	\adjfunc_{\mleheight, \sepcontext_{1:\sepsize}, \sepsign_{1:\sepsize}}(\context)
	= \frac{\mleheight}{2} \basefunc(\context)
	+ \mleheight^2 \sum_{\sepind=1}^\sepsize \sepsign_\sepind
	\bumpfunc_{\mlewidth}(\sepcontext_\sepind - \context),
\]
where $\mlewidth^\holderparam = 2^{1-\holderparam} \mleheight^2$.
Intuitively, each such function consists of a large offset term of
size $\mleheight$, which may be either extended or counteracted
(depending on the sign of $\sepsign$) by a smaller offset of size
$\mleheight^2$. Note that since
$\sepcontext_{1:\sepsize}$ belong to $A(\eta)$---and hence satisfy a
separation condition---only a single term of the sum in
\pref{eq:g_function} will be non-zero for any $\context$.

Equipped with the construction above, for each $\holderparam \in
(0,1]$, we define a conditional density class
\*[
	\mleclass
	= \Big\{\context \mapsto 1/2 + \adjfunc_{\mleheight, \sepcontext_{1:\sepsize}, \sepsign_{1:\sepsize}}(\context) 
	\setdelim 
	\mleheight \in [0,1/4],
	\mlewidth^\holderparam = 2^{1-\holderparam} \mleheight^2,
	\sepcontext_{1:\sepsize} \in \sepset(\mlewidth),
	\sepsign_{1:\sepsize} \in \{\pm1\}^\sepsize \Big\},
\]
and take $\convexmleclass$ be the set of all finite, convex combinations of elements of $\mleclass$.

The following technical lemma establishes a number of basic properties
of this class.

\begin{lemma}\label{fact:mle-lower-properties}
For any function $\convexfullfunc\in\convexmleclass$ corresponding to
a convex combination of elements in $\mleclass$ with weights $\convexprob_{1:\convexsize}$ and parameter $\convexmleheight = \sum_{\convexind=1}^\convexsize \convexprob_\convexind \mleheight_\convexind\in[0,1/4]$,

\begin{itemize}
\item[a)] For all $\context \in [-1/2,1/2],
  \abs[0]{\convexfullfunc(\context) - 1/2} \leq \convexmleheight
  \abs{\basefunc(\context)} \leq 1/16$.

\item[b)] $\EE_{\context\sim\uniformdist[-1/2,1/2]} \brk*{\distsym_{\hellsym}^2(1/2, \convexfullfunc(\context))} \geq \convexmleheight^{2} / 192$.

\item[c)] For all $\context,\dumcontext \in [-1/2,1/2], \abs[0]{\convexfullfunc(\context) - \convexfullfunc(\dumcontext)} \leq \abs{\context-\dumcontext}^\holderparam$.
\end{itemize}
\end{lemma}

\begin{proof}[Proof of \cref{fact:mle-lower-properties}]
This proof is very similar to the proof of Proposition 2 in \citet{birge93mce}, but we repeat the arguments here for completeness and to highlight the differences (namely, that the setting is classification with covariates).

\begin{itemize}
\item[a)] 
Let $\mleheight\in[0,1/4]$,
$\mlewidth^\holderparam=2^{1-\holderparam}\mleheight^2$,
$\sepcontext_{1:\sepsize} \in \sepset(\mlewidth)$, and
$\sepsign_{1:\sepsize} \in \{\pm1\}^\sepsize$ be arbitrary. Since $\sepcontext_{1:\sepsize}\in\sepset(\mlewidth)$, for each $\bumpfunc_\mlewidth(\sepcontext_\sepind - \context) \neq 0$ for only a single $\sepind\in[\sepsize]$. Further, if $\sum_{\sepind=1}^\sepsize \sepsign_\sepind \bumpfunc_{\mlewidth}(\sepcontext_\sepind - \context_t) \neq 0$ then $\abs{\basefunc(\context)} = 1/4$. Thus, either $\abs{\adjfunc_{\mleheight, \sepcontext_{1:\sepsize}, \sepsign_{1:\sepsize}}(\context)} = \mleheight\abs{\basefunc(\context)}/2$ or 
\*[
	\abs{\adjfunc_{\mleheight, \sepcontext_{1:\sepsize}, \sepsign_{1:\sepsize}}(\context)} 
	&\leq \mleheight \abs{\basefunc(\context)}/2 + \mleheight^2 \abs{\bumpfunc_{\mlewidth}(\sepcontext_\sepind - \context)} \\
	&\leq \mleheight \abs{\basefunc(\context)}/2 + \mleheight^2/2 \\
	&\leq \mleheight \abs{\basefunc(\context)}/2 + \mleheight \abs{\basefunc(\context)} /2 \\
	&= \mleheight \abs{\basefunc(\context)},
\]
where the last inequality follows since $\mleheight \leq 1/4$. 
The final result holds by applying Jensen's inequality to the convex combination.

\item[b)] Since $\basefunc$ is
  constant whenever $\bumpfunc_{\mlewidth}(\sepcontext_\sepind -
  \context) \neq 0$ (this follows from the definition of $\sepset(\mlewidth)$), and since $\bumpfunc_\mlewidth(\sepcontext_\sepind -
  \context)$ is odd on $(\sepcontext_\sepind - \mlewidth,
  \sepcontext_\sepind + \mlewidth)$ and zero outside this domain, we have that for all $\sepind\in[\sepsize]$
\[\label{eq:mle-orthogonal}
	\int \basefunc(\context) \bumpfunc_{\mlewidth}(\sepcontext_\sepind - \context) \dee \context = 0.
      \]

Now, let $\sepcontext\convexup_{1:\sepsize_\convexind}$ and
$\sepsign\convexup_{1:\sepsize_\convexind}$ denote the remaining parameters for
the functions $\adjfunc_{\mleheight_\convexind, \sepcontext\convexup_{1:\sepsize_\convexind}, \sepsign\convexup_{1:\sepsize_\convexind}}$ that constitute $\convexfullfunc$ for each $\convexind\in[\convexsize]$.
By \cref{eq:mle-orthogonal}, 
\[\label{eq:mle-int-bound}
	&\hspace{-1em}\int \Big(\sum_{\convexind=1}^\convexsize \convexprob_\convexind \adjfunc_{\mleheight_\convexind, \sepcontext\convexup_{1:\sepsize_\convexind}, \sepsign\convexup_{1:\sepsize_\convexind}}(\context)\Big)^2 \dee \context \\
	&= \int \Bigg[ \Big( \frac{1}{2}\sum_{\convexind=1}^\convexsize \convexprob_\convexind \mleheight_\convexind \basefunc(\context) \Big)^2
	+ \Big(\sum_{\convexind=1}^\convexsize \convexprob_\convexind \mleheight_\convexind \basefunc(\context) \Big) 
	\Big(\sum_{\convexind=1}^\convexsize \convexprob_\convexind \mleheight_\convexind^2 \sum_{\sepind=1}^{\sepsize_\convexind} \sepsign\convexup_\sepind \bumpfunc_{\mlewidth_\convexind}(\sepcontext\convexup_\sepind - \context)\Big) \\ 
	&\qquad\qquad + \Big(\sum_{\convexind=1}^\convexsize \convexprob_\convexind \mleheight_\convexind^2 \sum_{\sepind=1}^{\sepsize_\convexind} \sepsign\convexup_\sepind \bumpfunc_{\mlewidth_\convexind}(\sepcontext\convexup_\sepind - \context)\Big)^2  \Bigg] \dee \context \\
	&= \int  \Big( \frac{1}{2}\sum_{\convexind=1}^\convexsize \convexprob_\convexind \mleheight_\convexind \basefunc(\context) \Big)^2 \dee \context
	+ \sum_{\convexind=1}^\convexsize \convexprob_\convexind \mleheight_\convexind 
	\sum_{\convexinddum=1}^\convexsize \convexprob_{\convexinddum} \mleheight_{\convexinddum}^2 \sum_{\sepind=1}^{\sepsize_{\convexinddum}} \sepsign\convexupdum_\sepind \int \basefunc(\context) \bumpfunc_{\mlewidth_{\convexinddum}}(\sepcontext\convexupdum_\sepind - \context) \dee \context \\ 
	&\qquad\qquad + \int \Big(\sum_{\convexind=1}^\convexsize \convexprob_\convexind \mleheight_\convexind^2 \sum_{\sepind=1}^{\sepsize_\convexind} \sepsign\convexup_\sepind \bumpfunc_{\mlewidth_\convexind}(\sepcontext\convexup_\sepind - \context)\Big)^2 \dee \context \\
	&\geq \int \Big( \frac{1}{2}\sum_{\convexind=1}^\convexsize \convexprob_\convexind \mleheight_\convexind \basefunc(\context) \Big)^2 \dee \context \\
	&= (\convexmleheight^{2}/4) \int [\basefunc(\context)]^2 \dee \context \\
	&= \convexmleheight^{2}/96.
\]
To conclude, note that for any $\beta \in (-1/2, 1/2)$, it is easy to check that
\*[
	\distsym_{\hellsym}^2(1/2, 1/2 + \beta)
	\geq \beta^2/2.
\]
The result thus follows from \cref{eq:mle-int-bound} and part a).

\item[c)] 
By the same reasoning as the proof of Proposition 2.iv) in \citet{birge93mce}, it suffices to show that for all $\delta \in (0,1/4)$ and $\eps \in (0,\mlewidth)$,
\*[
	\mleheight \delta + \mleheight^2 \eps / \mlewidth \leq 2^{\holderparam-1}(\eps^\holderparam+\delta^\holderparam).
\]
Since $\holderparam<1$, $\mleheight \delta + \mleheight^{\bumpexp}\eps
/ \mlewidth \leq \mleheight \delta^\holderparam + \mleheight^2
(\eps/\mlewidth)^\holderparam = \mleheight \delta^\holderparam +
2^{\holderparam-1} \eps^\holderparam$, and we have $\mleheight \leq 1/2 \leq 2^{\holderparam-1}$. Finally, we again apply Jensen's inequality to the convex combination to get the desired result.
\end{itemize}
\end{proof}

\begin{proof}[Proof of \cref{fact:mle-lower-conditional}]
First, note that parts a) and c) of \cref{fact:mle-lower-properties} verify that $\convexmleclass$ contains only $\holderparam$-H\"{o}lder continuous maps from $[-1/2,1/2]$ to $[7/16, 9/16]$, and it trivially includes the constant $1/2$ by taking $\mleheight=0$.

The remainder of the proof can be decomposed into two sections. First,
we show that all elements of $\convexmleclass$ have likelihood upper
bounded by a certain function of $\mleheight$ and $n$ with high
probability. Second, we show that there exists a carefully chosen
element of $\mleclass$ that has comparatively large likelihood, as
well as large corresponding $\convexmleheight$; since the likelihood
for this element is large, it is selected by the MLE. Since part b) of \cref{fact:mle-lower-properties} guarantees that a large $\convexmleheight$ implies a large risk, this suffices to obtain the desired result.

For any function $f:\contextspace \to [0,1]$, we log-likelihood is
given by
\*[
	\loglik_n(f)
	= \frac{1}{n} \sum_{t=1}^n \data_t \log(f(\context_t)) + (1-\data_t)\log(1-f(\context_t)).
\]
\paragraph*{Likelihood upper bound}
We first prove the aforementioned uniform upper bound. To begin, we approximate the log-likelihood using a first-order Taylor expansion. 
In particular,
for any $h\geq{}-1/2$, 
\*[
	\log(1/2+h) 
	&\leq -\log2 + 2h.
\]
Thus, for any $\convexfullfunc \in \convexmleclass$ with corresponding parameters $\convexprob_{1:\convexsize}$, $\mleheight_{1:\convexsize}$, $\sepsize_{1:\convexsize}$, $(\sepcontext\convexup_{1:\sepsize_{\convexind}})_{\convexind\in[\convexsize]}$, and $(\sepsign\convexup_{1:\sepsize_{\convexind}})_{\convexind\in[\convexsize]}$,
\*[
	n[\loglik_n(\convexfullfunc) - \loglik_n(1/2)]
	&\leq\sum_{t=1}^n \Big[2\data_t 
	[\convexfullfunc(\context_t)-1/2]
	- 2(1-\data_t)
	[\convexfullfunc(\context_t)-1/2]
	\Big] \\
	&\leq \convexmleheight \sum_{t=1}^n (2\data_t-1)\basefunc(\context_t)
	+ 2\sum_{\convexind=1}^\convexsize \convexprob_\convexind \mleheight_\convexind^2\sum_{t=1}^n \sum_{\sepind=1}^{\sepsize_\convexind} \abs{\bumpfunc_{\mlewidth_\convexind}(\sepcontext\convexup_\sepind - \context_t)},
\]
where the last line follows by considering the largest possible contribution for $\sepsign\convexup_{1:\sepsize_\convexind}$. 

We upper bound the second term using the following lemma, which
follows from a claim proved in the proof of Theorem~3 of \citet{birge93mce}. We reproduce the proof at the end of this section for clarity and completeness.
\begin{lemma}\label{fact:mle-example-concentration}
For all $\delta>0$, there exists a constant $C'>0$ such that for all
$n$, 
\*[
	\PP_{\context_{1:n}}\Bigg[\sup_{\lambda\in(0,1/4], \sepcontext_{1:\sepsize}\in\sepset(\mlewidth)} \mleheight \sum_{t=1}^n \sum_{\sepind=1}^\sepsize \abs{\bumpfunc_{\mlewidth}(\sepcontext_\sepind - \context_t)} \leq C' n (\log(n)/n)^{\holderparam/2} \Bigg]
	\geq 1 - \delta.
\]
\end{lemma}

By \cref{fact:mle-example-concentration}, with arbitrarily large probability there exists
$C'>0$ such that,
\[\label{eq:mle-lik-upper}
	n[\loglik_n(\convexfullfunc) - \loglik_n(1/2)]
	\leq \convexmleheight \sum_{t=1}^n (2\data_t-1)\basefunc(\context_t)
	+ 2C' n \convexmleheight \bigg(\frac{\log n}{n} \bigg)^{\holderparam/2}.
\]
\paragraph*{Likelihood lower bound}
We now exhibit an element of $\Gamma_{\gamma}$ with large likelihood. First, we use a second-order Taylor expansion to obtain that, for any $\abs{h} \leq 1/16$,
\*[
	\log(1/2+h) 
	&\geq -\log2 + 2h - 2.2h^2.
\]
In particular, this follows by noting there is some $b<\infty$ such
that $\log(1/2+h) \geq -\log2 + 2h - b h^2$ for all $h \in
[-1/16,1/16]$, and that the inequality becomes tight at $h=-1/16$
first. We arrive at $2.2$ by rounding the solution $b$ to $\log(1/2-1/16) = -\log2 + 2(-1/16) - b (1/16)^2$ up to one decimal place (such precision suffices for our argument).

Thus, for any $\fullfunc \in \mleclass$,
\*[
	&n[\loglik_n(\fullfunc) - \loglik_n(1/2)] \\
	&\geq\sum_{t=1}^n \Big[\data_t \Big(2\adjfunc_{\mleheight, \sepcontext_{1:\sepsize}, \sepsign_{1:\sepsize}}(\context_t) - 2.2\adjfunc^2_{\mleheight, \sepcontext_{1:\sepsize}, \sepsign_{1:\sepsize}}(\context_t)\Big)\\ 
	&\qquad\quad+ (1-\data_t)\Big(-2\adjfunc_{\mleheight, \sepcontext_{1:\sepsize}, \sepsign_{1:\sepsize}}(\context_t) - 2.2\adjfunc^2_{\mleheight, \sepcontext_{1:\sepsize}, \sepsign_{1:\sepsize}}(\context_t) \Big)  \Big] \\
	&= \sum_{t=1}^n
	\Big[2(2\data_t-1)\adjfunc_{\mleheight, \sepcontext_{1:\sepsize}, \sepsign_{1:\sepsize}}(\context_t) - 2.2\adjfunc^2_{\mleheight, \sepcontext_{1:\sepsize}, \sepsign_{1:\sepsize}}(\context_t) \Big] \\
	&\geq \mleheight \sum_{t=1}^n (2\data_t-1)\basefunc(\context_t)
	+ 2\mleheight^2\sum_{t=1}^n \sum_{\sepind=1}^\sepsize (2\data_t-1)\sepsign_\sepind \bumpfunc_{\mlewidth}(\sepcontext_\sepind - \context_t)
	- 2.2\mleheight^2 \sum_{t=1}^n (\basefunc(\context_t))^2,
\]
where the last line again follows from part a) of \cref{fact:mle-lower-properties}.%

By Lemma 2 of \citet{birge93mce}, for arbitrarily large probability, if $n$ is large enough then there exists a subset
$\mlesubset\subseteq\context_{1:n}$ of size at least $n/8$ such that $\mlesubset \in \sepset(0.2/(n+1))$ and $\abs{\context_t - \context_s} \geq 0.2/(n+1)$ for all $s\neq t$. 
Let $\specfunc \in \mleclass$ denote the element with parameters
$\hat\mlewidth = 0.2/(n+1)$, $\hat\sepcontext_{1:\sepsize} =
\{\context + \hat\mlewidth/2 \setdelim \context \in \mlesubset\}$, and
$\hat\sepsign_\sepind = 2\data_\sepind - 1$ for each $\sepind \in
[\sepsize]$. By the separation property of $\mlesubset$, this gives
\*[
	\sum_{t=1}^n \sum_{\sepind=1}^\sepsize (2\data_t-1)\hat\sepsign_\sepind \bumpfunc_{\hat\mlewidth}(\hat\sepcontext_\sepind - \context_t)
	= \sepsize/2
	\geq n/16.
\]
That is, with arbitrarily large probability, for sufficiently large $n$ 
\[\label{eq:mle-lik-lower}
	n[\loglik_n(\hat\fullfunc_n) - \loglik_n(1/2)]
	\geq \hat\mleheight \sum_{t=1}^n (2\data_t-1)\basefunc(\context_t)
	+ \hat\mleheight^2n/8
	- 2.2\hat\mleheight^2 \sum_{t=1}^n (\basefunc(\context_t))^2.
\]

\paragraph*{Final result}
Combining \cref{eq:mle-lik-upper,eq:mle-lik-lower}, we have that with arbitrarily large probability there exists $C'>0$ such that for sufficiently large $n$, for
any $\convexfullfunc\in\convexmleclass$
\*[
	&\hspace{-1em}\loglik_n(\hat\fullfunc_n) - \loglik_n(\convexfullfunc) \\
	&\geq (\hat\mleheight-\convexmleheight)\frac{1}{n}\sum_{t=1}^n (2\data_t-1)\basefunc(\context_t)
	+ \hat\mleheight^2/8 
	- 2C' \convexmleheight \bigg(\frac{\log n}{n} \bigg)^{\holderparam/2} 
	- 2.2\hat\mleheight^2\frac{1}{n}\sum_{t=1}^n (\basefunc(\context_t))^2.
\]

Next, by independence, $\EE[(2\data_t-1)\basefunc(\context_t)] = 0$ for each $t$, so by the strong law of large numbers 
\*[
	\frac{1}{n}\sum_{t=1}^n (2\data_t-1)\basefunc(\context_t) \longrightarrow 0 \qquad \as
\]
Furthermore, $\EE (f(\context_t))^2 = \int [f(\context)]^2 \dee
\context = 1/24$, so the strong law of large numbers also implies that
\*[
	\frac{1}{n}\sum_{t=1}^n (\basefunc(\context_t))^2 \longrightarrow 1/24 \qquad \as
\]

Thus, for any $\epsilon > 0$, with arbitrarily large probability there exists $C'>0$ such that for sufficiently large $n$, for
any $\convexfullfunc\in\convexmleclass$
\*[
	\loglik_n(\hat\fullfunc_n) - \loglik_n(\convexfullfunc)
	\geq -\epsilon
	+ [1/8 - 2.2(1/24+\epsilon)]\hat\mleheight^2
	- 2C' \convexmleheight \bigg(\frac{\log n}{n} \bigg)^{\holderparam/2}.
\]

For small enough $\epsilon$, this means there is a constant $k>0$ such that
\*[
	\loglik_n(\hat\fullfunc_n) - \loglik_n(\convexfullfunc)
	\geq k n^{-\holderparam}
	- 2C' \convexmleheight \bigg(\frac{\log n}{n} \bigg)^{\holderparam/2},
\]
where we have used that $2.2/24 < 1/8$ and that $\hat\mlewidth = 0.2/(n+1)$ implies 
\*[
	\hat\mleheight
	= [0.2/(n+1)]^{\holderparam/2} 2^{(\holderparam-1)/2}
	\geq k' n^{-\holderparam/2}
\]
for some constant $k'>0$.
Finally, since $\loglik_n(\hat\fullfunc_n) - \loglik_n(\holdermle{n})
\leq 0$ almost surely, this implies (after rearranging) that
$\convexmleheight^{\scriptscriptstyle\mathrm{MLE}} \geq C (n \log
n)^{-\holderparam/2}$ for some constant $C>0$, where
$\convexmleheight^{\scriptscriptstyle\mathrm{MLE}}$ is the average parameter for the MLE. Combined with
part b) of \cref{fact:mle-lower-properties}, this implies the theorem statement.
\end{proof}

\begin{proof}[Proof of \cref{fact:mle-example-concentration}]
This proof closely follows that of a claim proved within Theorem~3 of \citet{birge93mce}, but we reproduce it for completeness.

Let $\context_{1:n}$ be a realization of covariates, and let
$F_n(\context) = (1/n)\sum_{t=1}^n \ind{\context_t \leq \context}$
be the sample CDF.
For any $\mleheight$, $\sepsize$, and $\sepcontext_{1:\sepsize}\in\sepset(\mlewidth)$,
\[\label{eq:sum-bound}
	\sum_{t=1}^n \sum_{\sepind=1}^\sepsize \abs{\bumpfunc_{\mlewidth}(\sepcontext_\sepind - \context_t)}
	&\leq \frac{\sepsize}{2} \sup_{\sepind\in[\sepsize]} \abs[2]{\{t \setdelim \bumpfunc_{\mlewidth}(\sepcontext_\sepind - \context_t) \neq 0 \}} \\
	&\leq \frac{n}{2\mlewidth} \sup_{\context,\dumcontext: \abs{\context-\dumcontext}\leq \mlewidth} \abs{F_n(\context) - F_n(\dumcontext)},
\]
where we have used the separation property for $\sepcontext_{1:\sepsize}\in{}A(\eta)$ to both bound $\sepsize$ and bound the cardinality of $\{t \setdelim \bumpfunc_{\mlewidth}(\sepcontext_\sepind - \context_t) \neq 0 \}$. Using a result from \citet[p.\ 545]{shorack86empirical} that \citet{birge93mce} invoke in the proof of their Theorem~3, we obtain that for all $\mlewidth, z>0$,
\*[
	\PP_{\context_{1:n}\sim\uniformdist(-1/2,1/2)}\Bigg[\sup_{\context,\dumcontext: \abs{\context-\dumcontext}\leq \mlewidth} \sqrt{n}\abs{F_n(\context) - F_n(\dumcontext)} \geq z \Bigg]
	\leq \frac{C_1}{\mlewidth} \exp\Bigg\{-\frac{C_2 z^2}{\mlewidth + z/\sqrt{n}} \Bigg\}
      \]
      for $C_1, C_2>0$.
Thus, for all $\delta>0$ there exists a constant $C_\delta>0$ such that
\*[
	&\hspace{-1em}\PP_{\context_{1:n}\sim\uniformdist(-1/2,1/2)}\Bigg[\exists \mlewidth \geq (\log n)/n : \sup_{\context,\dumcontext: \abs{\context-\dumcontext}\leq \mlewidth} \sqrt{n}\abs{F_n(\context) - F_n(\dumcontext)} \geq C_\delta (\mlewidth \log n)^{1/2} \Bigg] \\
	&\leq \sum_{j=\log_2(\log_2(n))}^{\lfloor \log_2(n) \rfloor}
	\frac{C_1 n}{2^j} \exp\Bigg\{-\frac{C_2 C_\delta^2 2^j \log n}{2^j + C_\delta 2^{j/2} (\log n)^{1/2}} \Bigg\} \\
	&= \sum_{j=\log_2(\log_2(n))}^{\lfloor \log_2(n) \rfloor}
	\frac{C_1 n}{2^j} \exp\Bigg\{-\frac{C_2 C_\delta^2 \log n}{1 + C_\delta 2^{-j/2} (\log n)^{1/2}} \Bigg\} \\
	&\leq \sum_{j=\log_2(\log_2(n))}^{\lfloor \log_2(n) \rfloor}
	\frac{C_1 n}{2^j} n^{-C_2C_\delta^2/2} \\
	&\leq \delta,
\]
where the second last step uses that $j \geq \log_2(\log_2 (n))$ and the last step chooses $C_\delta$ sufficiently large.

Substituting this into \cref{eq:sum-bound}, for all $\delta>0$ there exists a constant $C'_\delta>0$
such that uniformly for all
$\mleheight \geq ((\log n) / n)^{\holderparam/2}$ (equivalently,
$\mlewidth \geq (\log n) / n$), 
\*[
	\mleheight \sum_{t=1}^n \sum_{\sepind=1}^\sepsize \abs{\bumpfunc_{\mlewidth}(\sepcontext_\sepind - \context_t)}
	\leq C'_\delta \frac{\mleheight}{\mlewidth} (\mlewidth n \log n)^{1/2}
	\leq C'_\delta n ((\log n) / n)^{\holderparam/2}
\]
with probability at least $1-\delta$.
Specifically, we have used that
\*[
	\frac{\mleheight}{\sqrt{\mlewidth}}
	\leq \mleheight^{1-1/\holderparam}
	\leq \bigg(\frac{\log n}{n} \bigg)^{\frac{\holderparam-1}{2}}
	= \bigg(\frac{\log n}{n} \bigg)^{\holderparam/2} \bigg(\frac{n}{\log n} \bigg)^{1/2}.
\]
On the other hand, for $\mleheight \leq ((\log n) /
n)^{\holderparam/2}$, we trivially have the following almost sure bound:
\*[
	\mleheight \sum_{t=1}^n \sum_{\sepind=1}^\sepsize \abs{\bumpfunc_{\mlewidth}(\sepcontext_\sepind - \context_t)}
	\leq n \mleheight / 2
	\leq n ((\log n) / n)^{\holderparam/2}.
\]

\end{proof}

%% file: section-files/regret.tex
\section{Regret bounds}
\label{sec:regret}
In this section we briefly sketch extensions of our results to give
bounds on the \emph{KL regret} defined as
\*[
  \predregret{\predfunc{}}{\truefunc, \contextdistn}{n}
  = \EE_{\context_{1:n} \sim \contextdistn} \EE_{\data_{1:n} \sim \truefunc(\context_{1:n})} \sum_{t=1}^n \log\left(\frac{[\truefunc(\context_t)](\data_t)}{[\predfunc{t-1}(\context_t)](\data_t)} \right),
\]
and the corresponding \emph{minimax regret} given by
\*[
  \minimaxregret{\functionset}{n}
  = \inf_{\predfunc{}} \sup_{\contextdistn} \sup_{\truefunc \in \functionset} \predregret{\predfunc{}}{\truefunc, \contextdistn}{n}.
\]

Under the \iid{} assumption on the data, minimax regret is closely related to the minimax risk.
In particular, we derive upper bounds on regret using the same arguments made for the upper bound on risk, while lower bounds follow from a well-known online-to-batch argument.
While these extensions are quite straightforward, it is useful to be able to refer to regret bounds in order to compare to the existing literature on both risk and regret.%

First, we show that the performance of a sequential analogue of the minimax estimator defined in \cref{sec:minimax-estimator} can be upper bounded as a function of the minimax risk. 
To obtain this new estimator, instead of simply splitting the data sample into two halves---one for building the empirical cover and one for aggregating---we split the $n$ rounds into epochs according to an exponential schedule.

Without loss of generality, we suppose that $\numepochs = \log_2(n+1)$ is an integer. If this is not the case, observe that if $\overline{n}$ is the smallest integer larger than $n$ such that $\log_2(\overline{n}+1)$ is an integer, and $\minimaxregret{\functionset}{\overline{n}} \leq B(\overline{n})$ for some increasing function $B$, since regret is increasing it holds that $\minimaxregret{\functionset}{n} \leq B(2n)$.
We will use $\numepochs$ epochs of lengths $2^{\epochidx-1}$ for each
$\epochidx \in [\numepochs]$. Denote the start of the $\epochidx$th
epoch $\epochrd{\epochidx-1} = 2^{\epochidx-1}$, so that the epoch
ends at round $\epochrd{\epochidx} - 1$. For notational clarity, let
the rounds of the $m$th epoch be denoted by $\epochrds{\epochidx} =
[\epochrd{\epochidx-1}, \epochrd{\epochidx} - 1]$, and let $x_{\tau_m}=\crl*{x_{n_{m-1}},\ldots,x_{n_m-1}}$.

For each $\epochidx \in [\numepochs]$, define $\coverset_\epochidx$ to be a minimal $(\hellsym,2)$-cover of $\functionset$ on
$(\context_t)_{t\in\epochrds{\epochidx}}$.
For each $t \in \epochrds{\epochidx+1}$, define
\*[
  [\predfunc{t,\epochidx+1}(\context)](\data)
  = \sum_{\coverfunc \in \coverset_{\epochidx}} [\smooth{\smoothparam}{\coverfunc}(\context)](\data) 
  \left[ 
  \frac{\prod_{s=\epochrd{\epochidx}}^{t} [\smooth{\smoothparam}{\coverfunc}(\context_s)](\data_s)}{\sum_{\coverfuncdumm \in \coverset_{\epochidx}} \prod_{s=\epochrd{\epochidx}}^{t} [\smooth{\smoothparam}{\coverfuncdumm}(\context_s)](\data_s)}
  \right],
  \qquad \context\in\contextspace,\data\in\dataspace,
\]
which is the same as the intermediate estimator used for the risk, but defined only on the current epoch.
Notice that the first round of each epoch is predicted using $[\predfunc{\epochrd{\epochidx}-1, \epochidx+1}(\context)](\data) = \frac{1}{\abs{\coverset_\epochidx}}\sum_{\coverfunc \in \coverset_\epochidx} [\smooth{\smoothparam}{\coverfunc}(\context)](\data)$, which corresponds to a uniform average over the smoothed empirical cover.
For the very first round when no empirical cover has been formed, take $[\predfunc{0,1}(\context)](\data) = 1/\datasize$.

We have the following upper bound on the performance of this estimator.

\begin{theorem}[Regret upper bound]\label{fact:regret-upper}
If $\genent{\hellsym}{2}(\functionset, \eps, n) \lesssim
\eps^{-\funcdim}$ for $\funcdim > 0$, then for every $n \in \Nats$
\*[
R_n(\cF)
  \lesssim n^{\frac{\funcdim}{\funcdim+2}} \, \log(n \densbound \datasize) .
\]
If $\genent{\hellsym}{2}(\functionset, \eps, n) \lesssim \funcdim
\log(n/\eps)$ for $\funcdim>0$ and $n > C \funcdim^{1/3}$ for a
constant $C$, then
\*[
R_n(\cF)
  \lesssim (\funcdim \log n + 1) \log(n \densbound \datasize).
\]
\end{theorem}
\begin{proof}[Proof of \cref{fact:regret-upper}]

Fix $\contextdistn$ and $\truefunc\in\functionset$.
First, we have
\*[
  &\predregret{\predfunc{1:n, 1:\numepochs}}{\truefunc, \contextdistn}{n} \\
  &= \EE_{\context_{1:n}\sim\contextdistn}
  \EE_{\data_{1:n} \sim \truefunc(\context_{1:n})}
  \sum_{\epochidx=1}^{\numepochs}
  \sum_{t\in\epochrds{\epochidx}}
  \log\Bigg(\frac{[\truefunc(\context_t)](\data_t)}{[\predfunc{t-1,\epochidx}(\context_t)](\data_t)} \Bigg) \\
  &= \EE_{\context_{1}\sim\contextdistn}
  \EE_{\data_{1} \sim \truefunc(\context_{1})}
  \log\Bigg(\frac{[\truefunc(\context_1)](\data_1)}{[\predfunc{0,1}(\context_1)](\data_1)} \Bigg) \\
  &\qquad+ \sum_{\epochidx=2}^{\numepochs}
  \EE_{\context_{\epochrds{\epochidx-1}}\sim\contextdistn}
  \EE_{\context_{\epochrds{\epochidx}}\sim\contextdistn}
  \EE_{\data_{\epochrds{\epochidx}} \sim \truefunc(\context_{\epochrds{\epochidx}})}
  \sum_{t\in\epochrds{\epochidx}}
  \log\Bigg(\frac{[\truefunc(\context_t)](\data_t)}{[\predfunc{t-1,\epochidx}(\context_t)](\data_t)} \Bigg) \\
  &= \EE_{\context_{1}\sim\contextdistn}
  \EE_{\data_{1} \sim \truefunc(\context_{1})}
  \log\Bigg(\frac{[\truefunc(\context_1)](\data_1)}{[\predfunc{0,1}(\context_1)](\data_1)} \Bigg) \\
  &\qquad+ \sum_{\epochidx=2}^{\numepochs}
  \epochrd{\epochidx-1}
  \EE_{\context_{\epochrds{\epochidx-1}}\sim\contextdistn}
  \EE_{\context_{\epochrds{\epochidx}}\sim\contextdistn}
  \EE_{\data_{\epochrds{\epochidx}} \sim \truefunc(\context_{\epochrds{\epochidx}})}
  \frac{1}{\epochrd{\epochidx-1}}
  \sum_{t=\epochrd{\epochidx-1}}^{\epochrd{\epochidx}-1}
  \log\Bigg(\frac{[\truefunc(\context_t)](\data_t)}{[\predfunc{t-1,\epochidx}(\context_t)](\data_t)} \Bigg).
\]

The first term is trivially bounded by $\log(\densbound\datasize)$,
where we have used that $\predfunc{0,1} \equiv 1/\datasize$. Now,
consider the summation over $\epochidx$ in the second
term. Conditional on $\context_{\epochrds{\epochidx-1}}$ used to
define $\coverset_{\epochidx-1}$,
\*[
  \EE_{\context_{\epochrds{\epochidx}}\sim\contextdistn}
  \EE_{\data_{\epochrds{\epochidx}} \sim \truefunc(\context_{\epochrds{\epochidx}})}
  \frac{1}{\epochrd{\epochidx-1}}
  \sum_{t=\epochrd{\epochidx-1}}^{\epochrd{\epochidx}-1}
  \log\Bigg(\frac{[\truefunc(\context_t)](\data_t)}{[\predfunc{t-1,\epochidx}(\context_t)](\data_t)} \Bigg)
\]
is \emph{almost} exactly the \rhs
of \cref{eqn:risk-upper-1}. The only difference is that instead of applying a concentration result on the sample $\dumcontext_{1:n}$ used to build the cover $\coverset$, we must apply the concentration result on the sample $\context_{\epochrds{\epochidx-1}}$ used to build the cover $\coverset_{\epochidx-1}$. Thus, instead of using $n/2$ for this bound, we use $\epochrd{\epochidx-2}$. That is, the same argument as \cref{fact:risk-upper} gives
\*[
  &\predregret{\predfunc{1:n, 1:\numepochs}}{\truefunc, \contextdistn}{n} \\
  &\leq \log(\densbound\datasize)
  + \sum_{\epochidx=2}^{\numepochs} \epochrd{\epochidx-1}
  \Bigg\{
  \frac{\genent{\hellsym}{2}(\functionset, \eps, \epochrd{\epochidx-2})}{\epochrd{\epochidx-1}} \\
  &\qquad+
  2\Big[2 + \log(2\epochrd{\epochidx-2}\densbound\datasize)\Big]
  \Bigg[2\eps^2 + 106 \locradius{\epochrd{\epochidx-2}} + \frac{96\log \epochrd{\epochidx-2} + 576\log\log \epochrd{\epochidx-2} + 4}{\epochrd{\epochidx-2}} 
  \Bigg] 
  \Bigg\} \\
  &= \log(\densbound\datasize)
  + \sum_{\epochidx=1}^{\numepochs-1} 
  \epochrd{\epochidx}
  \Bigg\{
  \frac{\genent{\hellsym}{2}(\functionset, \eps, \epochrd{\epochidx-1})}{2\epochrd{\epochidx-1}} \\
  &\qquad+
  2\Big[2 + \log(2\epochrd{\epochidx-1}\densbound\datasize)\Big]
  \Bigg[2\eps^2 + 106 \locradius{\epochrd{\epochidx-1}} + \frac{96\log \epochrd{\epochidx-1} + 576\log\log \epochrd{\epochidx-1} + 4}{\epochrd{\epochidx-1}} 
  \Bigg]
  \Bigg\},
\]
where the last step uses that $\epochrd{\epochidx-1} = 2 \epochrd{\epochidx-2}$.

Now, suppose that there exist constants $\ratepow \in (0,1]$ and $\paramconstA > 0$ such that for each $\epochidx\in[\numepochs-1]$,
\*[
  &\frac{\genent{\hellsym}{2}(\functionset, \eps, \epochrd{\epochidx-1})}{2\epochrd{\epochidx-1}}
  +
  2\Big[2 + \log(2\epochrd{\epochidx-1}\densbound\datasize)\Big]
  \Bigg[2\eps^2 + 106 \locradius{\epochrd{\epochidx-1}} + \frac{96\log \epochrd{\epochidx-1} + 576\log\log \epochrd{\epochidx-1} + 4}{\epochrd{\epochidx-1}} \Bigg] \\
  &\lesssim \epochrd{\epochidx-1}^{-\ratepow} \log(\paramconstA \epochrd{\epochidx-1}).
\]
Then, straightforward computation gives
\*[
  &\hspace{-2em}\sum_{\epochidx=1}^{\numepochs-1} \epochrd{\epochidx} \epochrd{\epochidx-1}^{-\ratepow} \log(\paramconstA \epochrd{\epochidx-1}) \\
  &= 2\sum_{\epochidx=1}^{\numepochs-1} 2^{(1-\ratepow)(\epochidx-1)} (\epochidx-1) \log(2\paramconstA) \\
  &= 2\log(2\paramconstA) \frac{2^{-\numepochs \ratepow+\ratepow-1}\Big[-\numepochs(2^\ratepow - 2)2^{\ratepow+\numepochs} - 2^{\ratepow+\numepochs+2}+2^{2\ratepow+\numepochs}+2^{\ratepow\numepochs+2} \Big]}{(2^{\ratepow}-2)^2} \\
  &\lesssim \log(\paramconstA) \numepochs 2^{\numepochs(1-\ratepow)} \\
  &\lesssim \log(\paramconstA) n^{1-\ratepow} \log n. 
\]
The three cases in the theorem statement now follow from the rates derived in the proof of \cref{fact:risk-upper-rates}.
\end{proof}

Using a standard online-to-batch argument, we also obtain the following lower bound on minimax regret.
\begin{theorem}[Regret lower bound]\label{fact:regret-lower}
For all $n \in\Nats$,
\*[
  \minimaxregret{\functionset}{n}
  \geq n\cdot\minimaxrisk{\functionset}{n-1}.
\]
\end{theorem}
\begin{proof}[Proof of \cref{fact:regret-lower}]
\*[
  \minimaxregret{\functionset}{n}
  &= 
  \inf_{\predfunc{}}
  \sup_{\contextdistn}
  \sup_{\truefunc \in\functionset}
  \EE_{\context_{1:n} \sim \contextdistn} 
  \EE_{\data_{1:n} \sim \truefunc(\context_{1:n})} 
  \sum_{t=1}^n 
  \log\left(\frac{[\truefunc(\context_t)](\data_t)}{[\predfunc{t-1}(\context_t)](\data_t)}\right) \\
  &=
  n
  \inf_{\predfunc{}}
  \sup_{\contextdistn}
  \sup_{\truefunc \in\functionset}
  \EE_{t \sim [n]}
  \EE_{\context_{1:t} \sim \contextdistn} 
  \EE_{\data_{1:t} \sim \truefunc(\context_{1:t})} 
  \log\left(\frac{[\truefunc(\context_t)](\data_t)}{[\predfunc{t-1}(\context_t)](\data_t)}\right) \\
  &=
  n
  \inf_{\predfunc{}}
  \sup_{\contextdistn}
  \sup_{\truefunc \in\functionset}
  \EE_{\context_{1:n} \sim \contextdistn} 
  \EE_{\data_{1:n} \sim \truefunc(\context_{1:n})} 
  \EE_{\context \sim \contextdistn}
  \EE_{\data \sim \truefunc(\context)}
  \EE_{t \sim [n]}
  \log\left(\frac{[\truefunc(\context)](\data)}{[\predfunc{t-1}(\context)](\data)}\right) \\
  &\geq
  n
  \inf_{\predfunc{}}
  \sup_{\contextdistn}
  \sup_{\truefunc \in\functionset}
  \EE_{\context_{1:n} \sim \contextdistn} 
  \EE_{\data_{1:n} \sim \truefunc(\context_{1:n})} 
  \EE_{\context \sim \contextdistn}
  \EE_{\data \sim \truefunc(\context)}
  \log\left(\frac{[\truefunc(\context)](\data)}{[\avgpredfunc{n}(\context)](\data)}\right),
\]
where $[\avgpredfunc{n}(\context)](\data) = \frac{1}{n}
\sum_{t=0}^{n-1}[\predfunc{t}(\context)](\data)$; the final step
follows from Jensen's inequality. Finally, note that this upper bounds
the minimum risk over all choices for $\predfunc{}$, since an admissible choice is to take $\predfunc{n} = \bar f_n$.
\end{proof}

%% file: section-files/adaptive.tex
\section{Adaptive estimation}
\label{sec:adaptive}

\cref{fact:risk-upper-rates} requires that the model $\functionset$ is well-specified. Using the same aggregation techniques as in \cref{sec:minimax-estimator}, it is straightforward to upgrade our minimax estimator from \cref{sec:minimax-estimator} to an adaptive estimator with oracle risk bounds (similar approaches to adaptivity can be found in, for example, Section~2 of \citet{yang00mixing}).

Let $(\functionset_m)_{m\in\Nats}$ be a countable collection of candidate models and let $\prior$ be a probability distribution on $\Nats$. Rather than splitting the sample into two halves as we did in \cref{sec:minimax-estimator}, now split into thirds:
$(\context'_{1:n}, \data'_{1:n})$,  $(\context''_{1:n}, \data''_{1:n})$, and $(\context_{1:n}, \data_{1:n})$. Use $(\context'_{1:n}, \data'_{1:n})$ to construct the $(\hellsym, 2)$-covers $(\coverset_m)_{m\in\Nats}$ of each element of $(\functionset_m)_{m\in\Nats}$ respectively.
Similarly, let $\riskpredfuncmod{n}{m}$ be the estimator from \cref{sec:minimax-estimator} for the model $\functionset_m$, aggregated over $(\context''_{1:n}, \data''_{1:n})$.
Then, for each $m,n \in \Nats$, $\context\in\contextspace$, and $\data\in\dataspace$, define
\*[
	[\riskpredfuncadapt{n}(\context)](\data)
	= 
	\frac{1}{n+1}\sum_{t=0}^n
	\sum_{m \in \Nats} \prior(m) [\riskpredfuncmod{n}{m}(\context)](\data) 
  	\left[ 
  	\frac{\prod_{s=1}^{t} [\riskpredfuncmod{n}{m}(\context_s)](\data_s)}{\sum_{m' \in \Nats} \prior(m') \prod_{s=1}^{t} [\riskpredfuncmod{n}{m'}(\context_s)](\data_s)}
  	\right].
\]

Clearly, $\riskpredfuncadapt{}$ defines a valid conditional density, and we have the following guarantee.
\begin{theorem}\label{fact:adaptive}
For all $\contextdistn$ and $\truefunc \in \kernelset(\contextspace,\dataspace)$,
\*[
	\predrisk{\riskpredfuncadapt{n}}{\truefunc, \contextdistn}{n}
	\leq
	\inf_{m\in\Nats} \Big\{\predrisk{\riskpredfuncmod{2n/3}{m}}{\truefunc, \contextdistn}{2n/3}
	+ \frac{3}{n} \log\frac{1}{\prior(m)} \Big\}.
\]
\end{theorem}
\begin{remark}
This result can be immediately instantiated to obtain oracle risk bounds. Let $\prior(m) = 6/(m \pi^2)$.
By \cref{fact:risk-upper,fact:risk-lower}, when $\truefunc\in\functionset_m$ then $\predrisk{\riskpredfuncmod{}{m}}{\truefunc, \contextdistn}{n}$ is within logarithmic factors of $\minimaxrisk{\functionset_m}{n}$, and hence so is $\riskpredfuncadapt{}$ without requiring knowledge of the true $m$ in advance.
\end{remark}

\begin{proof}[Proof of \cref{fact:adaptive}]
By \cref{eqn:risk-upper-1,eqn:risk-upper-2}, for all $m\in\Nats$, conditional on $(\context'_{1:n}, \data'_{1:n})$ and $(\context''_{1:n}, \data''_{1:n})$ it holds that
\*[
	&\hspace{-1em}\EE_{\context_{1:n}\sim\contextdistn}
	\EE_{\data_{1:n} \sim \truefunc(\context_{1:n})}
	\EE_{\context \sim \contextdistn} 
	\distsym_{\klsym}^2(\truefunc(\context), \riskpredfuncadapt{n}(\context)) \\
	&\leq 
	\frac{\log(1/\prior(m))}{n+1} 
	+  
	\EE_{\context_{1:n+1}\sim\contextdistn}
	\EE_{\data_{1:n+1} \sim \truefunc(\context_{1:n+1})}
	\frac{1}{n+1}\sum_{t=1}^{n+1}
	\log\Bigg(\frac{[\truefunc(\context_t)](\data_t)}{[\riskpredfuncmod{n}{m}(\context_t)](\data_t)} \Bigg) \\
	&=
	\frac{\log(1/\prior(m))}{n+1} 
	+  
	\EE_{\context\sim\contextdistn}
	\distsym_{\klsym}^2(\truefunc(\context), \riskpredfuncmod{n}{m}(\context)).
\]
The result follows by taking expectation over the remaining training data and recalling that $n$ actually corresponds to ``$n/3$'' by sample splitting.
\end{proof}